\documentclass[11pt]{article}

\usepackage[margin=1in]{geometry}
\usepackage[T1]{fontenc}
\usepackage{lmodern}
\usepackage{amsmath,amssymb,amsthm,mathtools}
\usepackage{mathrsfs}
\usepackage{microtype}
\usepackage{enumitem}
\usepackage{xcolor}
\newif\ifrevisionmarks
\revisionmarkstrue

\usepackage[colorlinks=true,linkcolor=blue,citecolor=blue,urlcolor=blue]{hyperref}
\usepackage[nameinlink,capitalise]{cleveref}
\usepackage{etoolbox}
\newtheorem{theorem}{Theorem}[section]
\newtheorem{proposition}[theorem]{Proposition}
\newtheorem{lemma}[theorem]{Lemma}
\newtheorem{corollary}[theorem]{Corollary}
\theoremstyle{definition}
\newtheorem{definition}[theorem]{Definition}

\theoremstyle{remark}
\newtheorem{remark}[theorem]{Remark}

\newcommand{\T}{\mathbb T}
\newcommand{\Rr}{\mathbb R}
\newcommand{\Pp}{\mathcal P}
\newcommand{\PN}{\mathcal P_N}
\newcommand{\Qq}{\mathcal Q}
\newcommand{\F}{\mathcal F_\beta}
\newcommand{\Wass}{\mathsf W_2}
\newcommand{\Ent}{\operatorname{Ent}}
\newcommand{\Crit}{\operatorname{Crit}}
\newcommand{\dist}{\operatorname{dist}}
\newcommand{\dd}{\,\mathrm d}
\newcommand{\e}{\mathrm e}
\newcommand{\Hh}{\mathsf H}
\newcommand{\Vq}{\mathsf V}
\newcommand{\Gg}{\mathcal G}

\newcommand{\norm}[1]{\left\lVert #1\right\rVert}
\newcommand{\abs}[1]{\left\lvert #1\right\rvert}
\newcommand{\ip}[2]{\left\langle #1,#2\right\rangle}

\numberwithin{equation}{section}

\title{Metastability and Sharp Propagation-of-Chaos Thresholds
for Mean-Field Interacting Diffusions on the Circle}
\author{
Sayan Banerjee%
\thanks{Department of Statistics and Operations Research,
University of North Carolina at Chapel Hill,
318 Hanes Hall, CB\# 3260,
Chapel Hill, NC 27599-3260, USA.
Email: \texttt{sayan@email.unc.edu}.}
}
\date{}

\begin{document}
\maketitle

\begin{abstract}
We study metastability and long-time propagation-of-chaos for mean-field interacting diffusions on the circle whose free-energy landscape has multiple local minima. Working modulo rotations, we identify sharp exponential time scales for persistence in metastable wells, as well as thresholds governing validity and breakdown of the McKean--Vlasov approximation. Assuming finitely many critical rotation orbits and distinct local-minimum energies, we show that, for initial data in the attraction basin of a nonglobal minimum, metastability and propagation-of-chaos hold uniformly over deterministic times up to \(e^{N(D-\varepsilon)}\) and fail at every deterministic sequence growing faster than \(e^{N(D+\varepsilon)}\). Here \(N\) is the number of particles and \(D\) is the depth of the maximal energy cycle, in the sense of Freidlin--Wentzell, having that minimum as its bottom. The same threshold governs running averages of the squared approximation error in a stronger pathwise sense, with corresponding metastability and propagation-of-chaos bounds for suitably aligned time-averaged empirical measures. When the global-minimum orbit is unique, we establish uniform-in-time propagation-of-chaos in its attraction basin without assuming finitely many critical orbits. The proofs develop the Freidlin--Wentzell cycle analysis for the empirical measure process on the rotation quotient. The key technical steps are to pass from the fixed-horizon large deviation estimates of Dawson and G\"artner to the analysis of excursions of unbounded duration between metastable wells, and to develop this analysis into sharp cycle-exit estimates and localization bounds at deterministic times. We give a checkable Fourier criterion for finiteness of the critical orbits and explicit examples with multiple wells, and discuss applications to the noisy Kuramoto model and conditional consequences for noisy transformer interactions. We also investigate how phase transitions in the free-energy landscape, induced by varying the noise strength, affect metastability and propagation-of-chaos. We provide, to our knowledge, the first sharp exponential-time analysis of the Dawson--G\"artner `tunnelling phenomenon' for general smooth interactions, with two-sided probability bounds and matching propagation-of-chaos thresholds.\newline



\end{abstract}

\section{Introduction}

In this article, we investigate the long-time behavior of a mean-field system of interacting diffusions on the circle. Such systems arise in several areas, like the stochastic Kuramoto model for coupled oscillators \cite{BertiniGiacominPakdaman2010,BertiniGiacominPoquet2014}, the noisy Hegselmann–Krause model for opinion dynamics \cite{HegselmannKrause2002,CarrilloEtAl2020}, and noisy transformer models \cite{ShalovaSchlichting2026,BalasubramanianBanerjeeRigollet2025}.  Although large-particle asymptotics on fixed time intervals are well understood for such systems, many questions remain about their long-time behavior and relation to that of the `McKean--Vlasov' equations (see \eqref{eq:MV}) governing their mean-field limit. A central question is to determine the time scales over which the McKean--Vlasov equation accurately approximates the empirical measure of the particle system, in the sense of propagation-of-chaos modulo rotations defined below. We study this question in settings where the empirical measure exhibits \emph{metastability}, reflecting the presence of multiple wells in the nonconvex McKean--Vlasov free-energy landscape. We obtain a detailed description of the time spent in metastable wells and the transitions between them, and establish sharp criteria for the validity and breakdown of propagation-of-chaos.

\subsection{The model, quotient state space and free-energy}
\label{sec:model}

For \(N\geq1\), consider a system of $N$ interacting particles on $\T=\Rr/(2\pi\mathbb Z)$ whose joint angular evolution is prescribed by the following system of real-valued stochastic differential equations (SDE):
\begin{equation}
\label{eq:particle-system}
    \dd X_t^{i,N}
    =
    -\frac1N\sum_{j=1}^N
       W'\bigl(X_t^{i,N}-X_t^{j,N}\bigr)\,\dd t
    +\sqrt{2\beta^{-1}}\,\dd B_t^i,
    \qquad 1\leq i\leq N.
\end{equation}
The potential function $W: \T \rightarrow \mathbb{R}$ (identified with its periodic lift onto $\mathbb{R}$) captures the mean-field type drift interaction between the particles and the inverse temperature $\beta>0$ regulates the strength of the noise, which acts on the particles through independent standard Brownian motions $B_1,\dots,B_N$. Throughout, we assume
\begin{equation}
\label{eq:W-assumption}
    W\in C^4(\T),
    \qquad
    W(-x)=W(x).
\end{equation}
We can also view the system $(X^{1,N},\dots,X^{N,N})$ as a Langevin diffusion on $\T^N$. Define the \(N\)-particle Hamiltonian
\begin{equation}
\label{eq:HN}
    \mathcal H_N(x_1,\ldots,x_N)
    :=
    \frac1{2N}\sum_{i,j=1}^N W(x_i-x_j).
\end{equation}
Evenness of \(W\) gives
\(
    \partial_{x_i}\mathcal H_N
    =
    \frac1N\sum_{j=1}^N W'(x_i-x_j),
\)
and hence \eqref{eq:particle-system} is the reversible gradient
diffusion
\[
   \dd \mathbf X_t^N
    =
    -\nabla\mathcal H_N(\mathbf X_t^N)\,\dd t
    +\sqrt{2\beta^{-1}}\,\dd \mathbf B_t^N.
\]
Here $\mathbf X_t^N=(X_t^{1,N},\ldots,X_t^{N,N})$ and
$\mathbf B_t^N=(B_t^1,\ldots,B_t^N)$.
Write $\lambda(\dd x):=(2\pi)^{-1}\dd x$ for normalized Lebesgue
measure on $\T$. The invariant joint density of this system with respect to
$\lambda^{\otimes N}$ is
proportional to \(\exp(-\beta\mathcal H_N)\). 

The particle joint locations are encoded by the empirical measure
\[
    \mu_t^N:=\frac1N\sum_{i=1}^N\delta_{X_t^{i,N}}.
\]
We view $\mu^N$ as a random variable on the continuous path-space $\mathcal{C}\left([0,\infty) : \Pp(\T)\right)$; here $\Pp(\T)$ denotes the space of probability measures on $\T$ equipped with the Wasserstein-2 distance:
\[
 \Wass(\mu,\nu)
 :=\left(
   \inf_{\gamma\in\Pi(\mu,\nu)}
   \int_{\T\times\T}d_\T(x,y)^2\,\gamma(\dd x,\dd y)
 \right)^{1/2}, \qquad \mu,\nu \in \Pp(\T),
\]
where $\Pi(\mu,\nu)$ is the set of probability measures on
$\T\times\T$ with first marginal $\mu$ and second marginal $\nu$,
and
\(
 d_\T(x,y):=\min_{k\in\mathbb Z}|x-y+2\pi k|
\)
is the geodesic distance on $\T=\Rr/(2\pi\mathbb Z)$. 

Since the interaction depends only on
differences of positions, if all the initial particle locations are rotated by some angle $\theta \in \T$, then the new particle locations shift by the same angle $\theta$ at all subsequent times. This motivates us to work on the quotient space
\begin{equation*}
 \Qq:=\Pp(\T)/\T,
\end{equation*}
in which two probability measures are identified if one is derived by a
rotation of the other. More precisely, for \(\theta\in\T\), define
\[
    R_\theta\mu:=(x\mapsto x+\theta)_\#\mu.
\]
Identify two probability measures \(\mu,\nu\in\Pp(\T)\)
if \(\nu=R_\theta\mu\). Equip $\Qq$ with the metric
\[
 d_{\Qq}([\mu],[\nu])
 :=\inf_{\theta\in\T}\Wass(\mu,R_\theta\nu).
\]
Rotations are isometries for $\Wass$, so this definition does not
depend on the chosen representatives. The infimum is attained
because $\T$ is compact and the rotation action is continuous.
In particular, if $[\mu_n]\to[\mu]$, there exist $\theta_n\in\T$
such that
\[
 \Wass(R_{\theta_n}\mu_n,\mu)
 =d_{\Qq}([\mu_n],[\mu])\longrightarrow0.
\]
The space $\Qq$ is compact, being the continuous image of the
compact space $\Pp(\T)$.

Mean-field limits of systems given by \eqref{eq:particle-system} are classical; see, for example,
\cite[Chapter~I, Sections~1--2]{Sznitman1991}. Suppose that the initial configurations are independent of the
driving Brownian motions and that $\mu_0^N$ converges in probability
to a deterministic $\mu_0\in\Pp(\T)$. Then it is well-known that on compact time intervals $[0,T]$, $T< \infty$, $\mu^N$ converges in probability, as $\mathcal{C}\left([0,T] : \Pp(\T)\right)$ valued random variables, to a deterministic measure-valued process $\mu$ which, for every $t>0$, has a density $\rho_t$ with respect to $\lambda$ satisfying the \emph{McKean--Vlasov equation}:
\begin{equation}
\label{eq:MV}
    \partial_t\rho
    =
    \beta^{-1}\partial_{xx}\rho
    +\partial_x\bigl(\rho\,W'*\mu\bigr)
    =
    \partial_x\left[
       \rho\,\partial_x
       \bigl(\beta^{-1}\log\rho+W*\mu\bigr)
    \right],
\end{equation}
with initial condition $\mu_0$.
The McKean--Vlasov equation has recently garnered renewed interest owing to its connection with optimal transport. Indeed, it is the Wasserstein-2 gradient flow of the \emph{free-energy functional}:
\begin{equation}
\label{eq:free-energy}
    \F(\mu)
    :=
    \beta^{-1}\Ent(\mu\mid\lambda)
    +\frac12\iint_{\T^2}W(x-y)\,\mu(\dd x)\mu(\dd y),
\end{equation}
where \(\F(\mu)=+\infty\) when \(\mu\not\ll\lambda\).
The gradient-flow and path-large-deviation properties for the analogous model on the torus are discussed in \cite{GvalaniSchlichting2020}. Observe that, by translation invariance,
\[
    \F(R_\theta\mu)=\F(\mu),
    \qquad
    S_t(R_\theta\mu)=R_\theta S_t\mu,
\]
where \(S_t\) is the semiflow generated by \eqref{eq:MV}.

The free-energy and the semiflow descend to \(\Qq\). We also use $S_t$ for the induced semiflow on $\Qq$, writing
$S_t[\mu]:=[S_t\mu]$; when a distinction is needed, the quotient
semiflow is denoted by $\overline S_t$.

All subsets of \(\Qq\) appearing below are identified, when necessary,
with their rotation-invariant lifts to \(\Pp(\T)\).  Accordingly, a
hitting time of a set in \(\Qq\) is evaluated by asking whether
\([\mu_t^N]\) belongs to that set.  This convention avoids repeatedly
writing the quotient map.

We write $\Crit(\F)$ for the set of critical points of $\F$
under the unit-mass constraint:
\[
 \Crit(\F):=
 \left\{
 \mu=\rho\lambda\in\Pp(\T):
 \rho\in C^2(\T),\ \rho>0,\quad
 \beta^{-1}\log\rho+W*\mu
 \text{ is constant on }\T
 \right\},
\]
where $(W*\mu)(x):=\int_\T W(x-y)\,\mu(\dd y)$.
These are precisely the stationary measures of
\eqref{eq:MV}, equivalently those satisfying
$S_t\mu=\mu$ for every $t\ge0$.
The corresponding set of critical rotation orbits is
\[
 \Crit(\F)/\T
 :=\bigl\{[\mu]:\mu\in\Crit(\F)\bigr\}\subset\Qq.
\]
Except where explicitly stated otherwise (in particular, in the
global-basin results of Theorems~\ref{thm:poc-global}
and~\ref{thm:poc-running-averages}\textup{(iii)}, the corresponding
global-basin case of Corollary~\ref{cor:poc-averaged-measures},
and the Gibbs comparison estimates of
Lemmas~\ref{lem:uniform-uphill-endpoint}
and~\ref{lem:poc-global-maximal}), our
standing landscape hypothesis is
\begin{equation}
\label{eq:finite-critical-orbits}
    \Crit(\F)/\T
    \quad\text{is finite}.
\end{equation}
Some sufficient conditions for this to hold are given in Proposition~\ref{prop:app-finite}.

\subsection{Related work and comparison}
\label{subsec:related-work}

\paragraph{Propagation-of-chaos over long time horizons.}
The approximation of weakly interacting particle systems by nonlinear
diffusions and partial differential equations of the form \eqref{eq:MV} goes back to McKean~\cite{McKean1966}; see
Sznitman~\cite{Sznitman1991} and
Chaintron and Diez~\cite{ChaintronDiez2022} for systematic accounts of
propagation-of-chaos (PoC). On a fixed time interval, this approximation
is well understood for regular interactions. Extending it to time
horizons that grow with the number of particles requires additional
stability, since the constants in finite-time estimates typically grow
with time. Classical uniform-in-time results for granular media
equations exploit convexity and convergence to equilibrium
\cite{Malrieu2003,CattiauxGuillinMalrieu2008}. Reflection and synchronous
couplings also yield uniform-in-time PoC for suitable nonconvex
confinement and sufficiently weak interactions
\cite{DurmusEberleGuillinZimmer2020}.

Entropy methods provide another approach to this problem.
Jabin and Wang~\cite{JabinWang2018} obtain quantitative PoC estimates
through relative entropy, including for singular interactions.
Lacker~\cite{Lacker2023} develops a hierarchical entropy method giving
sharp estimates for fixed particle marginals, and Lacker and
Le Flem~\cite{LackerLeFlem2023} obtain its uniform-in-time counterpart
under suitable log-Sobolev and interaction assumptions.
In the context of optimal-transport-inspired sampling algorithms, namely Stein variational gradient descent (SVGD) and its Langevin-regularized version, Balasubramanian, Banerjee,
and Korba~\cite{BalasubramanianBanerjeeKorba2026} and
Banerjee and Kim~\cite{BanerjeeKim2026} combine quantitative
finite-time PoC estimates, derived through coupling or
moving-product relative-entropy methods in the spirit
of~\cite{JabinWang2018}, with entropy-based quantitative
convergence to the target, to obtain uniform-in-time PoC bounds.

For kinetic mean-field systems, Guillin and
Monmarch\'e~\cite{GuillinMonmarche2021} combine uniform relaxation
estimates with finite-time PoC to furnish uniform-in-time bounds in
nonconvex landscapes. In particular, they establish and utilize a uniform-in-particle-number log-Sobolev inequality for the associated Gibbs type joint invariant distribution. The connection between phase transitions,
log-Sobolev inequalities for the particle Gibbs measures, and
uniform-in-time PoC is studied in detail by Delgadino, Gvalani,
Pavliotis, and Smith~\cite{DelgadinoGvalaniPavliotisSmith2023}. Among other results, it is shown there (see \cite[Theorem 3.3]{DelgadinoGvalaniPavliotisSmith2023}) that the presence of a critical point for the associated free-energy that is not a global minimizer implies a decaying log-Sobolev constant in the particle number.  

When the free energy has several local minima, stability near one
minimum can coexist with rare transitions of the particle system to
`wells' around other minima. Local convergence results for the nonlinear equation \eqref{eq:MV}
remain available in this setting: Monmarch\'e and
Reygner~\cite{MonmarcheReygner2026} establish local nonlinear
log-Sobolev inequalities and quantitative convergence near stationary
solutions, including examples with a double-well confinement.
However, convergence of the nonlinear equation to a local minimum
does not by itself determine how long the finite-particle system
tracks that limit. In particular, a nonglobal minimum can attract the
mean-field trajectory even though the particle system eventually
leaves its neighborhood.

Continuous symmetries introduce a further issue. Bertini, Giacomin, and
Poquet~\cite{BertiniGiacominPoquet2014} study the long-time dynamics of the supercritical
noisy Kuramoto model, which corresponds to $W(x) =-K \cos x$, with $\beta K>2$, in \eqref{eq:particle-system}. They show that the empirical measure quickly approaches and
remains near the synchronization manifold, formed by the one-parameter family of rotations of the `clustered' stationary solution to \eqref{eq:MV}, while its phase (informally thought of as the rotational center of mass) undergoes
Brownian motion on time scales of order $N$. In particular, when the initialization of the empirical measure process approaches a stationary solution to \eqref{eq:MV} as $N \to \infty$, the limiting McKean--Vlasov trajectory remains fixed
but the pre-limit process can exhibit appreciable motion along the rotation orbit at times of order $N$. Coppini~\cite{Coppini2022} proves pathwise proximity to
the synchronized manifold on subexponential time horizons, also for
suitable interaction graphs. Delarue and Tse~\cite{DelarueTse2025}
obtain uniform-in-time weak PoC on the torus and, for the supercritical
Kuramoto model, an $O(N^{-1})$ weak-error estimate for smooth
rotation-invariant observables. These results motivate our use of
the rotation quotient $\Qq$, in which motion along an orbit is
identified. 

\paragraph{Large deviations and the free-energy landscape.}
The rare fluctuations responsible for transitions between wells around local minima are
described by a path-space large deviations principle for the empirical measure.
Dawson and G\"artner~\cite{DawsonGartner1987} establish such
a principle for weakly interacting diffusions, with
speed $N$ and a dynamical action measuring the cost of departures
from the McKean--Vlasov evolution. Budhiraja, Dupuis, and
Fischer~\cite{BudhirajaDupuisFischer2012} develop a general approach
based on variational representations and weak convergence of
controlled systems. In the reversible setting, Dawson and
G\"artner~\cite{DawsonGartner1989} relate the dynamical action,
quasipotentials, and the free-energy functional governing equilibrium
large deviations. These results provide the probabilistic foundation
for interpreting free-energy barriers as transition costs for the empirical measure process in the large deviations sense.

For periodic McKean--Vlasov equations, the structure of stationary
states (equivalently, critical points of the free energy) is studied by Carrillo, Gvalani,
Pavliotis, and Schlichting~\cite{CarrilloEtAl2020}, including an analysis of first-order bifurcations and phase transitions. Further light is shed on the geometry of stationary states, including higher-order bifurcations, periodicity, connections between bifurcations and discontinuous phase transitions, and the global free-energy minimizing branch, by Balasubramanian, Banerjee and Rigollet \cite{BalasubramanianBanerjeeRigollet2025}.  

Gvalani and Schlichting~\cite{GvalaniSchlichting2020} prove a `mountain
pass theorem' in the space of probability measures and apply it to
the McKean--Vlasov free energy. Under its hypotheses, the theorem identifies a critical point at the minimax free-energy level separating two metastable wells. When there is a discontinuous phase transition of the free-energy at some $\beta=\beta_c$,
they identify an intervening critical level of higher energy between the coexisting global minima (one of them the uniform distribution) and use the path-space
LDP of \cite{DawsonGartner1987} to show that the probability of a transition from the uniform state
to the clustered minimum state is exponentially small in the particle number $N$ on a fixed time horizon. This gives an appealing geometric picture of the free-energy landscape and its effect on rare transitions, but does not quantify the sharp transition thresholds between wells which are typically exponentially large in $N$. Indeed, for this, one cannot directly apply fixed-horizon large deviation estimates. Quantifying this sharp transition is one of our main goals.

\paragraph{Metastability and its relation to propagation-of-chaos.}
Metastability loosely refers to the tendency of a stochastic system to appear stationary over long periods of time before making an abrupt, rare transition to another state, where it can again persist for a long time. These persistent states, called metastable states, typically correspond to fixed points of the associated noiseless system.
The foundational framework for this analysis is the Freidlin--Wentzell
theory~\cite[Chapter~6]{FreidlinWentzell2012}. For small-noise
diffusions, through a multiscale path-space large deviations analysis, it lays down the so-called `cycle hierarchy', which groups metastable states into cycles by the energy barrier needed to be overcome in moving from one state to the other. This theory, which is most well-developed for finite-dimensional diffusions, furnishes sharp exponential time scales for transitions from one metastable state to another.

Dawson and G\"artner~\cite{DawsonGartner1986} already announced, without proof, such metastable phenomena for certain mean-field particle systems in the large-$N$ regime, which they described as `tunnelling'.
Quantitative transition-time results have subsequently been obtained
for specific continuous-spin models. In particular, Bashiri and
Menz~\cite{BashiriMenz2021} use potential theory to derive estimates on the expected transition time between metastable wells (Eyring--Kramers-type formula) for a continuous mean-field model with
double-well confinement and quadratic attractive interaction, with a relative
error tending to zero in a joint large-$N$, low-temperature regime.
There is also a substantial small-noise theory for the nonlinear McKean--Vlasov
diffusion itself (whose marginal laws satisfy \eqref{eq:MV}); Herrmann, Imkeller, and
Peithmann~\cite{HerrmannImkellerPeithmann2008} prove a large deviation
principle and a Kramers-type exit law for self-stabilizing diffusions. Adams, Engel, and Gvalani~\cite{AdamsEngelGvalani2026} establish,
in the low-temperature limit with a fixed number of particles,
a separation between rapid relaxation to a quasi-stationary
droplet and exponentially slow escape.
In our setting, the particle noise is fixed and the smallness parameter
arises from empirical fluctuations as $N\to\infty$; see \eqref{eq:smallNemp}. For mean-field jump processes on a finite state space, Yasodharan and
Sundaresan~\cite{YasodharanSundaresan2023} develop a cycle analysis
yielding sharp exponential estimates for exit times and convergence
to equilibrium. Journel and Le Bris~\cite{JournelLeBris2025} prove uniform-in-time
propagation-of-chaos for the Curie--Weiss model conditioned on
maintaining positive magnetization. These works provide finite-dimensional precedents
for the analysis of metastability in the large-particle limit.

A particularly close connection between metastability and PoC is
established by Monmarch\'e~\cite{Monmarche2025}. For particles with
quadratic attractive interaction, he proves quantitative chaos
estimates on exponentially long horizons near a stable stationary
state, together with exit times that are exponentially large in the particle number $N$ and approximately exponentially distributed after normalization, for suitable metastable domains. His argument modifies the dynamics
outside the selected domain so that the modified particle Gibbs
measure satisfies a log-Sobolev inequality uniformly in $N$.
This gives a rigorous form of metastable PoC; identifying matching
upper and lower exponents for the mean exit time is left
as an open problem there.

\paragraph{Comparison with our work.}
To our knowledge, the present work provides the first sharp exponential-time analysis of the Dawson--G\"artner tunnelling phenomenon~\cite{DawsonGartner1986} for mean-field diffusions with general smooth interaction potentials, under suitable landscape hypotheses. The closest precursors, Bashiri--Menz~\cite{BashiriMenz2021} and Monmarch\'e~\cite{Monmarche2025}, consider quadratic attractive interaction potentials where the interaction drift depends on the empirical distribution only through its mean, which partly simplifies the analysis. Under more general interactions, we have to work directly with empirical measures (modulo rotations) which live on an infinite-dimensional space. In~\cite[Theorem II]{BashiriMenz2021}, two-sided estimates on the expected transition time between metastable wells are derived, where the upper and lower bounds differ only up to multiplicative constants independent of $N$. However, this involves initialization from specially chosen last-exit distributions, which exhibit concentration on hyperplanes of prescribed mean. These expectation bounds also do not by themselves yield a two-sided probability statement. In comparison, our cycle-exit estimates give two-sided probability bounds with matching sharp exponents, uniformly over initial empirical measures in prescribed neighborhoods of the local minima. Moreover, we upgrade our exit estimates to localization estimates at deterministic times, which gives explicit deterministic thresholds for the validity and breakdown of propagation-of-chaos, a question that was left open in~\cite{Monmarche2025}.

A difference from the previous works described above is that they mostly work on $\mathbb{R}^d$, whereas our particle system lives on the circle $\T$. This choice is motivated by the specific applications we have in mind (eg. the Kuramoto model and noisy transformers), but it also buys us a number of technical advantages through compactness. For non-compact state spaces, we believe that several ideas from this paper will be applicable, but this is left for future work.

\section{Main Results}\label{sec:main-results}

We now state our main results. We first introduce some key objects. 

Let
\[
    \mathfrak M=\{\mathcal M_1,\ldots,\mathcal M_J\}
\]
be the collection of all local-minimum orbits of \(\F\) in \(\Qq\).

Write $f_i=\F(\mathcal M_i)$ and
$F_{\min}=\min_{\nu\in\Pp(\T)}\F(\nu)$.
For a local-minimum orbit $\mathcal M_i$, its deterministic basin of attraction is
\begin{equation}\label{eq:doa}
 \mathcal A_i=\{q\in\Qq:d_{\Qq}(S_tq,\mathcal M_i)\longrightarrow0\}.
\end{equation}
For $1 \le i,j \le J$, define the \emph{communication height} between $\mathcal M_i, \mathcal M_j$ by
\begin{equation}\label{eq:pairwise-communication-height}
 \Phi_{ij} :=
    \inf_{\substack{\gamma\in C([0,1];\Qq)\\
                    \gamma(0)=\mathcal M_i,\ \gamma(1)=\mathcal M_j}}
       \sup_{s\in[0,1]}\F(\gamma(s)).
\end{equation}
As we will see later, $\Phi_{ij}$ can be informally viewed as the height of the `minimal saddle' (in the large deviations sense) encountered by the (quotient) empirical measure process in transitioning from the attraction basin of $\mathcal M_i$ to the basin of $\mathcal M_j$.
For every \(i\) such that \(A_i:=\{j:f_j<f_i\}\) is nonempty,
define
\begin{equation}\label{eq:descent-height-definition}
 h_i:=\min_{j\in A_i}\Phi_{ij},
 \qquad
 D_i:=\beta(h_i-f_i).
\end{equation}
$D_i$ determines the exponential time scale for
reaching a lower-energy basin from a neighborhood of
$\mathcal M_i$. Set
\begin{equation}\label{eq:maximal-bottom-cycle}
 C^{max}_i=\{i\}\cup\{j\ne i:\Phi_{ij}<h_i\}.
\end{equation}
$C^{max}_i$ can be viewed as the `maximal metastable well' in which the empirical measure lives before it moves to a lower energy basin. $D_i$ will be called the `depth' of the cycle $C^{max}_i$.

Also set
\[
    \PN(\T)
    :=
    \left\{
       \frac1N\sum_{i=1}^N\delta_{x_i}:x_i\in\T
    \right\},
\]
which is the space the empirical measure lives on.
For $\nu\in\PN(\T)$, write $\mathbb P_\nu$ for the empirical
process law started at $\mu_0^N=\nu$.  Permutation symmetry makes this
law independent of the labeling of the initial configuration.

\subsection{Metastability and phase transition}
The following result quantifies a sharp exponential time metastability threshold for the empirical measure process $[\mu^N_t]$. It shows that, when started sufficiently close to a nonglobal minimum orbit $\mathcal M_i$, the process $[\mu^N_t]$ remains close in law to $\mathcal M_i$ for approximate time $e^{ND_i}$ (in the exponential scale). After this time, the process moves to the vicinity of a lower energy orbit and metastability around $\mathcal M_i$ ceases to hold.
\begin{theorem}[The sharp metastable window]
\label{prop:verified-MT}
Assume \eqref{eq:W-assumption} and
\eqref{eq:finite-critical-orbits}, and suppose that the
local-minimum values $f_1,\ldots,f_J$ are pairwise distinct. For a nonglobal minimum index $i$, set
$\mathcal L_i=\{\mathcal M_j:f_j<f_i\}$.
For every $r>0$ and $0<\varepsilon<D_i$ there are
$\eta,c,C>0$ such that
\begin{equation}\label{eq:verified-M1}
 \sup_{d_{\Qq}([\nu],\mathcal M_i)\le\eta}
 \sup_{0\le t\le e^{N(D_i-\varepsilon)}}
 \mathbb P_\nu\left(d_{\Qq}([\mu_t^N],\mathcal M_i)>r\right)
 \le Ce^{-cN}.
\end{equation}
For every $r,\varepsilon>0$, possibly with different $\eta,c,C$,
\begin{equation}\label{eq:verified-M2}
 \sup_{d_{\Qq}([\nu],\mathcal M_i)\le\eta}
 \sup_{t\ge e^{N(D_i+\varepsilon)}}
 \mathbb P_\nu\left(\dist_{\Qq}([\mu_t^N],\mathcal L_i)>r\right)
 \le Ce^{-cN}.
\end{equation}
Both suprema over initial measures are restricted to $\PN(\T)$, and
both bounds hold for all sufficiently large $N$. 
\end{theorem}

\begin{remark}[Localization and excursions to other wells]
Estimate~\eqref{eq:verified-M1} concerns the one-time distributions
of the empirical measure, uniformly over deterministic observation
times. The supremum over time is outside the probability; the
estimate does not assert that
\[
 \mathbb P_\nu\left(
   \sup_{0\le t\le e^{N(D_i-\varepsilon)}}
   d_{\Qq}([\mu_t^N],\mathcal M_i)>r
 \right)\longrightarrow0.
\]
When $C_i^{max}$ contains other minima, the empirical process may
visit neighborhoods of these wells and return to $\mathcal M_i$
on intermediate time scales below $e^{ND_i}$.
Since $\mathcal M_i$ is the unique bottom of $C_i^{max}$, all
these other minima have strictly higher free-energy.
Such excursions occur at random times, and the holding times
around these higher-energy minima are exponentially shorter
than the overall metastable time scale \(e^{ND_i}\)
associated with \(\mathcal M_i\). Consequently, this picture is compatible with
an exponentially small probability of being away from
$\mathcal M_i$ at any specified deterministic time $t\le e^{N(D_i-\varepsilon)}$, as in
\eqref{eq:verified-M1}.
The depth $D_i$ governs escape toward lower-energy wells,
rather than the first excursion to another well within
$C_i^{max}$.
\end{remark}

\subsection{Pointwise propagation-of-chaos results near metastable wells}

We now apply Theorem~\ref{prop:verified-MT} to initial empirical
measures converging to a law in the attraction basin of a minimum.
This gives sharp estimates for the time horizon over which propagation-of-chaos holds, or, in other words, the discrepancy in mean-field approximation of the empirical measure by the McKean--Vlasov limit, vanishes as $N \to \infty$. It also shows that, beyond these time horizons (in the exponential scale), this approximation fails in an appropriate sense. 

Throughout this section, assume that the initial
configuration, which may be random, is independent of the driving
Brownian motions. We also assume that the initial empirical measure $\mu_0^N$ satisfies
\begin{equation}\label{eq:poc-initial}
 \Wass(\mu_0^N,\mu_0)\xrightarrow{\mathbb P}0
 \quad\text{for a deterministic }\mu_0\in\Pp(\T),
\end{equation}
and write $\mu_t=S_t\mu_0$.

These empirical-measure statements do not assert convergence of
unrotated fixed-particle marginals to a deterministic product law.
In particular, a common rotational phase may fluctuate on long
time scales even for exchangeable configurations \cite{BertiniGiacominPoquet2014,Coppini2022}.
No independence or exchangeability of the initial particles is
required; in particular, both deterministic configurations and
i.i.d.\ initialization are allowed. $\mathbb E$ and $\mathbb P$ will denote the expectation and probability under the joint randomness of the initialization and evolution.

\begin{theorem}[The sharp metastable window for propagation-of-chaos]
\label{thm:poc-metastable}
Assume \eqref{eq:W-assumption}, and \eqref{eq:finite-critical-orbits} and that
$f_1,\ldots,f_J$ are pairwise distinct. Fix a nonglobal minimum $i$
and suppose \eqref{eq:poc-initial} holds with $[\mu_0]\in\mathcal A_i$.
Let $D_i$, defined in \eqref{eq:descent-height-definition}, be the depth of the maximal cycle $C_i^{max}$ from
\eqref{eq:maximal-bottom-cycle}. 
\begin{enumerate}[label=\textup{(\roman*)}]
\item For every $0<\varepsilon<D_i$,
\begin{equation}\label{eq:poc-metastable-before}
 \sup_{0\le t\le e^{N(D_i-\varepsilon)}}
 \mathbb E\,d_{\Qq}([\mu_t^N],[\mu_t])^2\longrightarrow0.
\end{equation}
Equivalently, for every $\delta>0$,
\[
 \sup_{0\le t\le e^{N(D_i-\varepsilon)}}
 \mathbb P\bigl(d_{\Qq}([\mu_t^N],[\mu_t])>\delta\bigr)\longrightarrow0.
\]
\item Define the positive constant
\[
 \delta_i:=
 \frac14\dist_{\Qq}(\mathcal M_i,\mathcal L_i),
 \qquad
 \mathcal L_i=\{\mathcal M_j:f_j<f_i\}.
\]
For every $\varepsilon>0$,
\begin{equation}\label{eq:poc-metastable-after}
 \sup_{t\ge e^{N(D_i+\varepsilon)}}
 \mathbb P\bigl(d_{\Qq}([\mu_t^N],[\mu_t])\le\delta_i\bigr)\longrightarrow0.
\end{equation}
In particular, the approximation of the empirical measure process $[\mu_t^N]$ by the McKean--Vlasov trajectory $[\mu_t]$ arising as the mean-field limit (in the sense of Lemma \ref{lem:poc-finite-time}) fails
at every deterministic sequence $t_N\ge e^{N(D_i+\varepsilon)}$.
\end{enumerate}
\end{theorem}

\begin{remark}[Role of the distinctness assumption]
\label{rem:poc-distinctness}
Pairwise distinctness of the minimum energies is imposed for
simplicity. For a fixed nonglobal minimum $i$, the proofs of
Theorems \ref{prop:verified-MT} and~\ref{thm:poc-metastable} only require
\begin{equation}\label{eq:pairrelax}
 f_j=f_i,\quad j\ne i
 \quad\Longrightarrow\quad
 \Phi_{ij}\ge h_i.
\end{equation}
Indeed, this condition ensures that $\mathcal M_i$ is the unique
bottom of $C_i^{max}$: every other minimum in that cycle has energy
strictly greater than $f_i$. This is the energy gap used in the
localization argument before the threshold. Equal energies among
the other minima are allowed.

Moreover, keeping the other hypotheses, the conclusions in \eqref{eq:verified-M2} and
\eqref{eq:poc-metastable-after} do not require any distinctness
assumption. Their proofs use only the time to reach a strictly
lower minimum and the subsequent localization near the set
$\mathcal L_i$.

Without any such assumption ensuring a unique bottom of $C_i^{max}$, we still obtain quantitative estimates for metastability and propagation-of-chaos (and its failure), but a sharp common threshold at depth $D_i$ is no longer guaranteed. Define
\[
 h_i^{\le}
 :=\min_{\substack{j\ne i\\ f_j\le f_i}}\Phi_{ij},
 \qquad
 D_i^{\le}:=\beta(h_i^{\le}-f_i).
\]
The same argument, applied to the cycle containing $i$
immediately below height $h_i^{\le}$, proves \eqref{eq:verified-M1} and
\eqref{eq:poc-metastable-before} with $D_i^{\le}$ in place of
$D_i$. The conclusions in \eqref{eq:verified-M2} and \eqref{eq:poc-metastable-after} still hold with $D_i$. These two depths may differ: the condition \eqref{eq:pairrelax} provides a sufficient condition for them to coincide. When the two depths differ, then for times between $e^{ND_i^{\le}}$ and $e^{ND_i}$, transitions between
equal-energy wells can destroy concentration near a single
minimum before a strictly lower minimum is reached.

\end{remark}

An application of the large-deviations theory for the Gibbs-type stationary measure of the $N$-particle system (see Lemma~\ref{lem:uniform-uphill-endpoint}) gives the following \emph{uniform-in-time} propagation-of-chaos in the attraction basin of the global minimum orbit.

\begin{theorem}[Uniform-in-time propagation-of-chaos in the global basin]
\label{thm:poc-global}
Suppose $\F$ has a unique global-minimum orbit $\mathcal M_*$ with attraction basin $\mathcal A_*$ defined as in \eqref{eq:doa},
and \eqref{eq:poc-initial} holds with $[\mu_0]\in\mathcal A_*$.
Then
\begin{equation}\label{eq:poc-global}
 \sup_{t\ge0}\mathbb E\,d_{\Qq}([\mu_t^N],[\mu_t])^2\longrightarrow0.
\end{equation}
Equivalently, for every $\delta>0$,
$\sup_{t\ge0}\mathbb P(d_{\Qq}([\mu_t^N],[\mu_t])>\delta)\to0$.
This theorem does not require finitely many critical orbits,
distinct energies of the other minima, or an initial entropy bound.
\end{theorem}

We also obtain the following propagation-of-chaos result in the pathwise sense, over subexponential time horizons. We call a local-minimum orbit $\mathcal M_i$
\emph{asymptotically stable} if its attraction basin
$\mathcal A_i$ contains a neighborhood of $\mathcal M_i$
and it is Lyapunov stable: for every $r>0$, there exists
$\delta>0$ such that
\[
 d_{\Qq}(q,\mathcal M_i)<\delta
 \quad\Longrightarrow\quad
 d_{\Qq}(S_tq,\mathcal M_i)<r
 \qquad\text{for all }t\ge0.
\]
In the present setting, these properties ensure the existence
of arbitrarily small `stable gates' in the sense of
Definition~\ref{def:stable-gate}.

\begin{theorem}[Pathwise propagation-of-chaos on subexponential horizons]
\label{thm:poc-pathwise}
Let $\mathcal M_i$ be an asymptotically stable local-minimum orbit,
and suppose \eqref{eq:poc-initial} holds with
$[\mu_0]\in\mathcal A_i$. For every deterministic sequence $T_N>0$
such that $\log(1+T_N)=o(N)$,
\begin{equation}\label{eq:poc-pathwise}
 \mathbb E\left[
   \sup_{0\le t\le T_N}d_{\Qq}([\mu_t^N],[\mu_t])^2
 \right]\longrightarrow0.
\end{equation}
In particular, under \eqref{eq:finite-critical-orbits} this holds
for every local minimum, without distinctness of the minimum
energies.
\end{theorem}

\subsection{Sharp pathwise propagation-of-chaos results for running averages}\label{sec:pathav}

The pointwise propagation-of-chaos results in Theorem~\ref{thm:poc-metastable} and Theorem~\ref{thm:poc-global} are obtained in the distributional sense with the supremum in time taken outside the probability (or expectation), and the pathwise result in Theorem~\ref{thm:poc-pathwise} is over a subexponential horizon. In the following results, we obtain sharp pathwise propagation-of-chaos thresholds, with the \emph{supremum inside the probability or expectation}, for time-averaged quantities involving the empirical measures. Theorem~\ref{thm:poc-running-averages} obtains such results for the time-averaged squared distance between the empirical measure and its McKean--Vlasov limit. The upper bounds are then translated into propagation-of-chaos results for (suitably rotated) time-averaged empirical measure paths, and associated occupation measures on $\Qq$ (elements of $\Pp(\Qq)$), in Corollary~\ref{cor:poc-averaged-measures}.

\begin{theorem}[Pathwise propagation-of-chaos results for running averages]
\label{thm:poc-running-averages}
\leavevmode\par
\begin{enumerate}[label=\textup{(\roman*)}]
\item Under the hypotheses of Theorem~\ref{thm:poc-metastable},
for every $0<\varepsilon<D_i$,
\begin{equation}\label{eq:poc-average-before}
 \mathbb E\left[
   \sup_{0<T\le e^{N(D_i-\varepsilon)}}\frac1T\int_0^Td_{\Qq}([\mu_t^N],[\mu_t])^2\,\dd t
 \right]\longrightarrow0.
\end{equation}
\item Under the same hypotheses, set $a_i:=\delta_i^2/2$.
For every $\varepsilon>0$,
\begin{equation}\label{eq:poc-average-after}
 \sup_{T\ge e^{N(D_i+\varepsilon)}}
 \mathbb P\Bigl(\frac1T\int_0^Td_{\Qq}([\mu_t^N],[\mu_t])^2\,\dd t\le a_i\Bigr)\longrightarrow0.
\end{equation}
In particular,
$$\mathbb P\left(\sup_{0<T\le e^{N(D_i+\varepsilon)}}\frac1T\int_0^Td_{\Qq}([\mu_t^N],[\mu_t])^2\,\dd t>a_i\right)
\to1.$$
\item Under the hypotheses of Theorem~\ref{thm:poc-global},
\begin{equation}\label{eq:poc-average-global}
 \mathbb E\left[\sup_{T>0}\frac1T\int_0^Td_{\Qq}([\mu_t^N],[\mu_t])^2\,\dd t\right]\longrightarrow0.
\end{equation}
\end{enumerate}
In (i) and (iii), the random suprema also converge to zero in
probability. 
\end{theorem}

\subsubsection{Time averages of measures}

To pass from averaged squared distances to distances between time-averages of measures, first
align the empirical measures with the deterministic reference.
Choose a jointly measurable minimizing rotation $\theta_N(t)$
and set
\begin{equation}\label{eq:poc-alignment}
 \widetilde\mu_t^N:=R_{\theta_N(t)}\mu_t^N,
 \qquad
 \Wass(\widetilde\mu_t^N,\mu_t)^2=d_{\Qq}([\mu_t^N],[\mu_t])^2.
\end{equation}
Such a choice exists by measurable selection: $\T$
is compact, the rotation action is continuous and the $\Wass$ cost is jointly continuous in both measures. Define, with integrals of measures understood
weakly,
\[
 \overline{\widetilde\mu}_T^N
 :=\frac1T\int_0^T\widetilde\mu_t^N\,\dd t,\qquad
 \overline\mu_T:=\frac1T\int_0^T\mu_t\,\dd t.
\]
There is also a construction requiring no choice of rotations:
the occupation measures on $\Qq$,
\[
 \Xi_T^N:=\frac1T\int_0^T\delta_{[\mu^N_t]}\,\dd t,\qquad
 \Xi_T:=\frac1T\int_0^T\delta_{[\mu_t]}\,\dd t.
\]
Write $\mathsf W_{2,\Qq}$ for the Wasserstein-2 distance on
$\Pp(\Qq)$ with ground metric $d_{\Qq}$.

\begin{corollary}[Bounds for averaged measures]
\label{cor:poc-averaged-measures}
For every $T>0$, pathwise,
\begin{align}
 d_{\Qq}\bigl(
   [\overline{\widetilde\mu}_T^N],[\overline\mu_T]\bigr)^2
 &\le
 \Wass\bigl(\overline{\widetilde\mu}_T^N,\overline\mu_T\bigr)^2
 \le \frac1T\int_0^Td_{\Qq}([\mu_t^N],[\mu_t])^2\,\dd t,\label{eq:poc-aligned-averages}\\
 \mathsf W_{2,\Qq}(\Xi_T^N,\Xi_T)^2
 &\le \frac1T\int_0^Td_{\Qq}([\mu_t^N],[\mu_t])^2\,\dd t.\label{eq:poc-occupation-averages}
\end{align}
Consequently, each quantity on the left satisfies the same
vanishing expectation bound, with a supremum inside the
expectation, as \eqref{eq:poc-average-before} and
\eqref{eq:poc-average-global} in their respective regimes.
\end{corollary}

\begin{remark}[Unaligned averages and the failure statement]
\label{rem:poc-unaligned}
The alignment in \eqref{eq:poc-aligned-averages} is necessary.
For a nonuniform measure $\nu$, let $\nu_t=R_{2\pi t}\nu$
and $\zeta_t=\nu$ on $[0,1]$. Their quotient distance is zero
at every time, but
\[
 \int_0^1\nu_t\,\dd t=\lambda,\qquad
 \int_0^1\zeta_t\,\dd t=\nu,\qquad
 d_{\Qq}([\lambda],[\nu])>0.
\]
Thus an average of quotient distances does not control the
quotient distance between unaligned averages.
Also, \eqref{eq:poc-average-after} proves failure for the
averaged squared error. The upper bounds
\eqref{eq:poc-aligned-averages}--\eqref{eq:poc-occupation-averages}
alone imply no corresponding lower bound for distances between
averaged measures.
\end{remark}

\subsection{Outline of the proofs}
\label{subsec:proof-outline}
The starting point for the above results is to view the empirical measure process $\mu^N_t$ informally as a small-noise process whose effective noise variance is of order $N^{-1}$, or whose fluctuation amplitude is of order $N^{-1/2}$. Indeed, for $\varphi\in C^2(\T)$, It\^o's formula gives
\begin{equation}\label{eq:smallNemp}
 \dd\big[\langle\mu^N,\varphi\rangle\big]_t
 =\frac{2}{\beta N^2}\sum_{i=1}^N
   |\varphi'(X_t^{i,N})|^2\,\dd t
 =\frac{2}{\beta N}
   \langle\mu_t^N,|\varphi'|^2\rangle\,\dd t,
\end{equation}
where $[\cdot]$ denotes quadratic variation. As discussed before, the Freidlin--Wentzell cycle method
\cite[Chapter~6]{FreidlinWentzell2012}, originally constructed for finite-dimensional small-noise diffusions, gives sharp transition estimates between metastable states through a detailed excursion theory and large-deviations analysis. The path-space large deviations principle (LDP) for $\mu^N_t$ over any given compact time horizon $[0,T]$ is well-understood due to the work of Dawson--G\"artner \cite{DawsonGartner1987,DawsonGartner1989} (see also \cite{GvalaniSchlichting2020}). Our key technical contribution is to exploit the Dawson--G\"artner LDP over long time-horizons and use it to extend the Freidlin--Wentzell theory to the infinite-dimensional quotient space $\Qq$. 
 
We outline the principal steps and the difficulties
they address.

\paragraph{Finite-horizon large deviations and uniformity.}

We specialize the Dawson--G\"artner path-space LDP to the present
periodic model, with the distributional and control domain
specified in Section~\ref{sec:LDP}; see
\cite{DawsonGartner1987,BudhirajaDupuisFischer2012} and the
torus formulation in \cite[Theorem~5.9]{GvalaniSchlichting2020}. Its speed is $N$, and its dynamical action is
\[
 I_{[0,T]}(\mu)
 =\frac{\beta}{4}\int_0^T
   \|\partial_t\mu_t-\Gg(\mu_t)\|_{\dot H^{-1}(\mu_t)}^2\,\dd t,
\]
with the domain and weighted norm specified in
Section~\ref{sec:LDP}; here $\Gg$ is the right-hand side of
\eqref{eq:MV}. Compactness of $\T$ identifies the uniform weak
and uniform Wasserstein path topologies. The sequential LDP for
converging initial empirical measures then yields the uniform
upper bounds and locally uniform lower bounds of
Proposition~\ref{prop:uniform-LDP}. This uniformity is essential
when restarting the process at random entrance points of gates,
which are small neighborhoods of critical orbits. It also allows
atomic initial measures, without an initial entropy assumption.

\paragraph{Connecting the free-energy geometry to action costs of paths.}
The action--energy identity in Lemma~\ref{lem:action-energy}
shows that reversing a deterministic measure-valued trajectory solving \eqref{eq:MV} costs $\beta$
times its free-energy increase. This links local transition action costs
to the free-energy landscape. Moreover, due to the gradient flow structure of \eqref{eq:MV} the free-energy decreases along its solution trajectories. This energy dissipation (see Lemma~\ref{lem:energy-dissipation}), along with finiteness of the number of critical orbits and parabolic smoothing, is used to prove convergence of deterministic solutions of \eqref{eq:MV}
to a single critical orbit and asymptotic stability of the local
minima.

A key obstacle is that Wasserstein closeness alone does not
control free-energy, and empirical measures have infinite
entropy relative to $\lambda$, even when they are close to a smooth measure in Wasserstein distance. This differs from the classical
finite-dimensional gradient setting of
\cite{FreidlinWentzell2012}, where the potential is finite and
continuous in the state variable. Thus, special care is required when connecting nearby points in a local minimum neighborhood by short paths of small action. This can be done either when the points are close in the $C^1(\T)$ topology (often achieved by `deterministic smoothing' via the flow \eqref{eq:MV}), or one point is a local minimum and the free-energy of the other is close to that of this minimum. See Step 2 in the proof of Lemma~\ref{lem:singleton-exit-cost} and Step 1 in the proof of Lemma~\ref{lem:uniform-shortening-arbitrary-gate-points} for such applications. Thus, to understand long-time behavior, solely tracking excursions of the empirical measure process between Wasserstein neighborhoods of the local minima is insufficient. One also needs to control auxiliary excursions between `low-energy cores', where the energy dissipation bounds control the excursion durations outside these cores; see the proof of Lemma~\ref{lem:uniform-shortening-arbitrary-gate-points}.

\paragraph{From fixed horizons to stopped excursions.}
A fixed-horizon LDP does not directly estimate a transition whose
duration is unbounded. To this end, we prove a uniform path-shortening result, Lemma~\ref{lem:uniform-shortening-arbitrary-gate-points}.
This key estimate shows that for any $A< \infty$ and $\varepsilon>0$, any path with action at most $A$
can be replaced by one with the same endpoints and duration bounded
only in terms of $A$ and $\varepsilon>0$, at an additional action
cost of at most $\varepsilon$. This lets us utilize the finite-horizon LDP to paths on unbounded time intervals.
Combining this with deterministic recurrence properties of solutions to \eqref{eq:MV} and connector constructions culminates in Lemma~\ref{lem:uniform-stopped-excursion}, which shows that the probability of transition from a (sufficiently small) gate around $\mathcal M_i$ to another small gate around $\mathcal M_j$ is approximately $e^{-N\Delta_{ij}}$, where $\Delta_{ij}$, defined in \eqref{eq:limiting-direct-cost-variational}, gives the `direct cost', namely, the minimum
action required to connect $\mathcal M_i$ to $\mathcal M_j$
while avoiding all local minimum orbits in the interior
of the path. 

Remark~\ref{rem:FWcomp} gives a comparison of our techniques with those in \cite{FreidlinWentzell2012}.

\paragraph{The cycle hierarchy and the physical clock.}
For the long-time excursion analysis, recording only the successive labels of the visited gates around the local minima loses information about
the entrance points, so the resulting `gate label process' need not be Markov.
However, by the preceding analysis, its transition probabilities satisfy uniform exponential
bounds conditional on the entire past (see Lemma~\ref{lem:gate-transition-reduction}). We compare these bounds
with the reference Markov chain whose off-diagonal transition
probabilities are
\[
 Q_N(i,j)=\frac1{2J}e^{-N\Delta_{ij}},\qquad i\ne j,
\]
with diagonal entries chosen so that each row sums to one.
Comparing actions of reversed paths gives
\[
 E_{ij}:=\beta f_i+\Delta_{ij}
        =\beta f_j+\Delta_{ji}=E_{ji},
\]
so this chain is reversible with invariant weights proportional
to $e^{-N\beta f_i}$. 
This reversibility links the quasipotential to the free-energy
landscape.
In particular, the well-understood cycle hierarchy of 
finite-state reversible Markov chains \cite{Catoni1999,CirilloNardiSohier2015} allows us to quantify the number of excursions required by the empirical measure process to `exit a cycle', that is, leave a neighborhood of the collection of local minima whose internal free-energy communication heights are strictly less than that required to connect to an external minimum. Lemma~\ref{lem:conditional-FW}
transfers the finite-state cycle estimates
to the actual gate process using its conditional transition
bounds; Markovianity of the labels is not required.

To move from excursion count to elapsed physical time, we
prove two complementary bounds: individual excursions cannot
last too long (Lemma~\ref{lem:physical-gate-clock}), and too many excursions cannot occur in a short
time interval (Lemma~\ref{lem:fast-gate-clock}), both measured in the exponential scale in $N$. Together, these bounds yield
Theorem~\ref{prop:two-sided-physical-cycles}: the exit time from
a proper cycle $C$ is approximately
\[
 e^{Nd(C)},\qquad
 d(C)=\beta\bigl(\min_{i\in C, \,j\notin C}
       \Phi_{ij}-\min_{i \in C}f_i\bigr),
\]
where $\min_{i\in C, \,j\notin C}
       \Phi_{ij}$ is its exit height and $\min_{i \in C}f_i$ is its bottom energy.
The cycle hierarchy accounts for repeated visits and repeated
escape attempts among internal wells. Consequently, the exit
exponent is determined by the cycle depth $d(C)$; it is a truly `multiscale' object and need not equal
the action of a single controlled path, which may accumulate
several separate uphill costs (see Remark~\ref{rem:two-well-illustration} and discussion before it).

\paragraph{Deterministic time localization and Gibbs LDP.}
To prove our main results, we need to move from random cycle-exit times to localization estimates at deterministic times.
We combine the cycle-exit estimates with uniform recurrence and confinement estimates for neighborhoods of local minima, and use a
comparison technique with the stationary Gibbs law of the joint particle system, whose large deviations are well-understood. This gives localization estimates for the empirical measure, both pathwise and in law, at
deterministic observation times
(Section~\ref{sec:endpoint-lemmas}).
In particular, before the cycle-exit threshold, the law remains concentrated near
$\mathcal M_i$, the unique bottom of $C_i^{max}$, even though
sample paths may visit other wells within this cycle.
After the threshold, it concentrates near the collection
$\mathcal L_i$ of strictly lower minima.
These conclusions give \eqref{eq:verified-M1} and
\eqref{eq:verified-M2}, proving
Theorem~\ref{prop:verified-MT}.

\paragraph{Propagation-of-chaos and running averages.}
The localization results, combined with finite-time propagation-of-chaos estimates, yield the sharp metastable PoC window and failure beyond it (Theorem~\ref{thm:poc-metastable}).
The Gibbs comparison also gives uniform-in-time PoC in the basin
of the unique global-minimum orbit, without assuming finitely many
critical orbits (Theorem~\ref{thm:poc-global}), while local confinement
gives pathwise PoC on subexponential horizons
(Theorem~\ref{thm:poc-pathwise}).
Finally, maximal estimates yield convergence of running averages
uniformly over the metastable window, and over all averaging
horizons in the global basin
(Theorem~\ref{thm:poc-running-averages}).
Aligning rotations then gives the corresponding bounds for
time-averaged measures
(Corollary~\ref{cor:poc-averaged-measures}).

\section{Applications and explicit interaction potentials}
\label{sec:applications}

We now present some applications of our results. In Section~\ref{subsec:app-kuramoto}, we apply them to the noisy Kuramoto model, where $W$ has only one Fourier mode and the energy landscape is well-understood, and recover existing results from the literature. In Section~\ref{subsec:app-finite}, we provide a checkable Fourier criterion for establishing the finiteness of the number of critical orbits. In Section~\ref{subsec:app-multiple-wells}, we connect our metastability and propagation-of-chaos results with phase transitions in the free-energy landscape. Finally, in Section~\ref{subsec:app-transformer}, we give applications to the noisy transformer model.

\subsection{The noisy Kuramoto model}
\label{subsec:app-kuramoto}

Take $W(x)=-K\cos x$, with $K>0$, and put $A=\beta K$. In the subcritical regime $A \le 2$, the uniform distribution gives the unique stationary measure for the associated McKean--Vlasov equation \eqref{eq:MV}. In the
supercritical regime $A>2$, the stationary measures are $\lambda$ and
the synchronized family
\begin{equation}\label{eq:app-kuramoto-stationary}
 \mathcal M_* = \{\rho_\psi\lambda:\psi\in\T\},
 \qquad
 \rho_\psi(x)=\frac{\exp\{Ar\cos(x-\psi)\}}{I_0(Ar)},
 \qquad
 r=\frac{I_1(Ar)}{I_0(Ar)}>0.
\end{equation}
Here
\[
 I_k(z):=\int_\T e^{z\cos x}\cos(kx)\,\lambda(\dd x),
 \qquad k\ge0,
\]
denote the modified Bessel functions of the first kind,
and the positive solution $r$ is unique; see \cite{BertiniGiacominPakdaman2010}. The synchronized family $\mathcal M_*$ is the unique
global-minimum orbit; see
\cite[Proposition~6.1]{CarrilloEtAl2020}.
For $\beta K>2$, the linearization of \eqref{eq:MV} at the
uniform measure $\lambda$ has the positive eigenvalue
$K/2-\beta^{-1}$. Thus $\lambda$ is linearly unstable
in this regime; see also \cite[Section~3.2]{CarrilloEtAl2020}.
Write
$m_k(\mu):=\int_\T e^{ikx}\,\mu(\dd x)$. If $m_1(\mu_0)\ne0$,
the deterministic solution approaches $\mathcal M_*$; see
\cite[Proposition~4.1]{DelarueTse2025}.
Consequently, Theorem~\ref{thm:poc-global} gives
\begin{equation}\label{eq:app-kuramoto-global}
 \sup_{t\ge0}\mathbb E\,
 d_{\Qq}([\mu_t^N],[\mu_t])^2\longrightarrow0
\end{equation}
under the empirical initial convergence assumption
\eqref{eq:poc-initial}. There is only one local minimum in $\Qq$, so the metastability result in Theorem~\ref{thm:poc-metastable} for nonglobal minima is not applicable in this case.

Delarue and Tse~\cite[Theorem~4.3]{DelarueTse2025} obtain the stronger
quantitative weak estimate
\begin{equation}\label{eq:app-dt-weak}
 \sup_{t\ge0}
 \left|\mathbb E\Phi(\mu_t^N)-\Phi(\mu_t)\right|
 \le \frac{C_{\Phi,\eta}}{N}
\end{equation}
for their regularity class of smooth rotation-invariant functionals $\Phi$,
with i.i.d.\ initial particles of law $\mu_0$ satisfying
$|m_1(\mu_0)|\ge\eta>0$. Their Corollary~4.5 gives convergence results in negative Sobolev norms, after rescaling space and time
to match our conventions,
\begin{equation}\label{eq:app-dt-distance}
 \mathbb E\inf_{\nu\in\mathcal M_*}
       \|\mu_t^N-\nu\|_{H^{-s}}^2
 \le \frac CN+Ce^{-ct},
 \qquad \frac12<s\le1.
\end{equation}
This implies \eqref{eq:app-kuramoto-global} on the common class of
initial conditions. Indeed, on the compact space $\Pp(\T)$ both
$H^{-s}$ distance and $\Wass$ metrize weak convergence and are
uniformly equivalent. Thus \eqref{eq:app-dt-distance} implies
\[
 \lim_{T\to\infty}\limsup_{N\to\infty}
 \sup_{t\ge T}
 \mathbb E\,d_{\Qq}([\mu_t^N],\mathcal M_*)^2=0.
\]
Combine this with $d_{\Qq}([\mu_t],\mathcal M_*)\to0$ and
finite-time propagation-of-chaos on $[0,T]$.

Conversely, \eqref{eq:app-kuramoto-global} implies
\[
 \sup_{t\ge0}
 \left|\mathbb E\Phi(\mu_t^N)-\Phi(\mu_t)\right|
 \longrightarrow0
\]
for every continuous rotation-invariant functional $\Phi$,
by uniform continuity on the compact quotient $\Qq$.
Thus, on the common class of initial conditions, the
qualitative conclusions are equivalent. This equivalence
does not recover the quantitative $N^{-1}$ weak-error bound.

Coppini~\cite[Theorem~2.3]{Coppini2022} proves proximity of the empirical measure process in probability to the
synchronized manifold $\mathcal M_*$ in $H^{-1}$, provided the initial empirical measure approaches this manifold in a suitable sense as $N \to \infty$, uniformly along the sample path
up to times $T_N$ with $\log(1+T_N)=o(N)$. His result includes
interaction graphs satisfying a quantitative connectivity condition.
For the complete graph, it is the Kuramoto analogue of
Theorem~\ref{thm:poc-pathwise}:
\begin{equation}\label{eq:app-kuramoto-pathwise}
 \mathbb E\left[
  \sup_{0\le t\le T_N}
  d_{\Qq}([\mu_t^N],[\mu_t])^2
 \right]\longrightarrow0.
\end{equation}
The extension from initialization near a synchronized profile to its
basin uses finite-time approximation and deterministic attraction,
as explained before his theorem. Compactness converts convergence in
probability in $H^{-1}$ into the bounded Wasserstein expectation in
\eqref{eq:app-kuramoto-pathwise}. 

\subsection{A Fourier criterion for finitely many critical orbits}
\label{subsec:app-finite}

Consider
\begin{equation}\label{eq:app-fourier-potential}
 W(x)=-\sum_{k\ge1}a_k\cos(kx),
 \qquad a_1>0,
 \qquad \sum_{k\ge1}k^4|a_k|<\infty.
\end{equation}
The other coefficients may have either sign. An additive constant
in $W$ has no effect. The summability assumption guarantees
$W\in C^4(\T)$, as required in \eqref{eq:W-assumption}.

\begin{proposition}[A contraction criterion for finiteness]
\label{prop:app-finite}
For \eqref{eq:app-fourier-potential}, set
\[
 A=\beta a_1,\qquad B=\beta|a_2|,
 \qquad R=\beta\sum_{k\ge3}|a_k|,
 \qquad
 q_0=B\left(1-\frac1{2\cosh(A)I_0(B)}\right).
\]
Either of the following conditions implies that
$\Crit(\F)/\T$ is finite:
\begin{align}
 B+R&<1; \label{eq:app-simple-condition}\\
 e^{2R}(q_0+R)&<1. \label{eq:app-refined-condition}
\end{align}
In particular, for a two-harmonic potential ($a_k=0$ for $k \ge 3$), $R=0$ and the refined condition is
\begin{equation}\label{eq:app-two-mode-condition}
 \beta|a_2|\left(
 1-\frac1{2\cosh(\beta a_1)I_0(\beta|a_2|)}
 \right)<1.
\end{equation}
\end{proposition}

\subsection{Transition points, metastability and propagation-of-chaos}
\label{subsec:app-multiple-wells}

Consider the potential in \eqref{eq:app-fourier-potential}, and set
\[
 a:=\max_{k\ge1}a_k>0,\qquad \beta_*:=\frac2a.
\]
It is convenient to use the rescaled free-energy
\begin{equation}\label{eq:app-scaled-energy}
 \mathcal E_\beta(\mu):=\beta\F(\mu)
 =\Ent(\mu\mid\lambda)-\frac\beta2
       \sum_{k\ge1}a_k|m_k(\mu)|^2,
 \qquad
 m_k(\mu):=\int_\T e^{ikx}\,\mu(\dd x).
\end{equation}
Thus $\mathcal E_\beta(\lambda)=0$. Multiplication by $\beta>0$
does not change the critical points or the local and global minima.
The value $\beta_*$ is the linear-stability threshold of $\lambda$:
the second variation of $\mathcal E_\beta(\cdot)$ is positive definite for $\beta<\beta_*$,
has a nontrivial kernel at $\beta=\beta_*$, and has a negative
direction for $\beta>\beta_*$. Consequently, under suitable assumptions, there exist bifurcating branches of nonuniform stationary solutions emerging from the uniform solution as $\beta$ varies in a neighborhood of $\beta_*$; see \cite{CarrilloEtAl2020,BalasubramanianBanerjeeRigollet2025}.

We use the phase-transition terminology of Chayes--Panferov
\cite[Section~2.4]{ChayesPanferov2010} and Carrillo--Gvalani--Pavliotis--Schlichting
\cite[Definitions~5.1 and~5.4]{CarrilloEtAl2020}.
A \emph{transition point} $\beta_c>0$ satisfies:
\begin{enumerate}[label=\textup{(\alph*)}]
\item $\lambda$ is the unique global minimizer for every
      $0<\beta<\beta_c$;
\item $\lambda$ is a global minimizer at $\beta=\beta_c$;
\item a nonuniform global minimizer exists for every $\beta>\beta_c$.
\end{enumerate}
For the present class of potentials, this point exists, is unique,
and is equivalently given by
\begin{equation}\label{eq:app-transition-point}
 \beta_c:=\inf\left\{\beta>0:
       \min_{\mu\in\Pp(\T)}\mathcal E_\beta(\mu)<0\right\},
 \qquad 0<\beta_c\le\beta_*.
\end{equation}
In particular, $\lambda$ is not a global minimizer for
$\beta>\beta_c$; see
\cite[Propositions~5.2 and~5.3]{CarrilloEtAl2020}.

The transition is \emph{continuous} if $\lambda$ is the unique
global minimizer at $\beta_c$ and every family of global
minimizers $\rho_\beta\lambda$, $\beta>\beta_c$, satisfies
\begin{equation}\label{eq:app-continuous-transition}
 \lim_{\beta\downarrow\beta_c}
       \|\rho_\beta-1\|_{L^1(\lambda)}=0.
\end{equation}
It is \emph{discontinuous} if at least one of these two conditions
fails. Here the densities are relative to the normalized measure
$\lambda$. We keep this definition separate from the additional
inequality $\beta_c<\beta_*$, which is what guarantees a nonempty
interval of metastability below the linear-stability threshold.

We record the consequences of phase transition in metastability and propagation-of-chaos in the following proposition.

\begin{proposition}[Transition points, metastability and propagation-of-chaos in attraction basins]
\label{prop:app-coexistence}
Suppose \eqref{eq:app-fourier-potential} holds. At every parameter
value to which the conclusions below are applied, assume that
$\Crit(\F)/\T$ is finite, for example by
Proposition~\ref{prop:app-finite}.
Write
\[
 \mathcal A_\lambda
 :=\{[\mu]\in\Qq:d_{\Qq}([S_t\mu],[\lambda])\longrightarrow0
                       \text{ as }t\to\infty\}.
\]
For the stochastic conclusions, impose \eqref{eq:poc-initial}
with $[\mu_0]\in\mathcal A_\lambda$.
All statements are for a fixed $\beta$ in the given range.
\begin{enumerate}[label=\textup{(\roman*)}]
\item \emph{Continuous transition.}
If the transition is continuous, then $\beta_c=\beta_*$.
For every $\beta\in(0,\beta_c]$, $\lambda$ is the unique global
minimizer and its orbit is asymptotically stable in the Wasserstein topology. In particular,
its basin of attraction contains a neighborhood of $[\lambda]$; stability,
uniform attraction on compact subsets of the basin, and the
existence of arbitrarily small stable gates follow from
Lemma~\ref{lem:poc-minimum-stability}.
Theorem~\ref{thm:poc-global} and
Theorem~\ref{thm:poc-running-averages}\textup{(iii)} give
\begin{equation}\label{eq:app-transition-global-poc}
 \sup_{t\ge0}\mathbb E\,d_{\Qq}([\mu_t^N],[\mu_t])^2\longrightarrow0,
 \qquad
 \mathbb E\left[\sup_{T>0}\frac1T\int_0^Td_{\Qq}([\mu_t^N],[\mu_t])^2\,\dd t\right]\longrightarrow0.
\end{equation}
Furthermore, Theorem~\ref{thm:poc-pathwise} gives
\begin{equation}\label{eq:app-transition-pathwise-poc}
 \mathbb E\left[\sup_{0\le t\le T_N}d_{\Qq}([\mu_t^N],[\mu_t])^2\right]
       \longrightarrow0
 \quad\text{whenever}\quad \log(1+T_N)=o(N).
\end{equation}
Corollary~\ref{cor:poc-averaged-measures} supplies the corresponding
uniform bounds for the aligned time-averaged empirical measures
and the occupation measures on $\Qq$.

\item \emph{Discontinuous transition.}
Suppose $\beta_c < \beta_*$. For every $\beta\in(\beta_c,\beta_*)$, $\lambda$ is a strict
nonglobal local minimum and at least one nonuniform
global-minimum orbit exists. Every local-minimum orbit is
asymptotically stable by Lemma~\ref{lem:poc-minimum-stability}.
Thus, by Theorem~\ref{thm:poc-pathwise}, \eqref{eq:app-transition-pathwise-poc} holds for initial conditions converging, modulo rotations, to a point in $\mathcal A_\lambda$; the same conclusion holds in the basin
of any other local minimum. To quantify metastability of the uniform phase, let $i$ index
the uniform well, so that $\mathcal M_i=[\lambda]$ and
$f_i=\F(\lambda)=0$, and set
\[
 h_i:=\min_{j:f_j<0}\Phi_{ij},\qquad
 D_i:=\beta h_i>0,\qquad
 \mathcal L_i:=\{\mathcal M_j:f_j<0\}.
\]
Assume additionally that the local-minimum energies are pairwise
distinct, or, more generally, that \eqref{eq:pairrelax} holds
for this $i$. Theorem~\ref{prop:verified-MT} then gives a sharp
metastability threshold for the uniform phase.
For every $r>0$ and $0<\varepsilon<D_i$, there exist
$\eta,c,C>0$ such that, uniformly over
$\nu\in\PN(\T)$ with $d_{\Qq}([\nu],[\lambda])\le\eta$,
\begin{align}
 &\sup_{0\le t\le e^{N(D_i-\varepsilon)}}
 \mathbb P_\nu\left(
   d_{\Qq}([\mu_t^N],[\lambda])>r
 \right)\le C e^{-cN},
 \label{eq:app-transition-localization-before}\\
 &\sup_{t\ge e^{N(D_i+\varepsilon)}}
 \mathbb P_\nu\left(
   \dist_{\Qq}([\mu_t^N],\mathcal L_i)>r
 \right)\le C e^{-cN},
 \label{eq:app-transition-localization-after}
\end{align}
for all sufficiently large $N$. The second estimate holds
for every $\varepsilon>0$, with possibly different constants.

Under the same assumptions, Theorem~\ref{thm:poc-metastable}
gives
\begin{align}
 &\sup_{0\le t\le e^{N(D_i-\varepsilon)}}
       \mathbb E\,d_{\Qq}([\mu_t^N],[\mu_t])^2\longrightarrow0,
       &&0<\varepsilon<D_i,
       \label{eq:app-transition-poc-before}\\
 &\sup_{t\ge e^{N(D_i+\varepsilon)}}
       \mathbb P\bigl(d_{\Qq}([\mu_t^N],[\mu_t]) \le\delta_i\bigr)
       \longrightarrow0,
       &&\varepsilon>0,
       \label{eq:app-transition-poc-after}
\end{align}
where
$\delta_i=\frac14\dist_{\Qq}
(\mathcal M_i,\{\mathcal M_j:f_j<f_i\})>0$. More generally, Remark~\ref{rem:poc-distinctness} still gives metastability and propagation-of-chaos up to
the smaller depth $D_i^{\le}$ and failure after depth $D_i$.
Theorem~\ref{thm:poc-running-averages}\textup{(i)--(ii)} give
the corresponding convergence and failure statements for the time average
$\frac1T\int_0^Td_{\Qq}([\mu_t^N],[\mu_t])^2\,\dd t$, and Corollary~\ref{cor:poc-averaged-measures} transfers
the convergence bound to the averaged measures.
\end{enumerate}
\end{proposition}

\begin{remark}[Checkable conditions for transition]
Checkable conditions for continuous and discontinuous phase transitions are demonstrated in \cite[Proposition~5.8, Theorem~5.11 and Theorem~5.19]{CarrilloEtAl2020} and \cite[Theorem~2.6 and Corollary~2.1]{BalasubramanianBanerjeeRigollet2025}. A particularly simple checkable sufficient condition for a discontinuous phase transition is the
subcritical-bifurcation criterion of \cite[Corollary~2.1(i)]{BalasubramanianBanerjeeRigollet2025}:
\begin{equation}\label{eq:app-discontinuous-fourier}
 a_1>0,\qquad a_k<a_1\quad(k\ge2),\qquad
 \frac{a_1}{2}<a_2<a_1.
\end{equation}
Their coupling parameter $\kappa$ is $\beta$ here, and their potential is $-W$. Their criterion gives a
discontinuous transition with $\beta_c<2/a_1$.
\end{remark}

The endpoint $\beta_*$ cannot in general be included in
part~\textup{(ii)}. For example, take
\[
 W(x)=-\cos x-\frac{11}{20}\cos(2x),\qquad \beta_*=2.
\]
For sufficiently small $\varepsilon$, the probability densities
\[
 \rho_\varepsilon(x)
 :=1+\varepsilon\cos x+\frac59\varepsilon^2\cos(2x)
\]
are positive and converge uniformly to $1$, but expansion of
the entropy in \eqref{eq:app-scaled-energy} gives
\begin{equation}\label{eq:app-spinodal-counterexample}
 \mathcal E_2(\rho_\varepsilon\lambda)
 =-\frac{\varepsilon^4}{288}+O(\varepsilon^6)<0
 \qquad (0<|\varepsilon|\text{ sufficiently small}).
\end{equation}
Thus $\lambda$ is not a local minimum at $\beta_*$.
This example satisfies Proposition~\ref{prop:app-finite} even
at $\beta_*$, as the bounds below show. Inclusion of this endpoint
in part~\textup{(ii)} therefore requires a separate local-minimality
assumption at $\beta_*$.

\begin{corollary}[An explicit family with finitely many critical orbits and multiple wells]
\label{cor:app-explicit-family}
Suppose \eqref{eq:app-fourier-potential} holds and
\begin{equation}\label{eq:app-explicit-coefficients}
 \frac12<\frac{a_2}{a_1}\le\frac{11}{20},
 \qquad
 \sum_{k\ge3}|a_k|\le\frac{a_1}{1000}.
\end{equation}
Then $\beta_*=2/a_1$ and there is a discontinuous transition
at $\beta_c<\beta_*$. There are finitely many critical rotation orbits
for every $0<\beta\le\beta_*$. For
$\beta\in(\beta_c,\beta_*)$, they include the strict nonglobal
minimum $\lambda$ and at least one nonuniform global-minimum orbit. The subexponential pathwise conclusion of
Proposition~\ref{prop:app-coexistence}\textup{(ii)} therefore applies,
as do the bounds with the two depths $D_i^{\le}$ and $D_i$ in
Remark~\ref{rem:poc-distinctness}. The sharp common threshold $D_i$
and its running-average version follow if the uniform minimum
satisfies \eqref{eq:pairrelax}, in particular if the local-minimum
energies are pairwise distinct.

\end{corollary}

One concrete potential in this family is
\begin{equation}\label{eq:app-explicit-potential}
 W(x)=-\cos x-\frac{11}{20}\cos(2x)
       -\frac1{250}\sum_{k\ge3}2^{-k}\cos(kx).
\end{equation}
Its tail has total coefficient mass $1/1000$ and decays
exponentially. Setting the tail to zero gives the two-mode
example above. 

\subsubsection{Noisy transformer interactions and consequences for propagation-of-chaos}
\label{subsec:app-transformer}

The noisy mean-field transformer model is motivated by
continuous-depth descriptions of self-attention, in which token
representations evolve through interactions weighted by their
similarity; see
\cite[Section~1.2.2]{BalasubramanianBanerjeeRigollet2025}.
The dynamics are given as follows:
\[
 \dd\theta_i
 =-\frac1N\sum_{j=1}^N
   e^{\alpha\cos(\theta_i-\theta_j)}
   \sin(\theta_i-\theta_j)\,\dd t
   +\sqrt{2\beta^{-1}}\,\dd B_i,
 \qquad 1\le i\le N,
\]
where the $\theta_i$ represent angular coordinates of the tokens on the circle.
The drift promotes alignment, while the noise disperses the
representations. Here $\alpha>0$ controls the dependence of the
interaction weights on similarity, and $\beta>0$ is the inverse
noise strength.

This system is precisely \eqref{eq:particle-system} with
\begin{equation}\label{eq:app-transformer-potential}
 W_\alpha(x)=-\frac{e^{\alpha\cos x}-I_0(\alpha)}{\alpha},
 \qquad
 a_k(\alpha)=\frac{2I_k(\alpha)}{\alpha},
\end{equation}
where $I_k$ is the modified Bessel function of the first kind.
The additive constant makes $W_\alpha$ mean zero and does not
affect the dynamics or minimizers. For fixed $\alpha$, the
Fourier coefficients decay faster than exponentially.

By \cite[Theorem~2.7]{BalasubramanianBanerjeeRigollet2025}, there exist $0 < \alpha_- < \alpha_+$ such that the associated phase transition is continuous for $\alpha \le \alpha_-$ and discontinuous for $\alpha> \alpha_+$. Moreover, $\alpha_+$ admits the explicit choice:
\[
 \alpha_+ := \inf\left\lbrace \alpha>0: \frac{I_2(\alpha)}{I_1(\alpha)}>\frac12\right\rbrace.
\]
Proposition~\ref{prop:app-coexistence} then provides conditional
metastability and propagation-of-chaos conclusions, provided
finiteness of the critical orbits is verified separately.
In the discontinuous-transition regime, a sharp common threshold
for the uniform well additionally requires \eqref{eq:pairrelax};
without it, the conclusions use the two depths
$D_i^{\le}$ and $D_i$.

The rest of the paper is organized as follows.
Sections~\ref{sec:barriers}--\ref{sec:deterministic} develop the
geometric, large-deviation, and deterministic ingredients of the
analysis.
Sections~\ref{sec:cycle-hierarchy}--\ref{sec:two-sided-cycles}
construct the cycle hierarchy and establish transition and
cycle-exit estimates, while Section~\ref{sec:endpoint-lemmas}
proves localization at deterministic times.
These estimates are used in Sections~\ref{sec:verify-assumption-one}
and~\ref{sec:metastable-poc} to prove the metastability and
propagation-of-chaos results, respectively.
Section~\ref{sec:applicationsproof} contains the proofs of the
application results from Section~\ref{sec:applications}.

\section{Communication heights and attracting gates}
\label{sec:barriers}

For nonempty sets \(A,B\subset\Qq\), we extend the definition \eqref{eq:pairwise-communication-height} of the free-energy
communication height as
\begin{equation}
\label{eq:communication-height}
    \Hh(A,B)
    :=
    \inf_{\substack{\gamma\in C([0,1];\Qq)\\
                    \gamma(0)\in A,\ \gamma(1)\in B}}
       \sup_{s\in[0,1]}\F(\gamma(s)).
\end{equation}
By abusing notation, we will subsequently write $x$ for singleton sets $\{x\}$ in $\Qq$. We will also write $A\Subset B$, for subsets $A,B$ of $\Qq$, when the closure $\overline A$ is compact
and contained in $B$. In the sequel, we will often say `deterministic trajectory/flow from $\nu$' to denote the solution path of \eqref{eq:MV} initiated at $\nu$.

We now introduce a terminology that will be used several times in the sequel.

\begin{definition}[Stable gate]
\label{def:stable-gate}
For a critical orbit \(\mathcal K\), any open neighborhood in the associated rotation-quotient Wasserstein-2 topology of $\Qq$ will be referred to as a \emph{gate}. Further, suppose \(\mathcal K\) is a stable critical point (equivalently, a local minimum of the free-energy), and let
\(U\subset\Qq\) be an open neighborhood contained in the
deterministic basin of attraction of \(\mathcal K\).
A \emph{stable gate} for \(\mathcal K\) inside \(U\) is an open set \(G\subset\Qq\) such that
\[
    \mathcal K\Subset G\Subset U,
    \qquad
    \partial G\cap \Crit(\F)/\T=\varnothing,
\]
and such that
\[
    S_t(\overline G)\subset U
    \qquad\text{for every }t\geq0,
\]
where \(S_t\) is the deterministic McKean--Vlasov semiflow.
Moreover, the attraction to \(\mathcal K\) is uniform over
\(\overline G\): for every neighborhood \(O\) of
\(\mathcal K\), there exists \(T_O<\infty\) such that
\[
    S_t(\overline G)\subset O
    \qquad\text{for every }t\geq T_O.
\]
\end{definition}

The existence of arbitrarily small stable gates around any local minimum orbit is proved in
Lemma~\ref{lem:poc-minimum-stability}.

\begin{remark}[A stable gate may replace its minimum]
\label{rem:gate-height}
If $G$ is a stable gate for $\mathcal M_j$, then
$\Hh(\mathcal M_i,G)=\Hh(\mathcal M_i,\mathcal M_j)$.
One inequality follows from $\mathcal M_j\subset G$.  For the reverse
inequality, continue a path ending in $G$ along the deterministic flow
to $\mathcal M_j$.  This continuation does not increase free-energy;
its infinite tail has a continuous reparametrization on a compact interval (see the proof of Proposition~\ref{prop:local-quasipotential}).
\end{remark}

For \(A\subset\Qq\), define
\begin{equation}
\label{eq:hitting-times}
  \tau_A^N
    :=
    \inf\{t\geq0:[\mu_t^N] \in A\}.
\end{equation}

\section{Finite-horizon large deviations}
\label{sec:LDP}

In this section, we record the path-space large deviations results of \cite{DawsonGartner1987} and its equivalent control-theoretic representation (see \cite{BudhirajaDupuisFischer2012}). Proposition~\ref{prop:uniform-LDP} records some uniform consequences of the LDP that will be crucially used in our proofs.

\subsection{The weighted negative Sobolev norm and the action}
 Let
\[
    \mathcal D'(\T):=(C^\infty(\T))'
\]
be the space of distributions on the circle. 
For a signed distribution \(s\in\mathcal D'(\T)\) satisfying
\(\langle s,1\rangle=0\), and for \(\mu\in\Pp(\T)\), define the
\emph{unscaled weighted negative Sobolev norm} by
\begin{equation}
\label{eq:Hminus1}
    \norm{s}_{\dot H^{-1}(\mu)}^2
    :=
    \sup_{\varphi\in C^\infty(\T)}
    \left\{
        2\ip{s}{\varphi}
        -
        \int_\T|\varphi'(x)|^2\,\mu(\dd x)
    \right\} =
\sup_{\substack{\varphi\in C^\infty(\mathbb T)\\
                 \int_{\mathbb T}|\varphi'|^2\,d\mu>0}}
\frac{|\langle s,\varphi\rangle|^2}
     {\displaystyle\int_{\mathbb T}|\varphi'|^2\,d\mu}.
\end{equation}
The value is understood to be \(+\infty\) if the supremum is
unbounded.
This norm has an equivalent minimal-flux representation.  If
\(
    \norm{s}_{\dot H^{-1}(\mu)}<\infty,
\)
then there exists a unique minimal velocity field
\[
    v_s\in
    \operatorname{Tan}_{\mu}\Pp_2(\T)
    :=
    \overline{
        \left\{
            \partial_x\varphi:
            \varphi\in C^\infty(\T)
        \right\}
    }^{\,L^2(\mu)}
\]
such that
\begin{equation}
\label{eq:static-flux-representation}
    s=-\partial_x(\mu v_s)
    \qquad\text{in }\mathcal D'(\T).
\end{equation}
Equivalently,
\[
    \langle s,\varphi\rangle
    =
    \int_\T v_s\,\varphi'\,\dd\mu
    \qquad
    \text{for every }\varphi\in C^\infty(\T).
\]
Moreover,
\begin{equation}
\label{eq:weighted-Hminus1-flux}
\begin{aligned}
    \norm{s}_{\dot H^{-1}(\mu)}^2
    =
    \int_\T|v_s|^2\,\dd\mu                                      
    =
    \min_{\substack{v\in L^2(\mu)\\
                    -\partial_x(\mu v)=s}}
        \int_\T|v|^2\,\dd\mu.
\end{aligned}
\end{equation}
Indeed, the map $\varphi' \rightarrow \ip{s}{\varphi}$ in \eqref{eq:Hminus1} extends to a bounded linear functional on
the \(L^2(\mu)\)-closure of the gradients, and
\eqref{eq:static-flux-representation}--\eqref{eq:weighted-Hminus1-flux}
follow from the Riesz representation theorem.

Thus \(\norm{s}_{\dot H^{-1}(\mu)}^2\) is the least kinetic energy of
a velocity field whose \(\mu\)-weighted flux has divergence \(-s\).
In particular, if
\[
    s=\partial_t\mu_t
\]
for a Wasserstein absolutely continuous curve, then
\[
    \partial_t\mu_t+\partial_x(\mu_tv_t)=0
\]
is the continuity equation and the minimal field \(v_t\) is the
Wasserstein tangent velocity.  See
\cite[Theorem~8.3.1 and the discussion around (8.0.2)]{AGS2008}
and \cite[Theorem~5.14]{Santambrogio2015}.

For \(\mu\in\Pp(\T)\), define
\begin{equation}
\label{eq:G-def}
    \Gg(\mu)
    :=
    \beta^{-1}\partial_{xx}\mu
    +
    \partial_x\bigl(\mu(W'*\mu)\bigr)
    \in\mathcal D'(\T).
\end{equation}
Equivalently,
\begin{equation}
\label{eq:G-def-weak}
    \langle\Gg(\mu),\varphi\rangle
    =
    \beta^{-1}\int_\T\varphi''\,\dd\mu
    -
    \int_\T\varphi'(W'*\mu)\,\dd\mu,
    \qquad
    \varphi\in C^\infty(\T).
\end{equation}
If \(\mu=\rho\lambda\), with \(\rho\) smooth and strictly positive,
then
\[
    \Gg(\mu)
    =
    \partial_x\left[
        \rho\,\partial_x
        \bigl(
            \beta^{-1}\log\rho+W*\mu
        \bigr)
    \right]\lambda.
\]
Thus the McKean--Vlasov equation \eqref{eq:MV} is
\[
    \partial_t\mu_t=\Gg(\mu_t).
\]

For finite $T>0$, set
\[
    \mathcal X_T:=C([0,T];\Pp(\T)),
\]
equipped with the uniform Wasserstein topology.  Following
\cite[Definition~4.1]{DawsonGartner1987}, we say that
\(\mu\in\mathcal X_T\) is \emph{absolutely continuous as a
distribution-valued path} if there exist a neighborhood \(U\) of
\(0\) in \(C^\infty(\T)\) and an absolutely continuous function
\(H:[0,T]\to\Rr\) such that
\begin{equation}
\label{eq:distributional-absolute-continuity}
    \left|
        \langle\mu_t-\mu_s,\varphi\rangle
    \right|
    \leq
    |H(t)-H(s)|
\end{equation}
for every \(s,t\in[0,T]\) and every \(\varphi\in U\).

For such a path, the derivative
\[
    \partial_t\mu_t
    :=
    \lim_{h\to0}
        \frac{\mu_{t+h}-\mu_t}{h}
\]
exists in \(\mathcal D'(\T)\) for almost every \(t\); see
\cite[Lemma~4.2]{DawsonGartner1987}.  This almost-everywhere
derivative represents the time derivative of \(\mu\) in the
space--time distributional sense.  Since
\(\langle\mu_t,1\rangle=1\), it satisfies
\(
    \langle\partial_t\mu_t,1\rangle=0
\)
for almost every \(t\).

Define the action by
\begin{equation}
\label{eq:action}
    I_{[0,T]}(\mu)
    :=
    \begin{cases}
    \displaystyle
    \frac{\beta}{4}
    \int_0^T
        \left\|
            \partial_t\mu_t-\Gg(\mu_t)
        \right\|_{\dot H^{-1}(\mu_t)}^2
    \,\dd t,
    &\text{if \(\mu\) is distributionally absolutely continuous,}
    \\[4mm]
    +\infty,
    &\text{otherwise.}
    \end{cases}
\end{equation}
It can be checked that  \eqref{eq:action} is precisely the specialization of
\cite[equations~(4.7)--(4.9) and Section~5.2]
{DawsonGartner1987} to the present model.

The action also admits an equivalent control representation.  For
\(\mu\in\mathcal X_T\), define
$$
L^2(\dd t\,\mu_t) := \left\lbrace u:[0,T]\times\T\longrightarrow\Rr \ : \  \int_0^T\int_\T
        |u_t(x)|^2\,\mu_t(\dd x)\,\dd t
    <\infty\right\rbrace.
$$
Let \(\mathcal U_T(\mu)\) denote the
collection of $u \in L^2(\dd t\,\mu_t)$
such that
\begin{equation}
\label{eq:controlled-MV-distributional}
    \partial_t\mu_t
    =
    \Gg(\mu_t)-\partial_x(\mu_tu_t)
\end{equation}
in space--time distributions.  Equivalently, for every
\(\varphi\in C^\infty(\T)\) and \(t\in[0,T]\),
\begin{align}
\label{eq:controlled-MV-weak}
    \int_\T\varphi\,\dd\mu_t
    -
    \int_\T\varphi\,\dd\mu_0
    =
    \int_0^t
    \Bigg[
        \beta^{-1}\int_\T\varphi''\,\dd\mu_s
        -
        \int_\T
            \varphi'(W'*\mu_s)\,\dd\mu_s
        +
        \int_\T\varphi'u_s\,\dd\mu_s
    \Bigg]\dd s.
\end{align}

The minimal-flux characterization
\eqref{eq:weighted-Hminus1-flux}, applied in the space--time Hilbert
space \(L^2(\dd t\,\mu_t)\), gives
\begin{equation}
\label{eq:action-control-representation}
    I_{[0,T]}(\mu)
    =
    \frac{\beta}{4}
    \inf_{u\in\mathcal U_T(\mu)}
        \int_0^T\int_\T
            |u_t(x)|^2\,\mu_t(\dd x)\,\dd t,
\end{equation}
where the infimum over the empty set is \(+\infty\).  Whenever the
action is finite, then again using the Riesz representation theorem, the infimum is attained by a unique minimal control (as an element of $L^2(\dd t\,\mu_t)$)
in
\[
    \overline{
        \left\{
            \partial_x\varphi:
            \varphi\in C^\infty([0,T]\times\T)
        \right\}
    }^{\,L^2(\dd t\,\mu_t)}.
\]
In particular,
\[
    \partial_t\mu_t-\Gg(\mu_t)
    =
    -\partial_x(\mu_tu_t)
\]
for almost every \(t\), and
\[
    \left\|
        \partial_t\mu_t-\Gg(\mu_t)
    \right\|_{\dot H^{-1}(\mu_t)}^2
    =
    \int_\T|u_t|^2\,\dd\mu_t
\]
for the minimal control $u$.


The mean-field drift of the associated McKean--Vlasov diffusion is
\[
    b(x,\mu):=-(W'*\mu)(x).
\]
Since \(W''\) is bounded,
\begin{equation}
\label{eq:drift-Lipschitz}
    |b(x,\mu)-b(y,\nu)|
    \leq
    \norm{W''}_\infty
    \bigl(
        d_\T(x,y)+\mathsf W_1(\mu,\nu)
    \bigr).
\end{equation}
Thus the coefficients are globally bounded and Lipschitz on the
compact state space. The hypotheses of the weakly interacting
diffusion LDP in
\cite[Theorem~3.1]{BudhirajaDupuisFischer2012}
are therefore satisfied. These regularity properties allow us to invoke the finite-horizon empirical-process LDP. We state its sequential form below and derive the uniform consequences in the proof; see also
\cite[Theorem~5.9]{GvalaniSchlichting2020}.
\begin{proposition}[Path-space LDP and uniform consequences]
\label{prop:uniform-LDP}
Fix \(T<\infty\), and equip
\[
    \mathcal X_T:=C([0,T];\Pp(\T))
\]
with the uniform Wasserstein topology.  For \(\nu\in\Pp(\T)\), define
\[
    I_{\nu,T}(\mu)
    :=
    \begin{cases}
       I_{[0,T]}(\mu),&\mu_0=\nu,\\
       +\infty,&\mu_0\neq\nu.
    \end{cases}
\]
If \(\nu_N\in\PN(\T)\) and \(\nu_N\to\nu\), then for every closed
\(F\subset\mathcal X_T\) and open \(O\subset\mathcal X_T\),
\begin{align}
\label{eq:sequential-LDP-upper}
    \limsup_{N\to\infty}\frac1N
    \log\mathbb P_{\nu_N}(\mu^N_\cdot\in F)
    &\leq
    -\inf_{\mu\in F}I_{\nu,T}(\mu),
    \\
\label{eq:sequential-LDP-lower}
    \liminf_{N\to\infty}\frac1N
    \log\mathbb P_{\nu_N}(\mu^N_\cdot\in O)
    &\geq
    -\inf_{\mu\in O}I_{\nu,T}(\mu).
\end{align}
For every $A<\infty$, the action sublevel
\(
 \{\mu\in\mathcal X_T:I_{[0,T]}(\mu)\le A\}
\)
is compact.
Moreover, the large deviation principle is uniform in the following sense:
\begin{enumerate}[label=\textup{(\roman*)}]
\item
For every closed \(F\subset\mathcal X_T\),
\begin{equation}
\label{eq:global-uniform-upper}
    \limsup_{N\to\infty}\frac1N
    \log
    \sup_{\eta\in\PN(\T)}
    \mathbb P_\eta(\mu^N_\cdot\in F)
    \leq
    -\inf_{\mu\in F}I_{[0,T]}(\mu).
\end{equation}
\item
Let \(\bar\mu\in\mathcal X_T\) have
\(\bar\mu_0=\nu\) and finite action, and let \(O\) be an open
neighborhood of \(\bar\mu\).  For every \(\varepsilon>0\), there are
a neighborhood \(V\) of \(\nu\) and \(N_0<\infty\) such that
\begin{equation}
\label{eq:local-uniform-lower}
    \inf_{\eta\in V\cap\PN(\T)}
    \mathbb P_\eta(\mu^N_\cdot\in O)
    \geq
    \exp\left\{
       -N\bigl(I_{[0,T]}(\bar\mu)+\varepsilon\bigr)
    \right\}
\end{equation}
for every \(N\geq N_0\).
\end{enumerate}
\end{proposition}

\begin{proof}
The sequential bounds
\eqref{eq:sequential-LDP-upper}--\eqref{eq:sequential-LDP-lower}
follow from the Dawson--G\"artner path-space LDP
\cite{DawsonGartner1987}, with the distributional action
\eqref{eq:action}. An alternative control formulation is given
in \cite[Theorem~3.1]{BudhirajaDupuisFischer2012}. 
The compactness assertion follows from
\cite[Theorem~5.1(c)]{DawsonGartner1987},
specialized to the torus, since $\Pp(\T)$ is compact.

We prove the two uniform consequences.

\smallskip
\noindent
\emph{Proof of \textup{(i)}.}
Suppose that \eqref{eq:global-uniform-upper} fails.  Then there
exist a closed set \(F\), a number \(\delta>0\), a subsequence
\(N_k\to\infty\), and initial empirical measures
\(\eta_{N_k}\in\mathcal P_{N_k}(\T)\) such that
\begin{equation}
\label{eq:uniform-upper-contradiction}
    \frac1{N_k}
    \log
    \mathbb P_{\eta_{N_k}}
    \bigl(\mu^{N_k}_{\cdot}\in F\bigr)
    >
    -\inf_{\rho\in F}I_{[0,T]}(\rho)+\delta.
\end{equation}

Since \(\Pp(\T)\) is compact, after passing to a further
subsequence there exists \(\eta\in\Pp(\T)\) such that
\(
    \eta_{N_k}\longrightarrow\eta.
\)
Applying the sequential upper bound
\eqref{eq:sequential-LDP-upper} along this subsequence gives
\[
\begin{aligned}
    \limsup_{k\to\infty}
    \frac1{N_k}
    \log
    \mathbb P_{\eta_{N_k}}
    \bigl(\mu^{N_k}_{\cdot}\in F\bigr)
    \le
    -\inf_{\rho\in F}
    I_{\eta,T}(\rho)
    \le
    -\inf_{\rho\in F}
    I_{[0,T]}(\rho),
\end{aligned}
\]
contradicting \eqref{eq:uniform-upper-contradiction}.  This
proves \eqref{eq:global-uniform-upper}.

\smallskip
\noindent
\emph{Proof of \textup{(ii)}.}
Fix \(\varepsilon>0\).  Suppose that no neighborhood \(V\) and
integer \(N_0\) satisfy \eqref{eq:local-uniform-lower}.  For
each \(k\ge1\), we may then choose \(N_k\ge k\) and
\(\eta_{N_k}\in\mathcal P_{N_k}(\T)\) such that
\(
    \Wass(\eta_{N_k},\nu)<\frac1k
\)
and
\[
    \mathbb P_{\eta_{N_k}}
    \bigl(\mu^{N_k}_{\cdot}\in O\bigr)
    <
    \exp\left\{
        -N_k\bigl(I_{[0,T]}(\bar\mu)+\varepsilon\bigr)
    \right\}.
\]
Thus \(\eta_{N_k}\to\nu\).  The sequential lower bound
\eqref{eq:sequential-LDP-lower} yields
\[
\begin{aligned}
    \liminf_{k\to\infty}
    \frac1{N_k}
    \log
    \mathbb P_{\eta_{N_k}}
    \bigl(\mu^{N_k}_{\cdot}\in O\bigr)
    &\ge
    -\inf_{\rho\in O}I_{\nu,T}(\rho)
    \\
    &\ge
    -I_{[0,T]}(\bar\mu),
\end{aligned}
\]
because \(\bar\mu\in O\) and \(\bar\mu_0=\nu\). This
contradicts the preceding strict upper bound and proves
\eqref{eq:local-uniform-lower}.
\end{proof}

\begin{corollary}[Uniform lower bound from robust controlled paths]
\label{cor:finite-cover-lower}
Let \(K\subset\Pp(\T)\) be compact, and let
\(E\subset\mathcal X_T\) be an event.  Suppose that there is
\(B<\infty\) such that, for every \(\nu\in K\), one can find a
path \(\rho^\nu\in\mathcal X_T\) and a radius \(r_\nu>0\)
satisfying
\[
    \rho^\nu_0=\nu,
    \qquad
    I_{[0,T]}(\rho^\nu)\le B,
\]
and
\[
    B_{\mathcal X_T}(\rho^\nu,r_\nu)\subset E.
\]
Then, for every \(\varepsilon>0\), there exists \(N_0<\infty\)
such that
\[
    \inf_{\eta\in K\cap\PN(\T)}
    \mathbb P_\eta
    \bigl(\mu^N_{\cdot}\in E\bigr)
    \ge
    e^{-N(B+\varepsilon)}
\]
for every \(N\ge N_0\).
\end{corollary}

\begin{proof}
For each \(\nu\in K\), apply
Proposition~\ref{prop:uniform-LDP}\textup{(ii)} to the open
tube
\[
    O_\nu
    :=
    B_{\mathcal X_T}(\rho^\nu,r_\nu)
\]
with error \(\varepsilon\).  This gives a neighborhood
\(V_\nu\) of \(\nu\) and an integer \(N_\nu\) such that
\[
    \inf_{\eta\in V_\nu\cap\PN(\T)}
    \mathbb P_\eta
    \bigl(\mu^N_{\cdot}\in O_\nu\bigr)
    \ge
    e^{-N(B+\varepsilon)}
\]
for all \(N\ge N_\nu\).

The sets \((V_\nu)_{\nu\in K}\) cover \(K\).  By compactness,
there exist \(\nu_1,\ldots,\nu_m\) such that
\[
    K\subset\bigcup_{j=1}^mV_{\nu_j}.
\]
Taking \(N_0:=\max_{1\le j\le m}N_{\nu_j}\), and using
\(O_{\nu_j}\subset E\), proves the claim.
\end{proof}

\paragraph{Action on the rotation quotient.}
Let $p:\Pp(\T)\to\Qq$ be the quotient map. For
$q\in C([0,T];\Qq)$, define the contraction action
\begin{equation}\label{eq:quotient-action}
 \overline I_{[0,T]}(q)
 :=\inf\bigl\{I_{[0,T]}(\mu):
       \mu\in\mathcal X_T,\ p(\mu_t)=q_t\ (0\le t\le T)\bigr\},
\end{equation}
with the infimum of the empty set equal to $+\infty$.
The map $\mu_\cdot\mapsto p(\mu_\cdot)$ is continuous in the
uniform path topology. The contraction principle therefore gives
the corresponding quotient LDP and the uniform consequences
above. By rotating the entire lifted path by a fixed angle, we can
make it start at any chosen representative of the initial orbit,
without changing its quotient path or its action.

Compactness of the measure-valued action sublevel sets shows that
the infimum in \eqref{eq:quotient-action} is attained whenever
it is finite, and that the quotient action sublevel sets are compact.
From now on we suppress the bar and write $I(q)$ for this
contracted action. Whenever a control, distributional derivative
or PDE calculation is written for a quotient path, it is
performed on an action-minimizing measure-valued lift, with the
same path symbol used for that lift during the calculation.
The resulting path is then projected onto $\Qq$.
This convention does not identify the actions of arbitrary
time-dependent rotations of a lift.

The quotient action is additive under concatenation. Indeed,
any lift of a concatenated path restricts to lifts of its two
pieces. Its action is therefore at least the sum of the quotient
actions of those pieces. Conversely, when both piece actions
are finite, choose an action-minimizing lift of each piece.
Rotate the second lift by a fixed angle so that its initial
measure coincides with the terminal measure of the first lift.
Concatenating these lifts gives a lift of the full path whose
action equals the sum of the two piecewise actions, proving the reverse inequality.

Free-energy is independent of the lift. Consequently, the
action--energy lower bound and the time-reversal identity given in Lemma~\ref{lem:action-energy} below
also hold on $\Qq$, by taking infima over lifts.

\section{The action--energy identity and quasipotentials}
\label{sec:action-energy}

For a smooth positive density \(\mu=\rho\lambda\), set
\begin{equation}
\label{eq:Psi}
    \Psi_\mu
    :=
    \beta^{-1}\log\rho+W*\mu.
\end{equation}

Let $\mu$ be a finite-action path on $[0,T]$ with $\F(\mu_0)<\infty$,
and let $u$ be its minimal control. The finite-initial-energy
condition upgrades distributional absolute continuity to the
regularity needed for the Wasserstein chain rule. More precisely,
for almost every $t$ one has $\mu_t=\rho_t\lambda$ and
\[
 \xi_t:=\beta^{-1}\partial_x\log\rho_t+W'*\mu_t
 \in L^2(\mu_t),\qquad
 \int_0^T\int_\T|\xi_t|^2\,\dd\mu_t\,\dd t
 \le 2\bigl(\F(\mu_0)-F_{\min}\bigr)
       +\int_0^T\int_\T|u_t|^2\,\dd\mu_t\,\dd t.
\]
Here $\rho_t|\partial_x\log\rho_t|^2$ is interpreted as
$4|\partial_x\sqrt{\rho_t}|^2$ on the domain of the Fisher
information. These assertions follow by a mollification scheme, lower semicontinuity of the free-energy and Fisher information, and Young's inequality.
Moreover, the interaction part has a bounded smooth derivative.
These facts imply that the velocity $v_t=u_t-\xi_t$ is square integrable and
satisfies $\partial_t\mu_t+\partial_x(\mu_tv_t)=0$.
In particular, $\mu\in AC^2([0,T];\Pp(\T))$, the space of absolutely continuous curves
with square-integrable metric speed; see \cite[Definition~1.1.1 and Theorem~1.1.2]{AGS2008}.

The Wasserstein chain rule now implies that $t\mapsto\F(\mu_t)$
is absolutely continuous, and
\begin{equation}\label{eq:free-energy-chain-rule}
 \F(\mu_T)-\F(\mu_0)
 =\int_0^T\int_\T \xi_t v_t\,\dd\mu_t\,\dd t
 =\int_0^T\langle\Psi_{\mu_t},\partial_t\mu_t\rangle\,\dd t.
\end{equation}
Here the pairing $\langle\Psi_{\mu_t},\partial_t\mu_t\rangle$
is defined as $\int_\T \xi_t v_t\,\dd\mu_t$.
For the chain rule, see
\cite[\S10.1.2(E) and Theorem~10.4.13]{AGS2008} and
\cite[proof of Theorem~5.10]{GvalaniSchlichting2020}.

\begin{lemma}[Energy dissipation]
\label{lem:energy-dissipation}
For every positive-time solution of \eqref{eq:MV},
\begin{equation}
\label{eq:dissipation}
    \frac{\dd}{\dd t}\F(\mu_t)
    =
    -\int_\T
       \abs{\partial_x\Psi_{\mu_t}}^2\,\mu_t(\dd x)
    \leq0.
\end{equation}
Equality holds at a time \(t>0\) if and only if \(\mu_t\) is
stationary.
\end{lemma}

\begin{proof}
Positive-time parabolic regularity gives a smooth strictly positive
density.  Evenness of \(W\), differentiation of
\eqref{eq:free-energy}, and integration by parts give
\[
    \frac{\dd}{\dd t}\F(\mu_t)
    =
    \ip{\Psi_{\mu_t}}{\partial_t\mu_t}
    =
    -\int_\T
       \abs{\partial_x\Psi_{\mu_t}}^2\,\dd\mu_t.
\]
The final integral vanishes exactly when
\(\partial_x\Psi_{\mu_t}=0\), which is the stationary
Euler--Lagrange equation.
\end{proof}

\begin{lemma}[Action--energy decomposition]
\label{lem:action-energy}
Let \(\mu\) be a finite-action path on \([0,T]\) with \(\mathcal F_\beta(\mu_0)<\infty\).  Then
\begin{equation}
\label{eq:action-energy-identity}
    I_{[0,T]}(\mu)
    =
    \frac{\beta}{4}\int_0^T
       \norm{\partial_t\mu_t+\Gg(\mu_t)}
             _{\dot H^{-1}(\mu_t)}^2\,\dd t
    +
    \beta\bigl[\F(\mu_T)-\F(\mu_0)\bigr].
\end{equation}
Consequently,
\begin{equation}
\label{eq:action-lower-energy}
    I_{[0,T]}(\mu)
    \geq
    \beta
    \max_{0\leq t\leq T}
       \bigl\{\F(\mu_t)-\F(\mu_0)\bigr\}.
\end{equation}
The reversal of a deterministic trajectory segment has action equal
to \(\beta\) times its free-energy increase.
\end{lemma}

\begin{proof}
The chain rule gives
\[
    \F(\mu_T)-\F(\mu_0)
    =
    \int_0^T\ip{\Psi_{\mu_t}}{\partial_t\mu_t}\,\dd t.
\]
For the Wasserstein free-energy chain rule, see
\cite[Section~10.1.2]{AGS2008}; its use for the McKean--Vlasov
action is also discussed in \cite{GvalaniSchlichting2020}.
In the weighted negative Sobolev pairing,
\[
    \ip{\partial_t\mu}{\Gg(\mu)}
       _{\dot H^{-1}(\mu)}
    =
    -\ip{\Psi_\mu}{\partial_t\mu}.
\]
Expanding the square in \eqref{eq:action} therefore gives
\eqref{eq:action-energy-identity}.  Apply the identity on
\([0,t_*]\), where \(t_*\) maximizes \(t\mapsto\F(\mu_t)\), to
obtain \eqref{eq:action-lower-energy}.  If a path is the reversal of
a deterministic segment, then
\(\partial_t\mu=-\Gg(\mu)\), so the square in
\eqref{eq:action-energy-identity} vanishes.
\end{proof}

\begin{definition}[Quasipotential]
\label{def:quasipotential}
For nonempty sets \(A,B\subset\Qq\), define
\begin{equation}
\label{eq:quasipotential}
    \Vq(A,B)
    :=
    \inf_{T>0}
   \inf_{\substack{\mu\in C([0,T];\Qq)\\
                    \mu(0)\in A,\ \mu(T)\in B}}
       I_{[0,T]}(\mu).
\end{equation}
\end{definition}

Lemma \ref{lem:action-energy} implies
\begin{equation}
\label{eq:V-lower}
    \Vq(\mathcal M,B)
    \geq
    \beta\bigl[
       \Hh(\mathcal M,B)-\F(\mathcal M)
    \bigr]
\end{equation}
when \(\mathcal M\) is a stationary orbit.

\begin{proposition}[Local reversible quasipotential in one basin]
\label{prop:local-quasipotential}
Let \(\mathcal M\) be an attracting local-minimum orbit, and let
\(U\) be an open neighborhood satisfying
\[
    \mathcal M\Subset U,
    \qquad
    \overline U
    \ \text{is contained in the deterministic basin of }\mathcal M.
\]
Then
\begin{equation}
\label{eq:local-quasipotential}
    \Vq(\mathcal M,U^c)
    =
    \beta\bigl[
       \Hh(\mathcal M,U^c)-\F(\mathcal M)
    \bigr].
\end{equation}
\end{proposition}

\begin{proof}
In view of \eqref{eq:V-lower}, it remains to prove 
\begin{equation}\label{eq:Vupper}
\mathsf V(\mathcal M,U^c)
    \le
    \beta\left[
        \mathsf H(\mathcal M,U^c)
        -
        \mathcal F_\beta(\mathcal M)
    \right].
    \end{equation} 

If $U=\Qq$, both sides are $+\infty$, with the convention that
an infimum over an empty target set is $+\infty$.
Suppose therefore that $U$ is proper, and set
\[
 m_U:=\min_{\nu\in\partial U}\F(\nu)\in\Rr\cup\{+\infty\}.
\]
The boundary is nonempty since $\Qq$ is connected and $U$ is nonempty,
open and proper. Compactness of $\partial U$ and lower
semicontinuity of \(\mathcal F_\beta\) imply that the minimum above is attained. 
If $m_U=+\infty$, \eqref{eq:V-lower} proves the result.
In the rest of the proof assume $m_U<\infty$.

We first identify the communication height.  Every continuous path from
\(\mathcal M\subset U\) to \(U^c\) meets \(\partial U\).  Consequently,
\[
    \mathsf H(\mathcal M,U^c)\geq m_U.
\]
For the reverse inequality, fix \(\nu\in\partial U\) such that
\(\mathcal F_\beta(\nu)=m_U\).  Since
\(\overline U\) is contained in the basin of \(\mathcal M\),
\[
    S_t\nu\longrightarrow\mathcal M
    \qquad\text{in }\mathcal Q
\]
as \(t\to\infty\).  The curve
\[
    \gamma_\nu(s)
    :=
    \begin{cases}
        \mathcal M,
        &s=0,\\[1mm]
        S_{(1-s)/s}\nu,
        &0<s\leq1,
    \end{cases}
\]
is therefore a continuous path in \(\mathcal Q\) from
\(\mathcal M\) to \(\nu\).  Since the deterministic flow does not
increase the free-energy,
\[
    \max_{s\in[0,1]}
        \mathcal F_\beta(\gamma_\nu(s))
    =
    \mathcal F_\beta(\nu)
    =
    m_U.
\]
As \(\nu\in\partial U\subset U^c\), this gives
\[
    \mathsf H(\mathcal M,U^c)\leq m_U.
\]
Hence
\begin{equation}
\label{eq:height-boundary-minimum}
    \mathsf H(\mathcal M,U^c)
    =
    \min_{\nu\in\partial U}\mathcal F_\beta(\nu).
\end{equation}
Now we prove \eqref{eq:Vupper}.  Fix again a minimizer
\(\nu\in\partial U\).  For \(T>0\), consider the reversal of the
deterministic trajectory segment from \(\nu\) to \(S_T\nu\):
\[
    \widehat\mu_t^T:=S_{T-t}\nu,
    \qquad 0\leq t\leq T.
\]
This path runs from \(S_T\nu\) to \(\nu\).  By Lemma \ref{lem:action-energy},
\begin{equation}
\label{eq:reversed-segment-action}
    I_{[0,T]}(\widehat\mu^T)
    =
    \beta\left[
        \mathcal F_\beta(\nu)
        -
        \mathcal F_\beta(S_T\nu)
    \right].
\end{equation}
Choose a smooth positive representative
\(\bar\mu=\bar\rho\,\lambda\in\mathcal M\).
Since \(S_T\nu\to\mathcal M\), modulo rotations, as $T \to \infty$, parabolic regularity
and compactness yield rotations \(R_{\theta_T}\) such that
\[
    R_{\theta_T\#}S_T\nu=\rho_T\lambda,
    \qquad
    \rho_T\longrightarrow\bar\rho
    \quad\text{in }C^1(\T).
\]
For \(T\) large, the affine interpolation
\[
    \rho_s^T=(1-s)\bar\rho+s\rho_T,
    \qquad 0\le s\le1,
\]
is uniformly bounded away from zero.  Since
\(\Gg(\bar\mu)=0\), the definition of the weighted
\(\dot H^{-1}\)-norm gives
\[
    I_{[0,1]}(\rho^T\lambda)
    \le C\|\rho_T-\bar\rho\|_{C^1(\T)}^2
    \longrightarrow0.
\]
After rotating this interpolation back by \(R_{-\theta_T}\), we obtain
a connector from some representative of \(\mathcal M\) to \(S_T\nu\)
whose action tends to zero. Consequently, there
are paths \(\alpha^T\) from \(\mathcal M\) to \(S_T\nu\) such that
\begin{equation}
\label{eq:small-connector}
    I(\alpha^T)\longrightarrow0
    \qquad\text{as }T\to\infty.
\end{equation}
This is the standard small-action connector near a smooth positive
stationary component.

Concatenating \(\alpha^T\) with \(\widehat\mu^T\) gives a path from
\(\mathcal M\) to \(\nu\in U^c\), and therefore
\[
\begin{aligned}
    \mathsf V(\mathcal M,U^c)
    &\leq
    I(\alpha^T)
    +
    \beta\left[
        \mathcal F_\beta(\nu)
        -
        \mathcal F_\beta(S_T\nu)
    \right].
\end{aligned}
\]
Since
\[
    I(\alpha^T)\longrightarrow0,
    \qquad
    \mathcal F_\beta(S_T\nu)
    \longrightarrow
    \mathcal F_\beta(\mathcal M),
\]
letting \(T\to\infty\) gives
\[
    \mathsf V(\mathcal M,U^c)
    \leq
    \beta\left[
        m_U-\mathcal F_\beta(\mathcal M)
    \right].
\]
Combining the two bounds and using
\eqref{eq:height-boundary-minimum} proves
\[
    \mathsf V(\mathcal M,U^c)
    =
    \beta\left[
        \mathsf H(\mathcal M,U^c)
        -
        \mathcal F_\beta(\mathcal M)
    \right].
\]
The small-connector and reversed-trajectory construction is the local
reversible quasipotential argument of \cite{DawsonGartner1989}.
\end{proof}

\begin{remark}[Role of reversibility]
\label{rem:false-global-qp}
For general nonreversible diffusions in the Freidlin--Wentzell
theory \cite{FreidlinWentzell2012}, the quasipotential is determined
by minimizing the dynamical action and generally has no simple
expression in terms of an energy barrier. In the present model, the reversible gradient structure,
together with the quadratic form of the action, yields the
action--energy identity \eqref{eq:action-energy-identity}. Reversing deterministic
paths and adding connectors of arbitrarily small action then gives
\eqref{eq:local-quasipotential}: under the hypotheses of
Proposition~\ref{prop:local-quasipotential}, the escape cost is
exactly the free-energy barrier multiplied by \(\beta\).
The restriction to one attraction basin is essential. A path crossing several wells may make several separate uphill
excursions, whose action costs add, whereas the communication height
records only the largest free-energy level visited. Thus, in general, for $i \neq j$,
\[
    \Vq(\mathcal M_i,\mathcal M_j)
    \neq
    \beta\bigl[\Hh(\mathcal M_i,\mathcal M_j)-f_i\bigr].
\]
\end{remark}

\section{Deterministic recurrence properties}
\label{sec:deterministic}

In this section, we record recurrence properties of solution trajectories to the McKean--Vlasov quotient semiflow \eqref{eq:MV}.

\begin{lemma}[Convergence to one critical orbit]
\label{lem:deterministic-convergence}
Every solution of \eqref{eq:MV} converges in \(\Qq\) to a single
critical orbit.
\end{lemma}

\begin{proof}
Fix \(t_0>0\). Uniform parabolic regularity implies that, for
a fixed regularity index larger than two,
\[
    \sup_{t\ge t_0}\|\rho_t\|_{C^3(\T)}<\infty.
\]
Hence the positive-time orbit is precompact in \(C^2(\T)\).
Its \(\omega\)-limit set
\[
    \omega(\rho_0)
    :=
    \bigcap_{T>0}
    \overline{\{\rho_t:t\ge T\}}^{\,C^2}
\]
is therefore nonempty and compact. It is connected because each
tail \(\{\rho_t:t\ge T\}\) is connected and the \(\omega\)-limit
set is the nested intersection of the closures of these tails.

By Lemma \ref{lem:energy-dissipation},
the function \(t\mapsto\F(\rho_t)\) decreases to some limit
\(\ell_\infty\). If \(\bar\rho\in\omega(\rho_0)\), then
\(\rho_{t_n}\to\bar\rho\) for some \(t_n\to\infty\), and hence
\[
    \F(\bar\rho)=\ell_\infty.
\]
Moreover, the \(\omega\)-limit set is forward invariant:
\(S_s\bar\rho\in\omega(\rho_0)\) for every \(s\ge0\).
Consequently,
\[
    \F(S_s\bar\rho)=\F(\bar\rho)
    \qquad\text{for every }s\ge0.
\]
Applying the energy-dissipation identity to the trajectory
\(S_s\bar\rho\) gives
\[
    0
    =
    \F(S_s\bar\rho)-\F(\bar\rho)
    =-\int_0^s \int_\T
       \abs{\partial_x\Psi_{S_r\bar \rho}}^2\,\dd S_r\bar\rho\,\dd r.
\]
Thus \(\int_\T
       \abs{\partial_x\Psi_{\bar \rho}}^2\,\dd \bar\rho=0\), and therefore
\(\partial_x\Psi_{\bar\rho}=0\). Hence \(\bar\rho\) is stationary.

It follows that the image of \(\omega(\rho_0)\) in
\(\Pp(\T)/\T\) is a connected subset of
\(\Crit(\F)/\T\). Since the latter set is finite, this image is
a singleton. Precompactness then implies that \(\rho_t\)
converges in the quotient to this single critical orbit.
\end{proof}

The following lemma exhibits the existence of arbitrarily small stable gates around the isolated local minima.

\begin{lemma}[Stability of the local minima]
\label{lem:poc-minimum-stability}
Under \eqref{eq:finite-critical-orbits}, every local-minimum orbit
$\mathcal M_i$ is asymptotically stable in $\Qq$: it is Lyapunov
stable and attracts a neighborhood under the McKean--Vlasov quotient semiflow \eqref{eq:MV}. Attraction is uniform on compact
subsets of its basin $\mathcal A_i$, and it admits arbitrarily small
stable gates in the sense of Definition~\ref{def:stable-gate}.
\end{lemma}

\begin{proof}
Throughout this proof,
$B_r:=\{q\in\Qq:d_{\Qq}(q,\mathcal M_i)<r\}$.
Choose $R>0$ sufficiently small that $\F\ge f_i$ on
$\overline B_{2R}(\mathcal M_i)$ and this closed ball contains no
other critical orbit. Compactness and lower semicontinuity give
\[
 \ell_R:=
 \inf_{\{q:d_{\Qq}(q,\mathcal M_i)=R\}}\F(q)>f_i.
\]
Indeed, a point on this sphere with energy $f_i$ would itself be a
local minimum, hence a different critical orbit. Choose a finite
$h$ with $f_i<h<\ell_R$.

Continuity of the semiflow on $[0,1]$ ensures that trajectories
starting sufficiently close to $\mathcal M_i$ remain in $B_R$
during this interval. Positive-time smoothing implies that, as $q\to\mathcal M_i$,
the densities representing $\overline S_1q$, after suitable
rotations, converge in $C^1$ to a stationary density representing
$\mathcal M_i$. Hence
$\F(\overline S_1q)\to f_i$. Shrinking the initial
neighborhood thus makes $\F(\overline S_1q)<h$. Energy dissipation prevents
any subsequent exit from $B_R$, since reaching its boundary would
require energy at least $\ell_R$. This argument is valid even when
$\F(q)=+\infty$ at time zero.

The radius $R$ can be arbitrarily small, which proves Lyapunov
stability. Lemma~\ref{lem:deterministic-convergence} and the absence
of other critical orbits in $\overline B_R$ give convergence to
$\mathcal M_i$ from the chosen neighborhood.
To obtain uniform attraction on a compact subset of $\mathcal A_i$, fix a compact $F\subset\mathcal A_i$ and a neighborhood $O$ of
$\mathcal M_i$. Lyapunov stability gives an open neighborhood $V$ with
$\overline S_t(V)\subset O$ for all $t\ge0$. The flow started from each $q\in F$ enters $V$
at some time $t_q$, and continuity gives a neighborhood $N_q$
of $q$ such that $\overline S_{t_q}(N_q)\subset V$.
Choose a finite subcover $N_{q_1},\ldots,N_{q_m}$ of $F$ and set
$T=\max_j t_{q_j}$. Every trajectory from $F$ thus enters $V$
by time $T$ and remains in $O$ thereafter. This proves uniform attraction.

Together with isolation of $\mathcal M_i$, these properties imply
that, inside any prescribed neighborhood $U\subset\mathcal A_i$
of $\mathcal M_i$, sufficiently small balls about $\mathcal M_i$
are stable gates in the sense of Definition~\ref{def:stable-gate}.
\end{proof}

Write $\Crit(\F)/\T=\{\mathcal K_1,\ldots,\mathcal K_m\}$.
For each $\mathcal K_i$, choose an open neighborhood $V_i$ so
small that $\overline V_i$ contains no other critical orbit.
For such a neighborhood define the local exit cost
\begin{equation}\label{eq:local-exit-cost}
 d_{\mathrm{loc}}(\mathcal K_i;V_i):=\Vq(\mathcal K_i,V_i^c).
\end{equation}

\begin{lemma}[Non-minimizing components have zero local exit cost]
\label{lem:zero-local-exit-cost}
If \(\mathcal K_i\) is not a local minimum of \(\F\) on \(\Qq\),
then, for every sufficiently small neighborhood \(V_i\) whose closure contains no other critical orbit,
\[
    d_{\mathrm{loc}}(\mathcal K_i;V_i)=0.
\]
\end{lemma}

\begin{proof}
Because \(\mathcal K_i\) is not a local minimum, there are smooth
strictly positive measures \(\nu_n\to\mathcal K_i\) in \(V_i\) such
that
\[
    \F(\nu_n)<\F(\mathcal K_i).
\]
To argue the smoothness and positivity, first choose arbitrary
lower-energy measures converging to \(\mathcal K_i\), and then evolve
all of them via \eqref{eq:MV} for one fixed, sufficiently small positive time
\(t_0\).  Continuity and smoothing of \(S_{t_0}\) imply convergence
to \(S_{t_0}\mathcal K_i=\mathcal K_i\) in the $C^2$-norm,
while energy dissipation preserves the strict lower-energy
inequality.  Relabel these regularized measures as \(\nu_n\).

In a small smooth positive neighborhood of the stationary orbit, the
drift and the weighted \(\dot H^{-1}\)-metric depend continuously on
the density.  A short interpolation between sufficiently close
densities therefore has arbitrarily small action.  Indeed, choose \(\bar\nu_n\in\mathcal K_i\) and representatives
\(\bar\nu_n=\bar\rho_n\lambda\), \(\nu_n=\rho_n\lambda\) such that
\[
    d_n:=
    \|(\rho_n-\bar\rho_n)\lambda\|_{\dot H^{-1}(\lambda)}
    \longrightarrow0.
\]
Interpolating linearly from \(\bar\rho_n\) to \(\rho_n\) over a time
\(t_n\), uniform positivity and smoothness give
\[
    I_{[0,t_n]}
    \le C\left(\frac{d_n^2}{t_n}+t_n\right).
\]
Taking \(t_n=d_n\) produces paths from \(\mathcal K_i\) to \(\nu_n\)
whose actions tend to zero.  Hence
\begin{equation}\label{eq:nearby-small-action}
    \Vq(\mathcal K_i,\{\nu_n\})\longrightarrow0.
\end{equation}
The deterministic trajectory from \(\nu_n\) must leave \(V_i\) in
finite time.  Otherwise, Lemma \ref{lem:deterministic-convergence} and the absence of other critical orbits in $\overline V_i$ would force it to converge to \(\mathcal K_i\).
This is impossible because
\[
    \F(S_t\nu_n)
    \leq
    \F(\nu_n)
    <
    \F(\mathcal K_i)
    \qquad(t\geq0).
\]
Concatenate a near-minimizing path from \(\mathcal K_i\) to \(\nu_n\) (with action arbitrarily close to $\Vq(\mathcal K_i,\{\nu_n\})$)
with the zero-action deterministic segment that exits \(V_i\).
Then
\[
    \Vq(\mathcal K_i,V_i^c)
    \leq
    \Vq(\mathcal K_i,\{\nu_n\}),
\]
and \eqref{eq:nearby-small-action} proves the claim.
\end{proof}

Thus non-minimizing critical components can be transition gates or
saddles, but they do not form positive-depth traps by themselves.
The bottom of every positive-depth reversible cycle is made of
local-minimum components.

\section{Formal construction of the reversible cycle hierarchy}
\label{sec:cycle-hierarchy}

In this section, we adapt the cycle hierarchy of Freidlin--Wentzell \cite[Chapter~6]{FreidlinWentzell2012} to our reversible infinite-dimensional setup. An energy cycle groups local minimum orbits such that the energy-barrier (communication height) to go from one orbit to another within the cycle is strictly lower than that to go outside it from any internal orbit. The key quantity that governs exit times of the empirical measure process $[\mu^N]$ from neighborhoods of the cycle-indexed orbits is the \emph{cycle depth}, introduced in Definition~\ref{def:cycle-depth}. This is quantified in the crucial result Theorem~\ref{prop:two-sided-physical-cycles} which, in turn, dictates all the metastability and propagation-of-chaos results.

Recall the collection
\[
    \mathfrak M=\{\mathcal M_1,\ldots,\mathcal M_J\}
\]
of local-minimum orbits of \(\F\) in \(\Qq\).
Assume first that \(J\geq2\).  For distinct \(i,j\), recall
\begin{equation*}
    \Phi_{ij}:=\Hh(\mathcal M_i,\mathcal M_j).
\end{equation*}

\begin{lemma}[Ultrametric inequality for communication heights]
\label{lem:communication-strong-triangle}
For distinct \(i,j,k\),
\begin{equation}
\label{eq:communication-strong-triangle}
    \Phi_{ik}
    \leq
    \max\{\Phi_{ij},\Phi_{jk}\}.
\end{equation}
Moreover,
\begin{equation}
\label{eq:communication-endpoint-lower}
    \Phi_{ij}\geq\max\{f_i,f_j\}.
\end{equation}
\end{lemma}

\begin{proof}
Concatenate paths from \(\mathcal M_i\) to \(\mathcal M_j\) and from
\(\mathcal M_j\) to \(\mathcal M_k\) whose maximal energies are
within \(\varepsilon\) of the corresponding communication heights.
This gives
\[
    \Phi_{ik}
    \leq
    \max\{\Phi_{ij},\Phi_{jk}\}+\varepsilon.
\]
Let \(\varepsilon\downarrow0\).  The endpoint bound follows because
every admissible path contains both endpoints.
\end{proof}

Let
\[
    \Lambda
    :=
    \{\Phi_{ij}:1\leq i<j\leq J\}.
\]
For \(h\in\Rr\setminus\Lambda\), define
\begin{equation}
\label{eq:height-equivalence-relation}
    i\sim_h j
    \quad\Longleftrightarrow\quad
    i=j
    \ \text{or}\
    \Phi_{ij}<h.
\end{equation}

\begin{lemma}[Height partitions]
\label{lem:height-partitions}
For every \(h\notin\Lambda\), the relation \(\sim_h\) is an
equivalence relation.  Let \(\Pi(h)\) denote its partition of
\(\{1,\ldots,J\}\).  If \(h_1<h_2\), every block of \(\Pi(h_1)\)
is contained in a unique block of \(\Pi(h_2)\).
\end{lemma}

\begin{proof}
Only transitivity is nontrivial, and it follows immediately from
\eqref{eq:communication-strong-triangle}.  Monotonicity follows from
the implication \(\Phi_{ij}<h_1\Rightarrow\Phi_{ij}<h_2\).
\end{proof}

\begin{definition}[Energy cycle]
\label{def:energy-cycle}
A nonempty set \(\mathscr C\subset\{1,\ldots,J\}\) is an
\emph{energy cycle} if it is a block of \(\Pi(h)\) for some
\(h\notin\Lambda\).  The same notation is used for the corresponding
collection of local-minimum orbits.
\end{definition}

\begin{proposition}[Nestedness and finite recursion]
\label{prop:cycles-nested}
Any two energy cycles are disjoint or one contains the other.
\end{proposition}

\begin{proof}
If \(\mathscr C_r\in\Pi(h_r)\), assume \(h_1\leq h_2\).  By
Lemma \ref{lem:height-partitions}, \(\mathscr C_1\) lies in a unique block
of \(\Pi(h_2)\); hence it is either contained in
\(\mathscr C_2\) or disjoint from it.
\end{proof}

This mechanism gives a hierarchy that has only finitely many levels.  Indeed,
\(\Pi(h)\) can change only when \(h\) crosses one of the finitely many
communication heights in \(\Lambda\).  At every nontrivial change, at
least two existing blocks merge, while no block can split.  Thus the
number of blocks strictly decreases from \(J\) toward \(1\), so there
are at most \(J-1\) nontrivial merger levels.

More precisely, list the distinct values in \(\Lambda\) as
\(
    \lambda_1<\cdots<\lambda_L.
\)
Choose one threshold below \(\lambda_1\), one strictly between each
successive pair, and one above \(\lambda_L\).  The associated
partitions begin with the singleton partition and end with
\(\{\{1,\ldots,J\}\}\).  At level \(\lambda_r\), precisely those
blocks joined at that communication level merge.  Since \(L<\infty\),
the recursion terminates.

If $J=1$, the hierarchy consists only of $\{1\}$ and there
are no proper cycles.  The cycle-exit statements in our main results are then vacuous.

\subsection{Internal levels, exit levels, and bottom sets}

Let $\mathscr C$ be a proper energy cycle:
\[
    \varnothing\neq\mathscr C
    \subsetneq
    \{1,\ldots,J\}.
\]
  Define
\begin{equation}
\label{eq:cycle-internal-level}
    \Lambda_{\mathrm{in}}(\mathscr C)
    :=
    \begin{cases}
       \displaystyle
       \max_{\substack{i,j\in\mathscr C\\i\neq j}}\Phi_{ij},
       &|\mathscr C|\geq2,\\[2mm]
       -\infty,&|\mathscr C|=1,
    \end{cases}
\end{equation}
and its exit communication level
\begin{equation}
\label{eq:cycle-exit-level}
    \Gamma(\mathscr C)
    :=
    \min_{\substack{i\in\mathscr C\\j\notin\mathscr C}}
       \Phi_{ij}.
\end{equation}

\begin{lemma}[Separation of internal and external levels]
\label{lem:cycle-level-separation}
Every proper energy cycle satisfies
\begin{equation}
\label{eq:cycle-level-strict-separation}
    \Lambda_{\mathrm{in}}(\mathscr C)
    <
    \Gamma(\mathscr C).
\end{equation}
\end{lemma}

\begin{proof}
Choose \(h\notin\Lambda\) for which \(\mathscr C\) is a block of
\(\Pi(h)\).  Internal pairs satisfy \(\Phi_{ij}<h\), while every
pair with one index outside satisfies \(\Phi_{ij}>h\); equality is
excluded because \(h\notin\Lambda\).
\end{proof}

Define the bottom energy by
\begin{equation}
\label{eq:cycle-bottom-energy}
    F_{\mathscr C}:=\min_{i\in\mathscr C}f_i.
\end{equation}
We call the orbits minimizing the free-energy among those indexed by $\mathscr C$ as the bottom orbits of the cycle. 

\begin{definition}[Cycle depth]
\label{def:cycle-depth}
The \emph{action-scale depth} of a proper energy cycle is
\begin{equation}
\label{eq:cycle-depth-height}
    d(\mathscr C)
    :=
    \beta\bigl[
       \Gamma(\mathscr C)-F_{\mathscr C}
    \bigr].
\end{equation}
\end{definition}

The depth \(d(\mathscr C)\) governs the exponential scale
of exit times of the empirical measure process $[\mu^N]$ from neighborhoods of the critical orbits indexed by $\mathscr C$. Any finite-action path from a bottom orbit of $\mathscr C$ to a minimum outside \(\mathscr C\) must reach
free-energy at least \(\Gamma(\mathscr C)\). Hence
Lemma~\ref{lem:action-energy} bounds its action below by
\(d(\mathscr C)\). As noted in Remark~\ref{rem:false-global-qp}, this lower bound need not be sharp: a path
crossing several wells may incur successive uphill costs whose cumulative action exceeds the cycle depth. However, the stochastic process $\mu^N$ exhibits \emph{multiscale behavior} where it typically makes many transitions among
lower-rank subcycles before a successful exit from a neighborhood indexed by $\mathscr C$.
The Freidlin--Wentzell cycle recursion
\cite[Chapter~6]{FreidlinWentzell2012} accounts for these repeated
attempts and exhibits the surprising phenomenon that the cycle-exit time indeed behaves like \(\exp\{N d(\mathscr C)\}\) in their finite dimensional small-noise diffusion setting, even when it is smaller than the minimum
action required for a single transition from the cycle bottom
to an outside minimum. In our infinite-dimensional setting,
Theorem~\ref{prop:two-sided-physical-cycles} establishes the
precise exit-time estimates on the exponential scale
\(\exp\{N d(\mathscr C)\}\).

\begin{remark}[A two-well illustration]
\label{rem:two-well-illustration}
Consider a cycle consisting of two local minima
\(\mathcal M_1\) and \(\mathcal M_2\), with
\[
    \F(\mathcal M_1)=0,
    \qquad
    \F(\mathcal M_2)=9.
\]
Suppose that the two wells communicate through a saddle of height
\(10\), and that the cycle can be left through a gate of height \(11\)
accessible from \(\mathcal M_2\).  A natural one-shot path starting
from \(\mathcal M_1\) first climbs to level \(10\) via the time-reversal of a deterministic trajectory (with small added connectors), descends
deterministically into \(\mathcal M_2\), and then climbs from level
\(9\) to the exit gate at level \(11\).  Its action is therefore close
to
\[
    \beta\bigl[(10-0)+(11-9)\bigr]=12\beta.
\]

The stochastic exit mechanism has a smaller exponent.  A transition
from \(\mathcal M_1\) to \(\mathcal M_2\) occurs on the time scale
\(\e^{10\beta N}\).  Once in \(\mathcal M_2\), returning to
\(\mathcal M_1\) requires crossing a barrier of height \(10-9=1\),
whereas exiting the cycle requires crossing a barrier of height
\(11-9=2\).  Hence the probability of exiting before returning is
\[
    \e^{-\beta N+o(N)},
\]
so approximately \(\e^{\beta N}\) toggles between \(\mathcal M_1\) and \(\mathcal M_2\) are
needed for one successful exit.  The resulting exit-time scale is
therefore
\[
    \e^{10\beta N}\e^{\beta N}
    =
    \e^{11\beta N},
\]
corresponding to the cycle depth
\[
    d(\mathscr C)
    =
    \beta\bigl[\Gamma(\mathscr C)-F_{\mathscr C}\bigr]
    =
    11\beta.
\]
Thus the cycle depth records the cumulative effect of repeated
transition attempts and need not coincide with the action of a single
continuous path.
\end{remark}

The following key lemma connects the quasipotential cost between two local minima with the cycle depth.

\begin{lemma}[Exit elevation of an energy cycle]
\label{lem:singleton-exit-cost}
Assume $J \ge 2$. Let $\mathscr C$ be any proper energy cycle. Then
\begin{equation}
\label{eq:cycle-exit-elevation}
    \min_{\substack{i\in\mathscr C\\ j\notin\mathscr C}}
    \left\{
        \beta(f_i-F_{\mathscr C})
        +
        \Vq(\mathcal M_i,\mathcal M_j)
    \right\}
    =
    \beta\bigl(
        \Gamma(\mathscr C)-F_{\mathscr C}
    \bigr)
    =
    d(\mathscr C).
\end{equation}
Equivalently,
\begin{equation}
\label{eq:absolute-cycle-exit-elevation}
    \min_{\substack{i\in\mathscr C\\ j\notin\mathscr C}}
    \left\{
        \beta f_i+\Vq(\mathcal M_i,\mathcal M_j)
    \right\}
    =
    \beta\Gamma(\mathscr C).
\end{equation}
In particular, if $\mathscr C=\{i\}$ is a proper singleton cycle,
then $F_{\mathscr C}=f_i$ and
\begin{equation}
\label{eq:singleton-exit-cost}
    \min_{j\ne i}
    \Vq(\mathcal M_i,\mathcal M_j)
    =
    \beta\bigl(
        \Gamma(\{i\})-f_i
    \bigr)
    =
    d(\{i\}).
\end{equation}
\end{lemma}

\begin{proof}
We first record three consequences of the preceding results.

\medskip
\noindent
\emph{Step 1: Time reversal.}
Let $\phi$ be a finite-action path on $[0,T]$ with $\F(\phi_0)<\infty$, and define its time reversal by
\[
    \check\phi_t:=\phi_{T-t},
    \qquad 0\le t\le T.
\]
Expanding the two action functionals, or equivalently using the
action--energy decomposition, gives
\begin{equation}
\label{eq:action-reversal}
    I_{[0,T]}(\phi)-I_{[0,T]}(\check\phi)
    =
    \beta\left[
        \mathcal F_\beta(\phi_T)
        -
        \mathcal F_\beta(\phi_0)
    \right].
\end{equation}
In particular, if $\phi$ runs from $q$ to $q'$, then
\begin{equation}
\label{eq:reverse-action}
    I_{[0,T]}(\check\phi)
    =
    I_{[0,T]}(\phi)
    +
    \beta\left[
        \mathcal F_\beta(q)
        -
        \mathcal F_\beta(q')
    \right].
\end{equation}

\medskip
\noindent
\emph{Step 2: Small connectors on compact regular sets.}
Suppose that $q_n^-,q_n^+$ admit smooth positive representatives
$\rho_n^-\lambda,\rho_n^+\lambda$ such that, after suitable rotations,
\[
    \rho_n^+-\rho_n^-\longrightarrow0
    \qquad\text{in }C^1(\mathbb T),
\]
and that these densities are uniformly bounded away from zero and
uniformly bounded in $C^1(\mathbb T)$. Then
\begin{equation}
\label{eq:compact-small-connector}
    \Vq(q_n^-,q_n^+)\longrightarrow0.
\end{equation}
Indeed, set
\[
    r_n
    :=
    \left\|
        (\rho_n^+-\rho_n^-)\lambda
    \right\|_{\dot H^{-1}(\lambda)}
\]
and, when $r_n>0$, interpolate linearly on $[0,r_n]$:
\[
    \rho_t^{(n)}
    :=
    \left(1-\frac{t}{r_n}\right)\rho_n^-
    +
    \frac{t}{r_n}\rho_n^+.
\]
Uniform positivity makes the weighted negative Sobolev norms
uniformly equivalent to the corresponding norm weighted by
$\lambda$, while
$\Gg(\rho_t^{(n)}\lambda)$ is uniformly bounded in that norm.
Consequently,
\[
    I_{[0,r_n]}(\rho^{(n)}\lambda)
    \le
    C\left(
        \frac{r_n^2}{r_n}+r_n
    \right)
    =
    2Cr_n
    \longrightarrow0.
\]
The case $r_n=0$ is immediate.

\medskip
\noindent
\emph{Step 3: Zero-cost accessibility of a local minimum.}
For every smooth positive $q\in\Qq$, there exists a local-minimum
component $\mathcal M_k$ such that
\begin{equation}
\label{eq:zero-cost-accessibility}
    \Vq(q,\mathcal M_k)=0.
\end{equation}
Indeed, let $\mathcal K_0$ be the critical component approached by
the deterministic trajectory from $q$, as given by
Lemma~\ref{lem:deterministic-convergence}. Following the deterministic
trajectory has zero action, and parabolic regularity together with
Step~2 gives
\[
    \Vq(q,\mathcal K_0)=0.
\]
If $\mathcal K_0$ is a local-minimum component, the claim follows.
Otherwise, the proof of Lemma~\ref{lem:zero-local-exit-cost} provides
smooth points $\nu_n\to\mathcal K_0$ such that
\[
    \mathcal F_\beta(\nu_n)
    <
    \mathcal F_\beta(\mathcal K_0),
    \qquad
    \Vq(\mathcal K_0,\nu_n)\longrightarrow0.
\]
Let $\mathcal K(\nu_n)$ be the critical component approached by the
deterministic trajectory from $\nu_n$. Since there are only finitely
many critical components, after passing to a subsequence we may
assume that
\(
    \mathcal K(\nu_n)=\mathcal K_1
\)
is independent of $n$. Following the deterministic trajectory from
$\nu_n$ and using a small connector near its limit gives
\[
    \Vq(\mathcal K_0,\mathcal K_1)=0.
\]
Moreover,
\[
    \mathcal F_\beta(\mathcal K_1)
    \le
    \mathcal F_\beta(\nu_n)
    <
    \mathcal F_\beta(\mathcal K_0).
\]
If $\mathcal K_1$ is not a local minimum, repeat the argument. At
each step the critical value decreases strictly, so finiteness of
the critical components implies that the procedure terminates at
some local-minimum component $\mathcal M_k$. Concatenating the
arbitrarily small-action paths proves
\eqref{eq:zero-cost-accessibility}.

We now prove \eqref{eq:cycle-exit-elevation}. For every
$i\in\mathscr C$ and $j\notin\mathscr C$, the action--energy lower
bound in Lemma \ref{lem:action-energy} gives
\[
    \Vq(\mathcal M_i,\mathcal M_j)
    \ge
    \beta\bigl(\Phi_{ij}-f_i\bigr).
\]
Since $\Phi_{ij}\ge\Gamma(\mathscr C)$,
\begin{align*}
    \beta(f_i-F_{\mathscr C})
    +
    \Vq(\mathcal M_i,\mathcal M_j)
    &\ge
    \beta(f_i-F_{\mathscr C})
    +
    \beta(\Phi_{ij}-f_i)
    \\
    &=
    \beta(\Phi_{ij}-F_{\mathscr C})
     \ge
    \beta\bigl(
        \Gamma(\mathscr C)-F_{\mathscr C}
    \bigr)
    =
    d(\mathscr C).
\end{align*}
Therefore
\begin{equation}
\label{eq:cycle-exit-elevation-lower}
    \min_{\substack{i\in\mathscr C\\ j\notin\mathscr C}}
    \left\{
        \beta(f_i-F_{\mathscr C})
        +
        \Vq(\mathcal M_i,\mathcal M_j)
    \right\}
    \ge d(\mathscr C).
\end{equation}

For the converse inequality, choose
$\bar i\in\mathscr C$ and $\bar j\notin\mathscr C$ such that
\(
    \Phi_{\bar i \, \bar j}=\Gamma(\mathscr C),
\)
and choose $i_0\in\mathscr C$ such that
\(
    f_{i_0}=F_{\mathscr C}.
\)
If $i_0\ne\bar i$, Lemmas
\ref{lem:communication-strong-triangle} and
\ref{lem:cycle-level-separation} give
\[
    \Phi_{i_0\bar j}
    \le
    \max\{
        \Phi_{i_0\bar i},
        \Phi_{\bar i\bar j}
    \}
    =
    \Gamma(\mathscr C).
\]
The same conclusion is immediate if $i_0=\bar i$. Since
$i_0\in\mathscr C$ and $\bar j\notin\mathscr C$, the definition of
$\Gamma(\mathscr C)$ gives the reverse inequality. Hence
\begin{equation}
\label{eq:bottom-to-cycle-exit-height}
    \Phi_{i_0\bar j}=\Gamma(\mathscr C).
\end{equation}

Fix $\varepsilon>0$. By
\eqref{eq:bottom-to-cycle-exit-height}, there is a continuous path
$\gamma:[0,1]\to\Qq$ satisfying
\[
    \gamma(0)\in\mathcal M_{i_0},
    \qquad
    \gamma(1)\in\mathcal M_{\bar j},
    \qquad
   \sup_{0\le s\le1}
    \mathcal F_\beta(\gamma(s))
    \le
    \Gamma(\mathscr C)+\varepsilon.
\]
For a fixed $\tau>0$, replace $\gamma(s)$ by $S_\tau\gamma(s)$. The endpoints are unchanged, the free-energy does
not increase, and uniform parabolic regularity shows that the
regularized path is a compact $C^2$-continuous family, modulo
rotations, of smooth densities uniformly bounded away from zero. We
continue to denote this path by $\gamma$.

Define
\[
    A_{\mathscr C}
    :=
    \left\{
        s\in[0,1]:
        \Vq(\gamma(s),\mathcal M_\ell)=0
        \text{ for some }\ell\in\mathscr C
    \right\}.
\]
Clearly $0\in A_{\mathscr C}$.

Suppose first that $1\in A_{\mathscr C}$. Then, for some
$\ell\in\mathscr C$,
\[
    \Vq(\mathcal M_{\bar j},\mathcal M_\ell)=0.
\]
Choose paths from $\mathcal M_{\bar j}$ to $\mathcal M_\ell$ whose
actions tend to zero and reverse them. By
\eqref{eq:reverse-action},
\[
    \Vq(\mathcal M_\ell,\mathcal M_{\bar j})
    \le
    \beta(f_{\bar j}-f_\ell).
\]
Since
\[
    f_{\bar j}
    \le
    \Phi_{i_0\bar j}
    =
    \Gamma(\mathscr C),
\]
we obtain
\begin{align*}
    \beta(f_\ell-F_{\mathscr C})
    +
    \Vq(\mathcal M_\ell,\mathcal M_{\bar j})
    &\le
    \beta(f_{\bar j}-F_{\mathscr C})
    \\
    &\le
    \beta\bigl(
        \Gamma(\mathscr C)-F_{\mathscr C}
    \bigr)
    =
    d(\mathscr C).
\end{align*}
Together with \eqref{eq:cycle-exit-elevation-lower}, this proves the
claim in this case.

Suppose now that $1\notin A_{\mathscr C}$. Since
$0\in A_{\mathscr C}$ and $[0,1]$ is connected, there exist
$s_*\in[0,1]$ and sequences
\[
    s_n\in A_{\mathscr C},
    \qquad
    t_n\notin A_{\mathscr C},
    \qquad
    s_n,t_n\longrightarrow s_*.
\]
Set
\[
    q_n^-:=\gamma(s_n),
    \qquad
    q_n^+:=\gamma(t_n).
\]
By Step~2, there are connector paths $\chi_n$ from $q_n^-$ to
$q_n^+$ such that
\(
    I(\chi_n)\longrightarrow0.
\)
For each $n$, choose $\ell_n\in\mathscr C$ such that
\(
    \Vq(q_n^-,\mathcal M_{\ell_n})=0.
\)
By Step~3, choose $k_n$ such that
\[
    \Vq(q_n^+,\mathcal M_{k_n})=0.
\]
Since $t_n\notin A_{\mathscr C}$, necessarily
$k_n\notin\mathscr C$. Passing to a subsequence, we may assume that
\[
    \ell_n=\ell\in\mathscr C,
    \qquad
    k_n=k\notin\mathscr C
\]
are independent of $n$. Choose paths $\psi_n^-$ from $q_n^-$ to
$\mathcal M_\ell$ and paths $\psi_n^+$ from $q_n^+$ to
$\mathcal M_k$ such that
\[
    I(\psi_n^-)\longrightarrow0,
    \qquad
    I(\psi_n^+)\longrightarrow0.
\]

Concatenating $\check\psi_n^-$, $\chi_n$, and $\psi_n^+$ gives a
path from $\mathcal M_\ell$ to $\mathcal M_k$. By
\eqref{eq:reverse-action},
\begin{align*}
    \Vq(\mathcal M_\ell,\mathcal M_k)
    &\le
    I(\check\psi_n^-)
    +
    I(\chi_n)
    +
    I(\psi_n^+)
    \\
    &=
    I(\psi_n^-)
    +
    \beta\left[
        \mathcal F_\beta(q_n^-)-f_\ell
    \right]
    +
    I(\chi_n)
    +
    I(\psi_n^+)
    \\
    &\le
    \beta\left[
        \Gamma(\mathscr C)+\varepsilon-f_\ell
    \right]
    +o(1).
\end{align*}
It follows that
\[
    \beta(f_\ell-F_{\mathscr C})
    +
    \Vq(\mathcal M_\ell,\mathcal M_k)
    \le
    \beta\left[
        \Gamma(\mathscr C)+\varepsilon-F_{\mathscr C}
    \right]
    +o(1).
\]
Letting $n\to\infty$ and then $\varepsilon\downarrow0$ gives the
reverse inequality in
\eqref{eq:cycle-exit-elevation}. This proves
\eqref{eq:cycle-exit-elevation}, and
\eqref{eq:absolute-cycle-exit-elevation} follows by adding
$\beta F_{\mathscr C}$ to both sides. The singleton statement
\eqref{eq:singleton-exit-cost} is the special case
$\mathscr C=\{i\}$.
\end{proof}

\section{Gate transition estimates}
\label{sec:long-time}
In this section, we construct and analyze excursions of the empirical measure process between subsequent arbitrarily small stable gates around the free-energy minima (metastable wells). 
The uniform finite-horizon LDP, deterministic recurrence, and connector
constructions above are then used to give two-sided transition probability estimates in Lemma~\ref{lem:gate-transition-reduction} for the `gate label' process, which records the indices of the successively explored gates. A key technical input, required to utilize the fixed-horizon LDP results in obtaining these transition estimates, is to compare the action of arbitrarily long paths of bounded action with fixed-horizon paths. This is carried out via a `path-shortening' scheme described in Lemma~\ref{lem:uniform-shortening-arbitrary-gate-points} and its proof.

For suitably small $\rho_0>0$ and
\(0<\rho<\rho_0\), choose pairwise disjoint nested gates
\begin{equation}
\label{eq:minimum-gates}
    \mathcal M_i\Subset G_{i,\rho}^-\Subset G_{i,\rho}^+,
    \qquad 1\leq i\leq J,
\end{equation}
whose closures decrease to \(\mathcal M_i\) as \(\rho\downarrow0\)
and whose boundaries are disjoint from the critical orbit set \(\Crit(\F)/\T\).  They are
chosen inside deterministic attraction neighborhoods of the
corresponding local minima.  By Lemma~\ref{lem:poc-minimum-stability}, each inner gate
may be chosen to be a stable gate.
Set
\[
    K_\rho:=\bigcup_{i=1}^J K_{i,\rho},
    \qquad
    K_{i,\rho}:=\overline{G_{i,\rho}^-}.
\]
Let \((\mathcal F_t)_{t\ge0}\) denote the usual augmentation
of the natural filtration generated by $\mathbf X_0^N$ and the driving Brownian motions
\((B^1,\ldots,B^N)\), suppressing its dependence on \(N\).
All entrance and exit times below are stopping times with
respect to \((\mathcal F_t)\).
Write \(X_t^N=[\mu_t^N]\), and define
\[
    \tau_0:=\inf\{t\geq0:X_t^N\in K_\rho\}.
\]
If \(X_{\tau_n}^N\in K_{Y_n,\rho}\), let
\begin{equation}
\label{eq:gate-renewal-times}
\begin{aligned}
    \eta_n
    &:=
    \inf\{t\geq\tau_n:
        X_t^N\notin G_{Y_n,\rho}^+\},\\
    \tau_{n+1}
    &:=
    \inf\{t\geq\eta_n:
        X_t^N\in K_\rho\}.
\end{aligned}
\end{equation}
The label \(Y_{n+1}\) is the unique index such that
\[
    X_{\tau_{n+1}}^N\in K_{Y_{n+1},\rho}.
\]
We use the convention
\(\eta_{-1}:=0\). Once
\(\rho\) has been fixed, we often suppress it from the notation. 

The entrance-point chain
\(Z_n :=X_{\tau_n}^N\) (suppressing the superscript $N$ for notational convenience) is Markov on \(K\).  Its label process
\(Y_n\) need not itself be Markov, because the law of the next label
may depend on the point at which the present gate was entered.  The
uniform transition estimates below are precisely what is needed to
remove this dependence on the exponential scale.

For a path \(\phi \in C([0,\infty);\Qq)\) and $1 \le i, j \le J$, define the direct action cost to go from $\mathcal M_i$ to $\mathcal M_j$ by
\begin{equation}
\label{eq:limiting-direct-cost-variational}
\begin{aligned}
   \Delta_{ij}
    :=
    \inf \Bigl\{
        &I_{[0,T]}(\phi):
        T>0,\quad
        \phi_0 =\mathcal M_i,\quad
        \phi_T =\mathcal M_j,\\
        &\phi_t\notin \{\mathcal M_1, \dots, \mathcal M_J\}
        \text{ for every }0<t<T
    \Bigr\}.
\end{aligned}
\end{equation}

For $i\ne j$, the direct costs are finite. Smooth strictly
positive stationary representatives can be joined by an affine
path with a common positive lower bound for its densities and
bounded spatial derivatives; its action is finite.
Lemma~\ref{lem:critical-contact-removal} below removes any
interior critical contacts with arbitrarily small additional
action.

The quantities $\Delta_{ij}$ are the one-step costs used in the
gate estimates. Their definition imposes an avoidance constraint.
However, by Lemma~\ref{lem:critical-contact-removal}, this constraint does not change the infimum:
\begin{equation}\label{eq:DeltaeqV}
 \Delta_{ij}=\Vq(\mathcal M_i,\mathcal M_j),\qquad i\ne j.
\end{equation}
This equality does not identify the cost with a
communication-height difference or with a cycle depth; a
single path may still accumulate several uphill action costs.

The following lemma gives a lower bound on the gate transition probabilities in terms of the direct costs $\Delta_{ij}$.

\begin{lemma}[Direct lower bound for gate transitions]
\label{lem:direct-gate-transition-lower}
Let \(\Delta_{ij}\) denote the direct
variational cost in \eqref{eq:limiting-direct-cost-variational}.
For every $\alpha>0$, one can choose a shrinking family of nested attracting gates
\[
    \mathcal M_i\Subset G_{i,\rho}^-\Subset G_{i,\rho}^+,
    \qquad
    K_{i,\rho}:=\overline{G_{i,\rho}^-},
    \qquad
    K_\rho:=\bigcup_{i=1}^J K_{i,\rho},
\]
so that the following holds for every fixed sufficiently small $\rho>0$. Write \(X_t^N:=[\mu_t^N]\) with
\(X_0^N = q\in K_{i,\rho}\cap\PN(\T)\), and set
\[
    \sigma_{i,\rho}^N
    :=\inf\{t\geq0:X_t^N\notin G_{i,\rho}^+\},
    \qquad
    \vartheta_{i,\rho}^N
    :=\inf\{t\geq\sigma_{i,\rho}^N:X_t^N\in K_\rho\},
\]
and
\[
    p_{ij}^{N,\rho}(q)
    :=\mathbb P_q\bigl(
        \vartheta_{i,\rho}^N<\infty,\quad
        X_{\vartheta_{i,\rho}^N}^N\in K_{j,\rho}
    \bigr).
\]
For each such fixed choice of gates and fixed $\rho$, there exists $N_0=N_0(\alpha,\rho,\text{gates})<\infty$ such that, simultaneously for every \(i\neq j\), \(1 \le i,j \le J\),
and \(N\geq N_0\),
\begin{equation}
\label{eq:direct-gate-transition-lower}
    \inf_{q\in K_{i,\rho}\cap\PN(\T)}
        p_{ij}^{N,\rho}(q)
    \geq
    \exp\{-N(\Delta_{ij}+\alpha)\}.
\end{equation}
\end{lemma}

\begin{proof}
We use the open-tube argument of
\cite[Chapter~6, \S2, proof of Lemma~2.1]{FreidlinWentzell2012}.
For each \(i\neq j\), recalling \(\Delta_{ij}<\infty\), choose a path
\(\phi^{ij}:[0,T_{ij}]\to\Qq\) admissible for
\eqref{eq:limiting-direct-cost-variational}, with
\[
    I_{[0,T_{ij}]}(\phi^{ij})
    \leq \Delta_{ij}+\frac{\alpha}{3}.
\]
Choose arbitrarily small outer gates with pairwise disjoint closures.
For each path, choose \(a_{ij}\in(0,T_{ij})\) such that
\(\phi^{ij}([0,a_{ij}])\subset G_{i,\rho}^+\).
The compact tail \(\phi^{ij}([a_{ij},T_{ij}])\) avoids
\(\mathcal M_i\), and the full path avoids every
\(\mathcal M_k\), \(k\notin\{i,j\}\).  Thus the inner gates
can be chosen sufficiently small, simultaneously for all these paths,
that
\[
    \phi^{ij}([a_{ij},T_{ij}])
    \cap\bigcup_{k\neq j}K_{k,\rho}=\varnothing.
\]
Choose them also so that deterministic
trajectories starting in \(K_{i,\rho}\) remain in
\(G_{i,\rho}^+\) and converge to \(\mathcal M_i\).

For each \(q\in K_{i,\rho}\), follow its zero-action deterministic
trajectory and then append the small affine connector from
\eqref{eq:compact-small-connector}.  Attraction and positive-time
smoothing give a path \(\gamma^q: [0,S_q] \rightarrow \Qq\) from \(q\) to
\(\mathcal M_i\), contained in \(G_{i,\rho}^+\), with action
less than \(\alpha/3\).  Concatenate it with \(\phi^{ij}\),
aligning representatives by a rotation if necessary.  The resulting
path \(\psi^q\), of duration \(L_q\), satisfies
\[
    I_{[0,L_q]}(\psi^q)
    \leq \Delta_{ij}+\frac{2\alpha}{3}.
\]
Its initial segment, through time \(S_q+a_{ij}\), lies in
\(G_{i,\rho}^+\);
its remaining segment avoids every \(K_{k,\rho}\), \(k\neq j\),
and its endpoint lies in \(G_{j,\rho}^-\).
By compactness of these path segments, every path in a sufficiently small open
uniform tube \(B_q\) around \(\psi^q\) has the same properties.
Every path in this tube therefore makes a successful
\(j\)-transition by time \(L_q\).

Proposition~\ref{prop:uniform-LDP}\textup{(ii)}, applied to the
rotation-invariant lift of \(B_q\), gives a neighborhood \(V_q\)
of \(q\) such that, for all sufficiently large \(N\) and every
\(\nu\in V_q\cap K_{i,\rho}\cap\PN(\T)\),
\[
\begin{aligned}
    p_{ij}^{N,\rho}(\nu)
    &\geq\mathbb P_\nu\bigl(X^N|_{[0,L_q]}\in B_q\bigr)\\
    &\geq\exp\left\{-N\left(
        I_{[0,L_q]}(\psi^q)+\frac{\alpha}{3}\right)\right\}
    \geq\exp\{-N(\Delta_{ij}+\alpha)\}.
\end{aligned}
\]
A finite subcover of the compact set \(K_{i,\rho}\), followed by
a maximum over the finitely many ordered pairs, gives a single
threshold $N_0$, depending on the fixed gates. No uniformity
of $N_0$ as $\rho\downarrow0$ is asserted.
\end{proof}

Now, we work towards a comparable upper bound, which is considerably more involved.
The following `path-shortening' lemma is key to our analysis. It enables us to replace paths of uniformly bounded action by paths of uniformly bounded duration.

Recall \(\operatorname{Crit}(\mathcal F)/\mathbb T=\{\mathcal K_1,\ldots,\mathcal K_m\}\), and set
\[
e_\alpha:=\mathcal F(\mathcal K_\alpha),\qquad 1\le\alpha\le m.
\]

\begin{lemma}[Path-shortening]
\label{lem:uniform-shortening-arbitrary-gate-points}
Assume that
\(
    \operatorname{Crit}(\mathcal F)/\mathbb T
\)
is finite.
For every \(A<\infty\) and \(\varepsilon>0\), there exists
\(T_*=T_*(A,\varepsilon)<\infty\) such that the following holds.

Let
\[
    \phi:[0,T]\longrightarrow\mathcal Q
\]
be any finite-action path satisfying
\[
    I_{[0,T]}(\phi)\leq A.
\]
Then there is a finite-action path
\[
    \widetilde\phi:
    [0,\widetilde T]\longrightarrow\mathcal Q
\]
such that
\[
    \widetilde\phi_0=\phi_0,
    \qquad
    \widetilde\phi_{\widetilde T}=\phi_T,
    \qquad
    \widetilde T\leq T_*,
\]
and
\[
    I_{[0,\widetilde T]}(\widetilde\phi)
    \leq
    I_{[0,T]}(\phi)+\varepsilon.
\]
\end{lemma}

\begin{proof}
The proof is divided into four steps.

\medskip
\noindent
\emph{Step 1: Low-energy cores and local connectors.}
For each critical component \(\mathcal K_\alpha\), choose a smooth
strictly positive stationary representative
\(\bar\rho_\alpha\lambda\), and choose pairwise disjoint
neighborhoods
\[
    \mathcal K_\alpha
    \Subset U_\alpha
    \Subset V_\alpha
    \Subset\mathcal Q
\]
containing no other critical component.  These neighborhoods will be
made smaller below.

Fix \(\eta,\delta>0\).  By lower semicontinuity of \(\mathcal F\),
the \(V_\alpha\) may be chosen so that
\begin{equation}
\label{eq:shortening-energy-lower-near-critical}
    \mathcal F(q)\geq e_\alpha-\eta,
    \qquad q\in\overline V_\alpha.
\end{equation}
By continuity of the quotient semiflow and stationarity of
\(\mathcal K_\alpha\), the \(U_\alpha\) may then be chosen so that
\begin{equation}
\label{eq:shortening-flow-stays-local}
    \overline S_t(\overline U_\alpha)
    \subset V_\alpha,
    \qquad 0\leq t\leq1.
\end{equation}

Positive-time parabolic smoothing implies that
\(
    q\longmapsto \overline S_1q
\)
is continuous from the Wasserstein topology into the quotient
\(C^2\)-topology.  Consequently, after shrinking \(U_\alpha\), every
\(\overline S_1q\), \(q\in\overline U_\alpha\), has a representative
whose density is arbitrarily close in \(C^2\) to
\(\bar\rho_\alpha\), after applying a suitable rotation.

We use the following elementary local fact. If two smooth positive
densities \(\rho^0,\rho^1\) are sufficiently close in \(C^1\) to the
same stationary density \(\bar\rho_\alpha\), then they can be joined
in time \(1\), inside \(V_\alpha\), with action at most \(\delta\).
Indeed, consider
\[
    \rho_t:=(1-t)\rho^0+t\rho^1,
    \qquad 0\leq t\leq1.
\]
If the \(C^1\)-neighborhood is sufficiently small, then
\(\rho_t\geq c_\alpha>0\) and \([\rho_t\lambda]\in V_\alpha\).
Set
\[
    s_t
    :=
    \partial_t(\rho_t\lambda)
    -
    \mathcal G(\rho_t\lambda).
\]
This is a smooth signed distribution of total mass zero.  In one
dimension, choose the mean-zero periodic primitive \(J_t\) satisfying
\(
    -\partial_xJ_t=s_t
\)
and set
\(
    u_t:=\frac{J_t}{\rho_t}.
\)
Then
\[
    s_t=-\partial_x(\rho_tu_t)\lambda
\]
and
\[
    \|s_t\|_{\dot H^{-1}(\rho_t\lambda)}^2
    \leq
    \int_{\mathbb T}\frac{|J_t|^2}{\rho_t}\,\dd\lambda
    \leq
    c_\alpha^{-1}\|J_t\|_{L^2(\lambda)}^2.
\]
Since \(\bar\rho_\alpha\) is stationary, the right-hand side tends
uniformly to zero as
\[
    \|\rho^0-\bar\rho_\alpha\|_{C^1}
    +
    \|\rho^1-\bar\rho_\alpha\|_{C^1}
    \longrightarrow0.
\]
Thus the interpolation has action at most \(\delta\).

Define the low-energy core
\[
    L_{\alpha,\eta}
    :=
    \left\{
        q\in\overline U_\alpha:
        \mathcal F(q)\leq e_\alpha+\eta
    \right\}.
\]
We claim that any two points \(q,q'\in L_{\alpha,\eta}\) can be
joined inside \(V_\alpha\), in time at most \(3\), with action at
most
\begin{equation}
\label{eq:shortening-local-connector-cost}
    2\beta\eta+\delta.
\end{equation}
Indeed, concatenate the following three paths:
\[
    q\longrightarrow\overline S_1q,
    \qquad
    \overline S_1q\longrightarrow\overline S_1q',
    \qquad
    \overline S_1q'\longrightarrow q'.
\]
The first is a deterministic trajectory and therefore has zero
action.  The middle path is the preceding smooth interpolation and
has action at most \(\delta\).  The last path is the reversal of the
deterministic segment from \(q'\) to \(\overline S_1q'\), and hence
has action
\[
    \beta\left[
        \mathcal F(q')
        -
        \mathcal F(\overline S_1q')
    \right].
\]
By the definition of \(L_{\alpha,\eta}\) and
\eqref{eq:shortening-energy-lower-near-critical},
\[
    \mathcal F(q')\leq e_\alpha+\eta,
    \qquad
    \mathcal F(\overline S_1q')\geq e_\alpha-\eta.
\]
The reversed segment therefore has action at most \(2\beta\eta\),
which proves \eqref{eq:shortening-local-connector-cost}.  By
\eqref{eq:shortening-flow-stays-local}, the entire connector remains
inside \(V_\alpha\).  Notice that this argument uses no attracting
property of \(\mathcal K_\alpha\), and therefore applies equally to
local minima and saddles.

\medskip
\noindent
\emph{Step 2: A positive dissipation gap outside the cores.}
For \(\mu\in\mathcal P(\mathbb T)\), define
\begin{equation}
\label{eq:dissipation-extended-definition}
    \mathcal D(\mu)
    :=
    \begin{cases}
    \displaystyle
    \int_{\mathbb T}
        \left|
            2\beta^{-1}\partial_x\sqrt\rho
            +
            \sqrt\rho\,(W'*\mu)
        \right|^2
        \,\dd\lambda,
    &
    \begin{array}{l}
    \mu=\rho\lambda,\\[-1mm]
    \sqrt\rho\in H^1(\mathbb T),
    \end{array}
    \\[6mm]
    +\infty,
    &\text{otherwise.}
    \end{cases}
\end{equation}
We claim that
\begin{equation}
\label{eq:shortening-dissipation-gap}
    d_\eta
    :=
    \inf_{
        q\in\mathcal Q\setminus
        \bigcup_{\alpha=1}^mL_{\alpha,\eta}
    }
    \mathcal D(q)
    >0.
\end{equation}
Suppose otherwise.  There would exist
\(
    q_n=[\rho_n\lambda]
    \notin\bigcup_{\alpha=1}^mL_{\alpha,\eta}
\)
such that
\(
    \mathcal D(q_n)\longrightarrow0.
\)
The identity
\[
    2\beta^{-1}\partial_x\sqrt{\rho_n}
    =
    \sqrt{\rho_n}
    \left(
        \beta^{-1}\partial_x\log\rho_n
        +W'*\rho_n
        -W'*\rho_n
    \right)
\]
gives
\begin{equation}
\label{eq:shortening-fisher-bound}
    4\beta^{-2}
    \int_{\mathbb T}
        |\partial_x\sqrt{\rho_n}|^2\,\dd\lambda
    \leq
    2\mathcal D(q_n)+2\|W'\|_\infty^2.
\end{equation}
Thus \((\sqrt{\rho_n})\) is bounded in \(H^1(\mathbb T)\).  Since
\(\mathbb T\) is one-dimensional, after passing to a subsequence,
\(
    \sqrt{\rho_n}\longrightarrow f
    \qquad\text{uniformly on }\mathbb T.
\)
Hence
\[
    \rho_n\longrightarrow\rho:=f^2
    \qquad\text{uniformly},
\]
and \(\int\rho\,\dd\lambda=1\).
Let
\[
    J_n
    :=
    \beta^{-1}\partial_x\rho_n
    +
    \rho_n(W'*\rho_n)
    =
    \rho_n
    \left(
        \beta^{-1}\partial_x\log\rho_n
        +W'*\rho_n
    \right).
\]
By Cauchy--Schwarz,
\[
    \|J_n\|_{L^1(\lambda)}
    \leq
    \mathcal D(q_n)^{1/2}
    \longrightarrow0.
\]
Passing to the limit in the distributional identity defining \(J_n\)
gives
\[
    \beta^{-1}\partial_x\rho
    +
    \rho(W'*\rho)
    =0.
\]
Thus \(\rho\lambda\) is stationary.  In particular,
\(
    q:=[\rho\lambda]=\mathcal K_\alpha
\)
for some \(\alpha\).

Moreover, uniform convergence of the densities gives
\[
    \operatorname{Ent}(\rho_n\lambda\mid\lambda)
    \longrightarrow
    \operatorname{Ent}(\rho\lambda\mid\lambda),
\]
because \(r\mapsto r\log r\) is continuous on bounded intervals.
The interaction energies also converge, so
\[
    \mathcal F(q_n)\longrightarrow
    \mathcal F(q)=e_\alpha.
\]
Since \(q_n\to\mathcal K_\alpha\), we have
\(q_n\in U_\alpha\) for all sufficiently large \(n\).  But
\(q_n\notin L_{\alpha,\eta}\), and therefore
\[
    \mathcal F(q_n)>e_\alpha+\eta,
\]
contradicting \(\mathcal F(q_n)\to e_\alpha\).  This proves
\eqref{eq:shortening-dissipation-gap}.

\medskip
\noindent
\emph{Step 3: Positive-time energy regularization.}
The preceding energy argument cannot be started directly at an
arbitrary gate point, because \(\mathcal F\) is not bounded above on
a Wasserstein neighborhood.  We therefore prove that bounded action
gives a uniform positive-time energy bound.

There exists \(C_{\mathrm{reg}}(A)<\infty\) such that
\begin{equation}
\label{eq:shortening-positive-time-regularization}
    T\geq1,\quad I_{[0,T]}(\phi)\leq A
    \quad\Longrightarrow\quad
    \mathcal F(\phi_1)\leq C_{\mathrm{reg}}(A).
\end{equation}
Indeed, let \(u\) denote the minimal control representing \(\phi\):
\begin{equation}\label{eq:control}
    \partial_t\phi_t
    =
    \mathcal G(\phi_t)
    -
    \partial_x(\phi_tu_t),
    \qquad
    \frac{\beta}{4}
    \int_0^T\int_{\mathbb T}|u_t|^2\,\dd\phi_t\,\dd t
    =
    I_{[0,T]}(\phi).
\end{equation}
By the superposition principle, \(\phi_t\) is the time-\(t\)
marginal of a weak solution of
\[
    \dd X_t
    =
    -(W'*\phi_t)(X_t)\,\dd t
    +u_t(X_t)\,\dd t
    +\sqrt{2\beta^{-1}}\,\dd B_t.
\]
Let \(P\) denote its path law on \([0,1]\), and let \(R\) be the law
of the diffusion with the same initial distribution and with drift
\(
    b_t(x):=-(W'*\phi_t)(x),
\)
(i.e. with \(u_t=0\)).  By the Brownian variational representation and entropy duality \cite[Theorem 8.3 and Lemma 2.4(a)]{BudhirajaDupuis2019}, followed by contraction under the solution map, we have
\[
    \operatorname{Ent}(P\mid R)
    \leq
    \frac{\beta}{4}
    \int_0^1\int_{\mathbb T}|u_t|^2
    \,\dd\phi_t\,\dd t
    \leq A.
\]
By contraction of relative entropy under the endpoint map,
\[
    \operatorname{Ent}(\phi_1\mid r_1)\leq A,
\]
where \(r_1\) is the time-one marginal of \(R\).

Since
\[
    |b_t(x)|\leq\|W'\|_\infty,
\]
by the Aronson estimate for uniformly elliptic diffusions with bounded
measurable drift \cite{Aronson1967}, the time-one transition density of \(R\) is bounded above by a
constant \(C_{\mathrm{hk}}\) depending only on
\(\beta\), \(\|W'\|_\infty\), and \(\mathbb T\).  Thus
\[
    \frac{\dd r_1}{\dd\lambda}\leq C_{\mathrm{hk}},
\]
and consequently
\[
\begin{aligned}
    \operatorname{Ent}(\phi_1\mid\lambda)
    &=
    \operatorname{Ent}(\phi_1\mid r_1)
    +
    \int_{\mathbb T}
        \log\left(\frac{\dd r_1}{\dd\lambda}\right)
        \,\dd\phi_1                                      \\
    &\leq A+\log C_{\mathrm{hk}}.
\end{aligned}
\]
Since the interaction part of \(\mathcal F\) is bounded above by
\(\frac12\|W\|_\infty\), we may take
\begin{equation}\label{eq:creg}
    C_{\mathrm{reg}}(A)
    :=
    \beta^{-1}(A+\log C_{\mathrm{hk}})
    +\frac12\|W\|_\infty.
\end{equation}
This proves \eqref{eq:shortening-positive-time-regularization}.

\medskip
\noindent
\emph{Step 4: Chronological loop erasure and the duration bound.}
Choose \(\eta,\delta>0\) sufficiently small that
\begin{equation}
\label{eq:shortening-total-connector-cost}
    m(2\beta\eta+\delta)\leq\varepsilon.
\end{equation}
If \(T\leq1\), take \(\widetilde\phi=\phi\).  Suppose \(T>1\), and
leave the initial portion \(\phi|_{[0,1]}\) unchanged.

We now shorten \(\phi|_{[1,T]}\).  If this path never meets a core,
there is nothing to erase.  Otherwise, let \(a_1\) be its first
entrance time into the union of the cores, and let
\(\alpha_1\) be the corresponding label.  Define
\[
    b_1
    :=
    \sup\left\{
        t\in[a_1,T]:
        \phi_t\in L_{\alpha_1,\eta}
    \right\}.
\]
Since \(L_{\alpha_1,\eta}\) is closed,
\(\phi_{b_1}\in L_{\alpha_1,\eta}\).  Replace the portion
\(\phi|_{[a_1,b_1]}\) by the local connector from
\(\phi_{a_1}\) to \(\phi_{b_1}\) constructed in Step~1.

Continue from \(b_1\).  Let \(a_2\) be the first subsequent entrance
into any core and replace the path up to its last visit to that same
core.  Continuing inductively gives labels
\(
    \alpha_1,\ldots,\alpha_r
\)
that are all distinct: once the portion ending at the last visit to
\(L_{\alpha_\ell,\eta}\) has been replaced, that core is never
visited again.  Hence
\[
    r\leq m.
\]
The inserted connectors have total duration at most \(3m\) and,
by \eqref{eq:shortening-total-connector-cost}, total action at most
\(\varepsilon\).

The retained pieces of the original path form at most \(m+1\)
excursions through
\[
    \mathcal C_\eta
    :=
    \mathcal Q\setminus
    \bigcup_{\alpha=1}^mL_{\alpha,\eta}.
\]
Consider one such retained excursion
\[
    \psi:[a,b]\longrightarrow\mathcal C_\eta,
\]
allowing its endpoints to lie on the boundaries of the cores.  The
controlled free-energy identity with control $u$ (that is, using \eqref{eq:control} with $\psi$ in place of $\phi$) gives
\[
\begin{aligned}
    \mathcal F(\psi_b)-\mathcal F(\psi_a)
    &=
    -\int_a^b\mathcal D(\psi_t)\,\dd t
    +
    \int_a^b\int_{\mathbb T}
        \partial_x\Psi_{\psi_t}\,u_t
        \,\dd\psi_t\,\dd t                              \\
    &\leq
    -\frac12\int_a^b\mathcal D(\psi_t)\,\dd t
    +
    \frac12\int_a^b\int_{\mathbb T}|u_t|^2
        \,\dd\psi_t\,\dd t                              \\
    &=
    -\frac12\int_a^b\mathcal D(\psi_t)\,\dd t
    +
    \frac{2}{\beta}I_{[a,b]}(\psi).
\end{aligned}
\]
Using \(\mathcal D\geq d_\eta\) on \(\mathcal C_\eta\), obtained in Step 2, we obtain
\begin{equation}
\label{eq:shortening-one-excursion}
    \frac{d_\eta}{2}(b-a)
    \leq
    \mathcal F(\psi_a)-\mathcal F(\psi_b)
    +
    \frac{2}{\beta}I_{[a,b]}(\psi).
\end{equation}

The first retained excursion starts either at \(\phi_1\) or at a
core.  By Step~3,
\[
    \mathcal F(\phi_1)\leq C_{\mathrm{reg}}(A).
\]
Every other retained excursion starts in some core and therefore has
initial energy at most
\[
    C_{\mathrm{core}}
    :=
    \max_{1\leq\alpha\leq m}(e_\alpha+\eta).
\]
Finally,
\[
    F_{\min}:=\inf_{\mathcal Q}\mathcal F>-\infty.
\]
Consequently, if
\[
    C_A:=\max\{C_{\mathrm{reg}}(A),C_{\mathrm{core}}\},
\]
then summing \eqref{eq:shortening-one-excursion} over the at most
\(m+1\) retained excursions yields
\[
    T_{\mathrm{ext}}
    \leq
    \frac{2}{d_\eta}
    \left[
        (m+1)(C_A-F_{\min})
        +\frac{2A}{\beta}
    \right].
\]
The retained portions of the original path have total action at most
\(I_{[0,T]}(\phi)\), while the inserted connectors have total action
at most \(\varepsilon\).  Thus
\[
    I_{[0,\widetilde T]}(\widetilde\phi)
    \leq
    I_{[0,T]}(\phi)+\varepsilon.
\]
Moreover,
\[
    \widetilde T
    \leq
    1+3m+
    \frac{2}{d_\eta}
    \left[
        (m+1)(C_A-F_{\min})
        +\frac{2A}{\beta}
    \right]
    =:T_*(A,\varepsilon).
\]
By construction, the first and last points of the path have not been
changed.
\end{proof}

\begin{remark}[Comparison with Freidlin--Wentzell]\label{rem:FWcomp}
The replacement by paths of uniformly bounded duration in
Lemma~\ref{lem:uniform-shortening-arbitrary-gate-points}
has a finite-dimensional counterpart in
\cite[Chapter~6, \S1, Lemma~1.2]{FreidlinWentzell2012}.
There, a finite net of endpoint pairs and the small-action
connectors of their Lemma~1.1 give a uniform duration bound
for near-optimal paths. In the present setting, Wasserstein
closeness alone does not control free-energy or provide
analogous connectors between arbitrary nearby measures.
Our proof therefore is considerably more involved. It combines low-energy cores, positive-time
regularization, and dissipation estimates outside the cores
to obtain the required uniform duration bound.
\end{remark}

Step~1 of Lemma~\ref{lem:uniform-shortening-arbitrary-gate-points}
constructs low-action connectors between points in the same
low-energy core around a critical orbit $\mathcal K_\alpha$.
The following lemma shows that, when their endpoints are
noncritical, these connectors can be perturbed to avoid
$\mathcal K_\alpha$. Consequently, near-minimizing paths between
distinct minima $\mathcal M_i$ and $\mathcal M_j$ can be replaced
by paths admissible in
\eqref{eq:limiting-direct-cost-variational}, with arbitrarily
small additional action. In particular, this establishes \eqref{eq:DeltaeqV}.

\begin{lemma}[Removal of critical contacts]
\label{lem:critical-contact-removal}
Assume \eqref{eq:finite-critical-orbits}.
For any $\eta,\delta>0$, the neighborhoods and cores in Step~1
of Lemma~\ref{lem:uniform-shortening-arbitrary-gate-points}
can be chosen so that any two noncritical points of
$L_{\alpha,\eta}$ can be joined inside $V_\alpha$, in time $3$,
by a path of action at most $2\beta\eta+\delta$ whose interior
avoids every critical orbit.

Consequently, for every finite-action path $\phi$ between
distinct critical orbits and every $\varepsilon>0$, there is
a finite-duration path with the same endpoints, no critical
contacts in its interior, and action at most
$I(\phi)+\varepsilon$.
\end{lemma}

\begin{proof}
Use the three-piece connector from Step~1 of
Lemma~\ref{lem:uniform-shortening-arbitrary-gate-points},
such that the action of the affine interpolation is at most $\delta/2$.
Choose aligned measure-valued lifts of its pieces and denote
the resulting density by $\rho_t$, $0<t<3$.
It is spatially regular and piecewise continuously
differentiable in time, with a positive lower bound on each
compact subinterval of $(0,3)$.

Since $V_\alpha$ contains no other critical orbit, it suffices
to avoid $\mathcal K_\alpha$. The endpoints are noncritical,
so the set of critical-contact times is compactly contained
in $(0,3)$. Choose a nonnegative
$h\in C_c^\infty((0,3))$ that is positive on this set.
Let $f_1,f_2,f_3$ be smooth, mean-zero, orthonormal real
functions in $L^2(\lambda)$, and set
\[
 \rho_t^a
 :=\rho_t+h(t)\sum_{k=1}^3a_kf_k,
 \qquad a = (a_1,a_2,a_3)' \in\Rr^3.
\]
For sufficiently small $a$, this remains a probability density,
agrees with the unperturbed connector near its endpoints, and stays
inside $V_\alpha$.

Define
\[
 P\rho
 :=\bigl(\langle\rho,f_k\rangle_{L^2(\lambda)}\bigr)_{k=1}^3.
\]
If $\rho_t^a=R_\theta\bar\rho_\alpha$ at a time with $h(t)>0$,
then necessarily
\[
 a=\frac{P(R_\theta\bar\rho_\alpha)-P\rho_t}{h(t)}.
\]
The right-hand side is locally Lipschitz in the two variables
$(t,\theta)$ on $\{h>0\}\times\T$. Its image therefore has
three-dimensional Lebesgue measure zero.
Thus arbitrarily small choices of $a$ avoid
$\mathcal K_\alpha$ wherever $h>0$; where $h=0$, the original
path is unchanged and already avoids it.

On the support of $h$, the densities have a common positive
lower bound, while the perturbation and its time derivative
tend to zero in the spatial $C^1(\T)$ norm as $a\to0$.
The minimal-flux representation of the action \eqref{eq:weighted-Hminus1-flux} therefore gives
\[
 I_{[0,3]}(\rho^a\lambda)
 \le I_{[0,3]}(\rho\lambda)+o(1)
 \le 2\beta\eta+\frac{\delta}{2}+o(1),
\]
as $a \to 0$.
Choosing an admissible $a$ sufficiently small proves the
first assertion.

For the second assertion, first stop $\phi$ at its first
visit to the terminal critical orbit. Then discard the
segment before its last visit to the initial critical orbit. This removes interior contacts with both endpoint orbits without increasing action.

Choose the cores so that
$m(2\beta\eta+\delta)<\varepsilon$.
Apply the first-to-last-visit replacement procedure of
Step~4 of Lemma~\ref{lem:uniform-shortening-arbitrary-gate-points}
to the entire restricted path, using only the cores of the
intermediate critical orbits and the connectors constructed
above. The cut points are noncritical: continuity of the path
and the free-energy along it ensures that every critical contact lies in
the time-interior of the corresponding core-visit set.
There are at most $m-2$ replacements, and the new connectors
introduce no critical contacts. The resulting path has the
required endpoints and avoidance property, with total
additional action less than $\varepsilon$.
\end{proof}

To obtain the upper bound for gate transition probabilities, introduce auxiliary gates around every
critical component. Without loss of generality, assume that the components are labeled so that
\(\mathcal K_i=\mathcal M_i\) for \(1\leq i\leq J\), and write
\[
    \mathscr K:= \Crit(\F)/\T = \{\mathcal K_1,\ldots,\mathcal K_m\},
    \qquad \mathcal S:=\{J+1,\ldots,m\}.
\]
Choose pairwise disjoint nested gates
\begin{equation}
\label{eq:all-critical-gates}
    \mathcal K_\alpha\Subset\widehat G_{\alpha,\rho}^-
    \Subset\widehat G_{\alpha,\rho}^+,
    \qquad 1\leq\alpha\leq m,
\end{equation}
with closures shrinking to their respective critical components, and
with \(\widehat G_{i,\rho}^{\pm}=G_{i,\rho}^{\pm}\) for
\(1\leq i\leq J\).  As in the proof of
Lemma~\ref{lem:direct-gate-transition-lower}, the inner gates may be
chosen smaller after the outer gates have been fixed.  Set
\[
    \widehat K_{\alpha,\rho}:=
        \overline{\widehat G_{\alpha,\rho}^-},
    \qquad
    \widehat K_\rho:=\bigcup_{\alpha=1}^m\widehat K_{\alpha,\rho},
    \qquad
    \widehat\Gamma_{\alpha,\rho}:=
        \partial\widehat G_{\alpha,\rho}^+.
\]
The auxiliary entrance-point chain is defined by
\eqref{eq:gate-renewal-times}, with \(K_\rho\) and \(G_\alpha^+\)
replaced by \(\widehat K_\rho\) and
\(\widehat G_\alpha^+\).  Denote its entrance points by
\(\widehat Z_n\) and its labels by \(\widehat Y_n\).  Starting
at a minimum gate, the next visit of this chain to a label in
\(\{1,\ldots,J\}\), allowing a return to the initial label, is
exactly the next entrance of the original minimum-gate chain.

For distinct critical components define
\begin{equation}
\label{eq:all-critical-direct-cost}
\begin{aligned}
    \widehat\Delta_{\alpha\gamma}:=
    \inf\bigl\{&I_{[0,T]}(\phi):T>0,\quad
        \phi_0=\mathcal K_\alpha,\quad
        \phi_T=\mathcal K_\gamma,\\
        &\phi_t\notin\mathscr K
        \text{ for every }0<t<T\bigr\},
    \qquad \alpha\ne\gamma,
\end{aligned}
\end{equation}
and set \(\widehat\Delta_{\alpha\alpha}=0\).

These costs are finite, since smooth positive stationary
representatives can be joined by a finite-action affine path,
to which Lemma~\ref{lem:critical-contact-removal} applies.
The same lemma and inclusion of the admissible path classes give
\[
 \widehat\Delta_{\alpha\gamma}
 =\Vq(\mathcal K_\alpha,\mathcal K_\gamma),\qquad\alpha\ne\gamma.
\]

\begin{lemma}[Uniform stopped-excursion estimate]
\label{lem:uniform-stopped-excursion}
For every \(\varepsilon>0\), the nested gates can be chosen
sufficiently small that, for all sufficiently large \(N\) and every
\(i\ne j\), \(1 \le i,j \le J\), the probabilities \(p_{ij}^{N,\rho}(q)\) defined in
Lemma~\ref{lem:direct-gate-transition-lower} satisfy
\begin{align}
\label{eq:stopped-excursion-lower-detailed}
    \inf_{q\in K_{i,\rho}\cap\PN(\T)}p_{ij}^{N,\rho}(q)
    &\geq \exp\{-N(\Delta_{ij}+\varepsilon)\},\\
\label{eq:stopped-excursion-upper-detailed}
    \sup_{q\in K_{i,\rho}\cap\PN(\T)}p_{ij}^{N,\rho}(q)
    &\leq \exp\{-N(\Delta_{ij}-\varepsilon)\}.
\end{align}
\end{lemma}

\begin{proof}
The lower bound was obtained in
Lemma~\ref{lem:direct-gate-transition-lower}, with the same direct cost
$\Delta_{ij}$.  For the upper bound we first
prove the Freidlin--Wentzell estimate for the auxiliary chain and
then retain only its visits to minimum gates.

\smallskip
\noindent
\emph{Step 1: Duration outside all critical gates.}
Fix the gates and put
\[
    C_\rho:=\mathcal Q\setminus
        \bigcup_{\alpha=1}^m\widehat G_{\alpha,\rho}^-.
\]
This set is compact and contains no critical component. Choose the auxiliary neighborhoods $U_\alpha$ in Step~1 of
Lemma~\ref{lem:uniform-shortening-arbitrary-gate-points} so that
$\overline U_\alpha\subset\widehat G_{\alpha,\rho}^{-}$, and the corresponding low-energy cores
\[
    L_{\alpha}^{\rho}
    :=\{q\in\overline U_\alpha:
           \mathcal F(q)\le e_\alpha+\rho\}.
\]
Equation~\eqref{eq:shortening-dissipation-gap}
then gives
\[
    \mathcal D(q)\geq d_{\rho}>0,\qquad q\in C_\rho.
\]
Let \(\phi\) remain in \(C_\rho\) throughout \([0,T]\), with
\(T\geq1\) and \(A:=I_{[0,T]}(\phi)<\infty\).
By \eqref{eq:shortening-positive-time-regularization} and \eqref{eq:creg},
\(\mathcal F(\phi_1)\leq A/\beta+C_0\), where \(C_0\) is
independent of the initial measure.  The controlled energy
inequality used in \eqref{eq:shortening-one-excursion} yields
\[
\begin{aligned}
    \frac{d_\rho}{2}(T-1)
    &\leq \mathcal F(\phi_1)-\mathcal F(\phi_T)
          +\frac2\beta I_{[1,T]}(\phi)\\
    &\leq C_0-F_{\min}+\frac{3A}{\beta}.
\end{aligned}
\]
Consequently, for constants \(c_\rho>0\) and \(C_\rho'<\infty\),
\begin{equation}
\label{eq:all-critical-duration-action}
    \inf\{I_{[0,T]}(\phi):\phi_t\in C_\rho
                    \text{ for every }0\leq t\leq T\}
    \geq c_\rho T-C_\rho'.
\end{equation}
The path set in this display is closed.  Thus
Proposition~\ref{prop:uniform-LDP}\textup{(i)}, applied to its
rotation-invariant lift, shows that for every \(B<\infty\) there
is a fixed \(T_B<\infty\) (depending on $\rho$) such that
\begin{equation}
\label{eq:all-critical-duration-tail}
    \sup_{q\in C_\rho\cap\PN(\T)}
    \mathbb P_q(\widehat\tau^N>T_B)\leq e^{-NB},
    \qquad
    \widehat\tau^N:=\inf\{t\geq0:X_t^N\in\widehat K_\rho\},
\end{equation}
for all sufficiently large \(N\).  Iterating a fixed-block
estimate by the strong Markov property also gives the usual
exponential tail in the duration.  This is the analogue of
\cite[Chapter~6, Lemma~1.9]{FreidlinWentzell2012} needed below.

\smallskip
\noindent
\emph{Step 2: The auxiliary one-step estimates.}
Let \(\widehat p_{\alpha\gamma}^{N,\rho}(q)\) be the probability
that the next auxiliary entrance has label \(\gamma\), starting
from \(q\in\widehat K_{\alpha,\rho}\).  We show that, for every
\(\theta>0\), the gates can be chosen sufficiently small that
\begin{equation}
\label{eq:all-critical-transition-rates}
    e^{-N(\widehat\Delta_{\alpha\gamma}+\theta)}
    \leq\widehat p_{\alpha\gamma}^{N,\rho}(q)
    \leq e^{-N(\widehat\Delta_{\alpha\gamma}-\theta)},
    \qquad \alpha\ne\gamma,
\end{equation}
uniformly in \(q\) and for all sufficiently large \(N\).

For the lower bound, use the open-tube construction of
Lemma~\ref{lem:direct-gate-transition-lower}, with all critical
components in place of the minimum components.  Attraction is not
needed for its initial connector.  After the inner gate has been
chosen sufficiently small, the deterministic segment
\(q\mapsto \overline S_tq\), \(0\leq t\leq1\), stays in
\(\widehat G_\alpha^+\); positive-time smoothing makes
\(\overline S_1q\) uniformly close in \(C^1\), modulo rotation,
to the stationary representative of \(\mathcal K_\alpha\).
The affine connector in Step~1 of
Lemma~\ref{lem:uniform-shortening-arbitrary-gate-points} then joins
it to \(\mathcal K_\alpha\) at arbitrarily small action.
Append a fixed almost minimizing path for
\(\widehat\Delta_{\alpha\gamma}\).  Shrinking the inner gates
ensures that its tail avoids all other inner gates.  Open tubes and
a finite covering of the initial inner gate give the lower bound
in \eqref{eq:all-critical-transition-rates}.

For the upper bound, the strong Markov property at the first exit
from \(\widehat G_\alpha^+\) reduces the problem to starting points
\(z\in\widehat\Gamma_{\alpha,\rho}\).  We first record the
following uniform action bound.  For all sufficiently small gates,
every finite-action path starting on
\(\widehat\Gamma_{\alpha,\rho}\), ending in
\(\widehat K_{\gamma,\rho}\), and avoiding the open inner gates
\(\widehat G_{\ell,\rho}^-\), \(\ell\ne\gamma\), has action
at least \(\widehat\Delta_{\alpha\gamma}-\theta/4\).
Boundary contacts with the forbidden gates are allowed here.

To verify this assertion, suppose, to the contrary, that it fails
along a sequence of shrinking gates.  Then there exist
$\rho_n\downarrow0$ and finite-action paths
$\phi^n:[0,T_n]\to\mathcal Q$, satisfying the stated gate-avoidance
condition, such that
\[
    q_n:=\phi_0^n\in\widehat\Gamma_{\alpha,\rho_n},
    \qquad
    r_n:=\phi_{T_n}^n\in\widehat K_{\gamma,\rho_n},
    \qquad
    I_{[0,T_n]}(\phi^n)
    <\widehat\Delta_{\alpha\gamma}-\frac{\theta}{4}.
\]
Since the gates shrink to their respective critical components,
\[
    q_n\longrightarrow\mathcal K_\alpha,
    \qquad
    r_n\longrightarrow\mathcal K_\gamma
    \quad\text{in }\mathcal Q.
\]
The actions are uniformly bounded, but the durations $T_n$ need
not be.  Set $\varepsilon_0:=\theta/16$ and apply
Lemma~\ref{lem:uniform-shortening-arbitrary-gate-points}
with this fixed action error.  We obtain paths
$\psi^n:[0,S_n]\to\mathcal Q$ with the same endpoints and
\[
    S_n\le T_*<\infty,
    \qquad
    I_{[0,S_n]}(\psi^n)
    \le I_{[0,T_n]}(\phi^n)+\varepsilon_0
    <\widehat\Delta_{\alpha\gamma}-3\varepsilon_0,
\]
where $T_*$ is independent of $n$.

Extend each $\psi^n$ to $[0,T_*]$ by deterministic continuation:
\[
    \overline\psi_t^n:=
    \begin{cases}
        \psi_t^n,
        &0\le t\le S_n,\\[1mm]
        \overline S_{t-S_n}r_n,
        &S_n\le t\le T_*.
    \end{cases}
\]
The added segment has zero action.  Joint continuity of the
semiflow and stationarity of $\mathcal K_\gamma$ give
\[
    \sup_{0\le s\le T_*}
    d_{\mathcal Q}
    \bigl(\overline S_s r_n,\mathcal K_\gamma\bigr)
    \longrightarrow0.
\]
Thus
$\overline\psi_0^n\to\mathcal K_\alpha$ and
$\overline\psi_{T_*}^n\to\mathcal K_\gamma$.
Compactness of the finite-horizon action sublevel sets yields,
after passing to a subsequence,
\[
    \overline\psi^n\longrightarrow\psi
    \quad\text{in }C([0,T_*];\mathcal Q).
\]
By lower semicontinuity,
\[
    \psi_0=\mathcal K_\alpha,
    \qquad
    \psi_{T_*}=\mathcal K_\gamma,
    \qquad
    I_{[0,T_*]}(\psi)
    \le\liminf_{n\to\infty}
         I_{[0,T_*]}(\overline\psi^n)
    \le\widehat\Delta_{\alpha\gamma}-3\varepsilon_0.
\]

Since $\psi$ connects $\mathcal K_\alpha$ to $\mathcal K_\gamma$,
the previously established equality between the direct cost
and the quasipotential gives
\[
 \widehat\Delta_{\alpha\gamma}
 =\Vq(\mathcal K_\alpha,\mathcal K_\gamma)
 \le I_{[0,T_*]}(\psi)
 \le\widehat\Delta_{\alpha\gamma}-3\varepsilon_0,
\]
which is impossible since $\varepsilon_0>0$.

This contradiction proves the asserted uniform action bound.

For $T>0$, let $F_{\alpha\gamma}^{\rho,T}$ be the set of paths
$\phi\in C([0,T];\mathcal Q)$ such that
\(
    \phi_0\in\widehat\Gamma_{\alpha,\rho}
\)
and there exists $t\in[0,T]$ satisfying
\[
    \phi_t\in\widehat K_{\gamma,\rho},
    \qquad
    \phi_s\notin\widehat G_{\ell,\rho}^-
    \quad\text{for every }0\le s\le t
    \text{ and every }\ell\ne\gamma.
\]
There is no restriction on the path after the witnessing time $t$.
Boundary contacts with the forbidden inner gates are allowed.
The set $F_{\alpha\gamma}^{\rho,T}$ is closed in the uniform
path topology: along a uniformly convergent sequence, take a
convergent subsequence of the witnessing times and pass to the
limit using closedness of the initial boundary, the closure of the target gate,
and the complements of the forbidden open gates.

By the preceding action bound, applied up to the witnessing
time, and nonnegativity of the remaining action, for any $T>0$,
\[
    \inf_{\phi\in F_{\alpha\gamma}^{\rho,T}}
    I_{[0,T]}(\phi)
    \ge \widehat\Delta_{\alpha\gamma}-\frac{\theta}{4}.
\]

Now, we apply a fixed-time
decomposition as in
\cite[Chapter~6, proof of Lemma~2.1]{FreidlinWentzell2012}. Choose \(B>1+\max_{\alpha\ne\gamma}
\widehat\Delta_{\alpha\gamma}\) and use \(T=T_B\) from
\eqref{eq:all-critical-duration-tail}.  A successful transition to $\widehat K_{\gamma,\rho}$
starting on \(\widehat\Gamma_{\alpha,\rho}\) either lasts longer
than \(T_B\), or belongs to \(F_{\alpha\gamma}^{\rho,T_B}\).
Hence the uniform LDP upper bound gives
\begin{equation}
\label{eq:FW-auxiliary-split}
\begin{aligned}
    \sup_{z\in\widehat\Gamma_{\alpha,\rho}\cap\PN(\T)}
    \mathbb P_z(X_{\widehat\tau^N}^N\in\widehat K_{\gamma,\rho})
    &\leq e^{-NB}
       +\sup_z\mathbb P_z(X^N_\cdot\in F_{\alpha\gamma}^{\rho,T_B})\\
    &\leq e^{-NB}
       +e^{-N(\widehat\Delta_{\alpha\gamma}-\theta/2)}\\
    &\leq e^{-N(\widehat\Delta_{\alpha\gamma}-\theta)}
\end{aligned}
\end{equation}
for all sufficiently large \(N\). This completes the proof of \eqref{eq:all-critical-transition-rates}.

\smallskip
\noindent
\emph{Step 3: The costs after retaining only minimum gates.}
Let \(\mathfrak S_{ij}\) be the finite collection of simple routes
\[
    \mathfrak s=(a_0,\ldots,a_r),\qquad
    a_0=i,\quad a_r=j,\quad
    a_1,\ldots,a_{r-1}\in\mathcal S,
\]
where all labels are distinct; the direct route \((i,j)\) is
included. Then
\begin{equation}
\label{eq:critical-trace-cost}
    \Delta_{ij}
    =\min_{\mathfrak s\in\mathfrak S_{ij}}
        \sum_{\ell=1}^{r}
            \widehat\Delta_{a_{\ell-1}a_\ell}.
\end{equation}
Indeed, by \eqref{eq:DeltaeqV}, the identities
\(\widehat\Delta_{\alpha\gamma}
=\Vq(\mathcal K_\alpha,\mathcal K_\gamma)\), and the triangle
inequality for \(\Vq\), every such route satisfies
\[
    \Delta_{ij}
    =\Vq(\mathcal M_i,\mathcal M_j)
    \le \sum_{\ell=1}^{r}
        \widehat\Delta_{a_{\ell-1}a_\ell}.
\]
The direct route \((i,j)\) attains equality, proving
\eqref{eq:critical-trace-cost}.

Moreover, Step~3 of the proof of
Lemma~\ref{lem:singleton-exit-cost} shows that, for every
\(\alpha\in\mathcal S\), there is a minimum label
\(k(\alpha)\in\{1,\ldots,J\}\) such that
\[
    \widehat\Delta_{\alpha k(\alpha)}
    =\Vq(\mathcal K_\alpha,\mathcal M_{k(\alpha)})=0.
\]

\smallskip
\noindent
\emph{Step 4: The upper bound for the original chain.}
Since \(\widehat\Delta_{\alpha k(\alpha)}=0\),
the lower bound in \eqref{eq:all-critical-transition-rates}
shows that, from any nonminimum auxiliary entrance point,
the conditional probability that the next auxiliary entrance
has a minimum label is at least \(e^{-N\theta}\),
uniformly over the past.
If \(L\) denotes the number of auxiliary transitions until a
minimum is reached, successive conditioning gives
\[
    \mathbb P_q(L>\ell)
    \le (1-e^{-N\theta})^\ell,
    \qquad \ell\ge0.
\]
Consequently,
\begin{equation}
\label{eq:critical-trace-length}
    \sup_{\substack{\alpha\in\mathcal S\\
               q\in\widehat K_{\alpha,\rho}\cap\PN(\T)}}
        \mathbb E_q L
    \le M_N:=e^{N\theta}.
\end{equation}
If \(\mathcal S\) is empty, no intermediate transitions are
needed and we take \(M_N=1\).

On a successful \(i\)-to-\(j\) transition of the minimum-gate
chain, erase loops in the intervening auxiliary label sequence.
This gives a route
\(\mathfrak s=(a_0,\ldots,a_r)\in\mathfrak S_{ij}\), with
\(a_0=i\), \(a_r=j\), and all intermediate labels nonminimum.
Let \((\mathcal G_n)\) denote the natural filtration of the
auxiliary entrance-point chain \((\widehat Z_n)\), and set
\[
    \tau_{\min}
    :=\inf\{n\ge1:\widehat Y_n\in\{1,\ldots,J\}\}.
\]
Since a return to \(i\) also ends the excursion, a successful
excursion from $i$ to $j$ cannot revisit \(i\).  Thus loop erasure preserves
the first auxiliary transition, which must be \(a_0\to a_1\).

The preceding
bound on $L$ gives
\[
    \sup_{z:\,z\text{ lies in a nonminimum gate}}
    \mathbb E_z\tau_{\min}\le M_N.
\]
For an auxiliary edge \(a\to b\), with \(a\) nonminimum,
the above estimate, Markov property and the upper bound in \eqref{eq:all-critical-transition-rates} give, for all sufficiently large $N$ and every starting state $z$ in a nonminimum gate,
\begin{align}\label{eq:oppor}
 &\mathbb P_z\bigl(
   \exists\,0\le n<\tau_{\min}:
   (\widehat Y_n,\widehat Y_{n+1})=(a,b)\bigr)\notag\\
 &\quad\le
 \sum_{n\ge0}\mathbb E_z\!\left[
   \mathbf 1_{\{n<\tau_{\min},\,\widehat Y_n=a\}}
   \mathbb P_z(\widehat Y_{n+1}=b\mid\mathcal G_n)
 \right]\notag\\
 &\quad\le
 e^{-N(\widehat\Delta_{ab}-\theta)}
 \mathbb E_z\tau_{\min}
 \le M_Ne^{-N(\widehat\Delta_{ab}-\theta)}.
\end{align}
This estimate is uniform over starting states in nonminimum
gates, so it remains valid after restarting at any auxiliary
transition into such a gate.

Fix the route \(\mathfrak s\), and put
\[
    b_\ell
    :=e^{-N(\widehat\Delta_{a_{\ell-1}a_\ell}-\theta)},
    \qquad 1\le\ell\le r.
\]
Define successive first-occurrence indices by
\[
    S_1=
    \begin{cases}
      1,&\widehat Y_1=a_1,\\
      \infty,&\widehat Y_1\ne a_1,
    \end{cases}
\]
and, for \(2\le\ell\le r\),
\[
    S_\ell
    :=\inf\bigl\{
       n>S_{\ell-1}:n\le\tau_{\min},\
       (\widehat Y_{n-1},\widehat Y_n)=(a_{\ell-1},a_\ell)
    \bigr\},
\]
with \(\inf\varnothing=\infty\).
These are stopping times under the filtration \((\mathcal G_n)\).
On \(\{S_{\ell-1}<\infty\}\), for \(2\le\ell\le r\),
the entrance point \(\widehat Z_{S_{\ell-1}}\) lies in the
nonminimum gate labelled \(a_{\ell-1}\), and
\(S_{\ell-1}<\tau_{\min}\).
The strong Markov property at the corresponding auxiliary
transition time, together with the uniform
estimate \eqref{eq:oppor}, therefore gives
\[
    \mathbb P_q(S_\ell<\infty)
    \le M_Nb_\ell\,
        \mathbb P_q(S_{\ell-1}<\infty).
\]
Also, \(\mathbb P_q(S_1<\infty)\le b_1\).

The edges retained by loop erasure occur in their route order
in the original auxiliary sequence.  Consequently, if the
loop-erased route is \(\mathfrak s\), each of the above
successive searches succeeds.  Iterating the preceding
inequality yields
\begin{align*}
    \mathbb P_q(\text{the loop-erased route is }\mathfrak s)
    &\le \mathbb P_q(S_r<\infty)
    \le b_1\prod_{\ell=2}^r(M_Nb_\ell)\\
    &=M_N^{r-1}
    \exp\left\{-N\left(
        \sum_{\ell=1}^r\widehat\Delta_{a_{\ell-1}a_\ell}
        -r\theta\right)\right\}.
\end{align*}
By \eqref{eq:critical-trace-cost},
\[
    \sum_{\ell=1}^r
    \widehat\Delta_{a_{\ell-1}a_\ell}\ge\Delta_{ij}.
\]
Since \(r\le m-1\), the preceding bound on \(M_N\) gives,
uniformly over these route lengths,
\[
    M_N^{r-1}e^{Nr\theta}
    \le C_m e^{2Nm\theta}.
\]
Summing over the finite set \(\mathfrak S_{ij}\), and
enlarging \(C_m\) if necessary, therefore yields
\[
    \sup_{q\in K_{i,\rho}\cap\PN(\T)}
    p_{ij}^{N,\rho}(q)
    \le C_m e^{-N(\Delta_{ij}-2m\theta)}.
\]
Choose \(\theta \in (0, \varepsilon/(4m))\),
choose the gates accordingly, and increase \(N_0\) so that
\(C_m\le e^{N\varepsilon/2}\) for \(N\ge N_0\).
This proves \eqref{eq:stopped-excursion-upper-detailed}.
\end{proof}

\begin{lemma}[Gate-transition reduction]
\label{lem:gate-transition-reduction}
For every \(\varepsilon>0\), the gates can be chosen sufficiently
small that, for all sufficiently large \(N\) and every \(i\ne j\), $1 \le i,j \le J$,
\begin{equation}
\label{eq:uniform-gate-transition-rates}
\begin{aligned}
    \e^{-N(\Delta_{ij}+\varepsilon)}
    &\leq
    \mathbb P\left(
        Y_{n+1}=j
        \,\middle|\,
        \mathcal F_{\tau_n}
    \right)
    \leq
    \e^{-N(\Delta_{ij}-\varepsilon)}
\end{aligned}
\end{equation}
on \(\{Y_n=i\}\), uniformly in \(n\) and in the entrance point
\(X_{\tau_n}^N\in K_{i,\rho}\cap\PN(\T)\).
\end{lemma}

\begin{proof}
Lemma~\ref{lem:uniform-stopped-excursion} gives both bounds in
\eqref{eq:uniform-gate-transition-rates}.  Applying the strong
Markov property at \(\tau_n\) makes them uniform conditionally on
the entire past and over the entrance point in \(K_{i,\rho}\).
\end{proof}

\section{Two-sided cycle estimates in physical time}
\label{sec:two-sided-cycles}

In this section, we use the transition estimates for the gate label process obtained in Section~\ref{sec:long-time} to obtain physical time estimates for the empirical measure process $[\mu^N]$ to exit the neighborhood of minima indexed by a given cycle. Through Lemma~\ref{lem:direct-graph-heights} and Lemma~\ref{lem:conditional-FW}, we obtain upper and lower bounds on the cycle-exit time for the (discrete-time) gate label process.
Lemma~\ref{lem:physical-gate-clock} bounds the physical time elapsed during the time between successive outer-gate exits by $e^{\kappa N}$ for any prescribed $\kappa>0$,
after choosing sufficiently small fixed gates, simultaneously over
exponentially many excursions.
Lemma~\ref{lem:fast-gate-clock} bounds the number of excursions
completed by an exponentially growing observation time. The estimates in both Lemma~\ref{lem:physical-gate-clock} and Lemma~\ref{lem:fast-gate-clock} hold with superexponentially small failure
probabilities. All these estimates are then combined to obtain Theorem~\ref{prop:two-sided-physical-cycles}, our main result on physical cycle-exit times.

For \(i\ne j\), let \(\Omega_{ij}\) denote the set of finite
sequences \(\omega=(i_0,\ldots,i_r)\), with \(r\ge1\),
\(i_0=i\), \(i_r=j\), and consecutive labels distinct. The \emph{direct-cost graph} is the complete undirected graph
with vertex set \(\{1,\ldots,J\}\), where vertex \(i\)
represents the minimum orbit \(\mathcal M_i\).
The edge \(\{i,j\}\), \(i\ne j\), has height
\[
 E_{ij}:=H_i+\Delta_{ij}
       =H_j+\Delta_{ji}=E_{ji},
 \qquad H_i=\beta f_i.
\]
To see the equality above, we use the observation that reversal of a
direct gate-to-gate path preserves the avoidance constraint in \eqref{eq:limiting-direct-cost-variational}, along with \eqref{eq:action-reversal}.

A graph route is a finite sequence
\(\omega=(i_0,\ldots,i_r)\) with consecutive labels distinct.

\begin{lemma}[Communication heights of the direct-cost graph]
\label{lem:direct-graph-heights}
For $i \neq j$, one has
\begin{equation}\label{eq:graph-communication-height}
 \min_{\omega\in\Omega_{ij}}\max_{(k,\ell)\in\omega}E_{k\ell}
 =\beta\Phi_{ij}.
\end{equation}
\end{lemma}

\begin{proof}
The action--energy inequality \eqref{eq:action-lower-energy} gives
$\beta\Phi_{k\ell}\le H_k+\Delta_{k\ell}$ on every edge.
Concatenation and the ultrametric inequality therefore give the lower
bound in \eqref{eq:graph-communication-height}.

For the reverse bound, take a continuous path from $\mathcal M_i$ to
$\mathcal M_j$ of free-energy height at most $\Phi_{ij}+a$, $a>0$.
Apply the deterministic smoothing map $S_s$, with a fixed $s>0$, to
the whole path.  Its endpoints are unchanged, its energy does not
increase, and its image is compact in the quotient $C^2$ topology,
with a common positive lower bound for the densities.
On this compact set, two sufficiently close densities, after a rotational alignment,
can be joined with arbitrarily small action.  Indeed, as in the proof of Lemma~\ref{lem:zero-local-exit-cost}, affine
interpolation over duration $u$ has action at most
$C(\|\rho-\widetilde\rho\|_{C^1}^2/u+u)$; take
$u=\|\rho-\widetilde\rho\|_{C^1}$.

Fix $b>0$ and choose a finite mesh $q_0,\ldots,q_m$ along the smoothed
path so that consecutive points admit connectors of action at most $b$.
The small-action descent construction in Step~3 of the proof of
Lemma~\ref{lem:singleton-exit-cost} gives, from every $q_k$, a path
$\alpha_k$ to some minimum $\mathcal M_{j_k}$ with action at most $b$.
At the endpoints use the constant paths, so $j_0=i$ and $j_m=j$.
Reverse \(\alpha_k\), append the connector from \(q_k\) to
\(q_{k+1}\), and then append \(\alpha_{k+1}\). Denote the
resulting path from \(\mathcal M_{j_k}\) to
\(\mathcal M_{j_{k+1}}\) by \(\phi_k\).
Whenever \(j_k\ne j_{k+1}\), \eqref{eq:DeltaeqV},
additivity of the action, and action reversal give
\[
\begin{aligned}
    E_{j_kj_{k+1}}
    &=H_{j_k}
      +\Vq(\mathcal M_{j_k},\mathcal M_{j_{k+1}})
    \le H_{j_k}+I(\phi_k)\\
    &\le \beta\F(q_k)
         +I(\alpha_k)+b+I(\alpha_{k+1})
    \le \beta(\Phi_{ij}+a)+3b.
\end{aligned}
\]
Deleting consecutive repeated labels from
\((j_0,\ldots,j_m)\) produces a route
\(\omega\in\Omega_{ij}\), every edge of which has height
at most \(\beta(\Phi_{ij}+a)+3b\).
Letting \(a,b\downarrow0\) proves the reverse inequality
in \eqref{eq:graph-communication-height}.
\end{proof}

\begin{lemma}[Finite-state comparison for conditional gate kernels]
\label{lem:conditional-FW}
The following statement concerns the entrance-point chain
$Z_n$ and its gate label process $Y_n$ constructed above.
Set \(\mathcal H_n:=\mathcal F_{\tau_n}\). For every nonempty
\(A\subseteq\{1,\ldots,J\}\) and every \(n\ge0\), define
\[
 \widehat\tau_A^{(n)}
 :=\inf\{k\ge0:Y_{n+k}\in A\},
 \qquad
 \widehat\theta_C^{(n)}
 :=\widehat\tau_{C^c}^{(n)},
\]
where complements are taken in \(\{1,\ldots,J\}\).
When \(n=0\), we omit the superscript.

For every \(u>0\) and \(\delta>0\), there exist
\(
 \alpha_0>0, c>0, C_0<\infty, N_0\in\mathbb N,
\)
depending only on \(u,\delta\),
such that the following implication holds for every \(N\ge N_0\)
and every \(\alpha\in(0,\alpha_0]\).

Suppose that, for every \(n\ge0\) and every \(i\ne j\),
\[
 e^{-N(\Delta_{ij}+\alpha)}
 \le
 \mathbb P(Y_{n+1}=j\mid\mathcal H_n)
 \le
 e^{-N(\Delta_{ij}-\alpha)}
 \qquad\text{almost surely on }\{Y_n=i\}.
\]
Then the following estimates hold for every \(n\ge0\):
\begin{enumerate}[label=(\roman*)]
\item
For every proper cycle \(C\),
\[
 \mathbb E\!\left[
   \widehat\theta_C^{(n)}\mid\mathcal H_n
 \right]
 \le e^{N(d(C)+u)}
 \qquad\text{almost surely on }\{Y_n\in C\}.
\]

\item
For every proper cycle \(C\), almost surely on \(\{Y_n\in C\}\),
\begin{equation}\label{eq:gate-two-sided}
\begin{aligned}
 \mathbb P\!\left(
   \widehat\theta_C^{(n)}>e^{N(d(C)+\delta)}
   \,\middle|\,\mathcal H_n
 \right)
 &\le \exp\{-e^{cN}\},\\
 \mathbb P\!\left(
   \widehat\theta_C^{(n)}<e^{N(d(C)-\delta)}
   \,\middle|\,\mathcal H_n
 \right)
 &\le C_0e^{-cN}
 \qquad\text{if }\delta<d(C).
\end{aligned}
\end{equation}

\item
For every \(i\) such that \(A_i=\{j:f_j<f_i\}\) is nonempty,
recall
\begin{equation*}
 h_i=\min_{j\in A_i}\Phi_{ij},
 \qquad
 D_i=\beta(h_i-f_i).
\end{equation*}
Then
\[
 \mathbb E\!\left[
   \widehat\tau_{A_i}^{(n)}\mid\mathcal H_n
 \right]
 \le e^{N(D_i+u)}
 \qquad\text{almost surely on }\{Y_n=i\}.
\]
\end{enumerate}
The same constants apply to every renewal index, cycle, and label,
and to every initial law for which the conditional transition bounds
hold.
\end{lemma}

\begin{proof}
Fix \(u,\delta>0\).

\smallskip
\noindent\emph{Markov chain hitting time estimates.}
Consider the reversible matrix
\[
 Q_N(i,j)=(2J)^{-1}e^{-N\Delta_{ij}}\quad(i\ne j),
 \qquad
 Q_N(i,i)=1-\sum_{j\ne i}Q_N(i,j),
\]
whose invariant law is
\(\varpi_N(i)=Z_{H,N}^{-1}e^{-NH_i}\), where $Z_{H,N}$ is the associated normalizing factor.
Write \(\tau_A\) for the hitting time of \(A\) by this reference chain, and write the law and expectation of this chain starting from $x \in \{1,\dots,J\}$ respectively as $\mathbb P_x^{Q_N}$ and $\mathbb E_x^{Q_N}$.

Lemma~\ref{lem:direct-graph-heights} and
\cite[Propositions~3.10 and~3.15]{CirilloNardiSohier2015}
identify the cycles of \(Q_N\) and their depths with ours.
By \cite[Proposition~2.5, equation~(2.26), and
Proposition~3.11]{CirilloNardiSohier2015},
\[
 \lim_{N\to\infty}\frac1N
 \log\mathbb E_x^{Q_N}\tau_{C^c}=d(C)
\]
for every proper cycle \(C\) and every \(x\in C\).
Since there are finitely many such pairs \((C,x)\), there exists
\(N_1\) such that
\begin{equation}\label{eq:benu1}
 \mathbb E_x^{Q_N}\tau_{C^c}\le e^{N(d(C)+\eta)}
\end{equation}
for every \(N\ge N_1\), uniformly over these pairs.

Moreover, \cite[Proposition~3.18, item~2]{CirilloNardiSohier2015}
gives constants \(\gamma>0\) and \(K<\infty\) such that,
after increasing \(N_1\),
\begin{equation}\label{eq:bana}
 \mathbb P_x^{Q_N}
 \bigl(\tau_{C^c}\le e^{N(d(C)-\delta/2)}\bigr)
 \le K e^{-\gamma N}
\end{equation}
for every \(N\ge N_1\), every proper cycle \(C\) with
\(\delta<d(C)\), and every \(x\in C\).

Let \((X_n)\) be the reference chain with transition kernel \(Q_N\),
and write \(\tau_i^+:=\inf\{n\ge1:X_n=i\}\). Set
\[
 \sigma:=\tau_{A_i}\wedge\tau_i^+,
 \qquad
 p:=\mathbb P_i^{Q_N}(\tau_{A_i}<\tau_i^+).
\]
The strong Markov property at \(\sigma\) gives
\[
 \mathbb E_i^{Q_N}\tau_{A_i}
 =\mathbb E_i^{Q_N}\sigma
   +(1-p)\mathbb E_i^{Q_N}\tau_{A_i}.
\]
Every time before \(\sigma\) is spent outside \(A_i\). Hence the
occupation formula
\cite[Lemma~10.5 and Remark~10.6]{LevinPeres2017},
applied to the return time \(\tau_i^+\), yields
\[
 p\,\mathbb E_i^{Q_N}\tau_{A_i}
 =\mathbb E_i^{Q_N}\sigma
 \le
 \mathbb E_i^{Q_N}\!\left[
   \sum_{n=0}^{\tau_i^+-1}\mathbf 1_{\{X_n\notin A_i\}}
 \right]
 =\frac{\varpi_N(A_i^c)}{\varpi_N(i)}.
\]
Therefore,
\[
 \mathbb E_i^{Q_N}\tau_{A_i}
 \le \frac{\varpi_N(A_i^c)}
          {\operatorname{cap}_N(i,A_i)},
 \qquad
 \operatorname{cap}_N(i,A_i)
 :=\varpi_N(i)\mathbb P_i^{Q_N}(\tau_{A_i}<\tau_i^+).
\]

By Lemma~\ref{lem:direct-graph-heights}, there is a route from
\(i\) to \(A_i\) whose maximal edge elevation is \(\beta h_i\).
Stopping at its first visit to \(A_i\) and erasing loops gives
a simple path \(\omega=(i_0,\ldots,i_m)\), with \(i_0=i\),
\(i_m\in A_i\), and \(i_r\notin A_i\) for \(r<m\).
Since its vertices are distinct, \(m\le J-1\), and every edge
satisfies \(E_{i_{r-1}i_r}\le\beta h_i\).
The conductances of the reversible reference chain are
\[
 c_N(k,\ell)
 :=\varpi_N(k)Q_N(k,\ell)
 =\frac{e^{-N(H_k+\Delta_{k\ell})}}{2JZ_{H,N}}
 =\frac{e^{-NE_{k\ell}}}{2JZ_{H,N}},
 \qquad k\ne\ell.
\]
Thus, by the series law for resistances, the total resistance along this path is bounded by
\[
 \sum_{r=1}^{m}\frac{1}{c_N(i_{r-1},i_r)}
 \le 2JmZ_{H,N}e^{N\beta h_i}
 \le 2J^2Z_{H,N}e^{N\beta h_i}.
\]
By Rayleigh's monotonicity principle \cite[Theorem~9.12 and Corollary~9.13]{LevinPeres2017}, retaining only the edges
of this path can only decrease the effective conductance.
Thus,
\[
 \operatorname{cap}_N(i,A_i)
 \ge
 \left(
   \sum_{r=1}^{m}\frac{1}{c_N(i_{r-1},i_r)}
 \right)^{-1}
 \ge \frac{e^{-N\beta h_i}}{2J^2Z_{H,N}}.
\]
Since \(\varpi_N(A_i^c)\le JZ_{H,N}^{-1}e^{-N\beta f_i}\),
\begin{equation}\label{eq:reference-descent-mean}
 \mathbb E_i^{Q_N}\tau_{A_i}\le 2J^3e^{ND_i}.
\end{equation}
Increasing \(N_1\), this is at most \(e^{N(D_i+\eta)}\),
simultaneously for all \(i\) with \(A_i\ne\varnothing\).

\smallskip
\noindent\emph{Comparison estimates for the gate label process.}
We transfer the reference-chain estimates to the
conditional label process using the general-state-space
comparison lemmas of Freidlin--Wentzell
\cite[Chapter~6, \S3, Lemmas~3.3 and~3.4]{FreidlinWentzell2012}.
Recall that the entrance-point chain
\(Z_n=X_{\tau_n}^N\) is Markov, with its state space partitioned
by the gates \(K_{i,\rho}\). Denote its transition kernel by
\(\mathsf P_N\). For $\alpha_0>0$ and \(0<\alpha\le\alpha_0\), the uniform gate
transition bounds give
\[
 R_N^{-1}Q_N(i,j)
 \le \mathsf P_N(z,K_{j,\rho})
 \le R_NQ_N(i,j),
 \qquad z\in K_{i,\rho},\quad i\ne j,
 \qquad R_N:=2J e^{N\alpha_0},
\]
where \(z\) ranges over the finite-\(N\) states in the indicated gate.

Fix a nonempty proper set of target labels \(A\), and put
\(r:=J-|A|\). Recall that
\[
 \widehat\tau_A
 =\inf\left\{k\ge0:
       Z_k\in\bigcup_{j\in A}K_{j,\rho}\right\}.
\]
Here \(\mathbb E_z\) denotes expectation for the entrance-point
chain started at \(Z_0=z\), whereas
\(\mathbb E_i^{Q_N}\tau_A\) is the corresponding mean hitting
time for the finite reference chain started at label \(i\).

Applying \cite[Chapter~6, Lemma~3.4]{FreidlinWentzell2012}
with \(p_{ij}=Q_N(i,j)\) and \(a=R_N\) gives
\[
 R_N^{-4^r}\,\mathbb E_i^{Q_N}\tau_A
 \le
 \mathbb E_z\widehat\tau_A
 \le
 R_N^{4^r}\,\mathbb E_i^{Q_N}\tau_A,
 \qquad i\notin A,\quad z\in K_{i,\rho},
\]
uniformly over the finite-\(N\) entrance points in the indicated
gate. Since \(r\le J\), we may use the common factor
\(B_N:=R_N^{4^J}\), obtaining
\[
 \sup_{z\in K_{i,\rho}\cap\PN(\T)}
 \mathbb E_z\widehat\tau_A
 \le B_N\,\mathbb E_i^{Q_N}\tau_A,
 \qquad i\notin A.
\]

To control early hitting, we compare the Laplace transforms of
the hitting times. Fix \(m>0\) and independently
kill each chain before every transition with probability \(1-e^{-1/m}\).
Equivalently, introduce an independent geometric killing step
\(\zeta\) satisfying
\[
 \mathbb P(\zeta>k)=e^{-k/m},\qquad k\ge0.
\]
Conditioning on the hitting time gives
\[
\begin{aligned}
 \mathbb P_z(\widehat\tau_A<\zeta)
 =\mathbb E_z[e^{-\widehat\tau_A/m}], \qquad
 \mathbb P_i^{Q_N}(\tau_A<\zeta)
 =\mathbb E_i^{Q_N}[e^{-\tau_A/m}].
\end{aligned}
\]
Thus each Laplace transform is the probability of reaching
the target before being killed.

Apply \cite[Chapter~6, Lemma~3.3]{FreidlinWentzell2012}
to the augmented chains, stopping on reaching either a target
gate or the cemetery state \(\partial\).
The target labels in that lemma are \(A\cup\{\partial\}\);
summing its hitting-location bounds over \(j\in A\) gives
the probability of reaching \(A\) before \(\partial\).
The killing probabilities agree for the two chains, while the
other off-diagonal probabilities retain their multiplicative
comparison by \(R_N\).
There are still \(r=J-|A|\) non-target pieces. Consequently,
\[
 R_N^{-4^r}\,
 \mathbb E_i^{Q_N}[e^{-\tau_A/m}]
 \le
 \mathbb E_z[e^{-\widehat\tau_A/m}]
 \le
 R_N^{4^r}\,
 \mathbb E_i^{Q_N}[e^{-\tau_A/m}],
\]
uniformly for \(m>0\), \(i\notin A\), and all finite-\(N\)
entrance points \(z\in K_{i,\rho}\).
In particular, the upper bound holds with the common factor
\(B_N=R_N^{4^J}\).

Both estimates are uniform over the entrance point in each gate.
The strong Markov property at \(\tau_n\) therefore gives,
almost surely on \(\{Y_n\notin A\}\),
\begin{equation}\label{eq:comparison}
\begin{aligned}
 \mathbb E[
   \widehat\tau_A^{(n)}\mid\mathcal H_n]
 &\le B_N\,\mathbb E_{Y_n}^{Q_N}\tau_A,\\
 \mathbb E[
   e^{-\widehat\tau_A^{(n)}/m}\mid\mathcal H_n]
 &\le B_N\,\mathbb E_{Y_n}^{Q_N}[e^{-\tau_A/m}],
 \qquad m>0.
\end{aligned}
\end{equation}

\smallskip
\noindent\emph{Choice of the error and the mean bounds.}
Set \(\eta=\frac14\min\{u,\delta\}\), and let
\(K,\gamma>0\) be the constants in the reference-chain
early-exit estimate in \eqref{eq:bana}. If no cycle has depth greater
than \(\delta\), the lower-tail assertion is vacuous and
we may take \(K=\gamma=1\).
Choose \(\alpha_0>0\) so that
\[
 2(4^J)\alpha_0<\min\{\eta,\gamma/2\}.
\]
For \(N\) sufficiently large,
\[
 B_N=(2Je^{N\alpha_0})^{4^J}
 \le e^{N\min\{\eta,\gamma/2\}}.
\]
The first comparison bound in \eqref{eq:comparison}, together with \eqref{eq:benu1} and \eqref{eq:reference-descent-mean}, now give
\[
\begin{aligned}
 \mathbb E[
   \widehat\theta_C^{(n)}\mid\mathcal H_n]
 &\le e^{N(d(C)+2\eta)}
 &&\text{on }\{Y_n\in C\},\\
 \mathbb E[
   \widehat\tau_{A_i}^{(n)}\mid\mathcal H_n]
 &\le e^{N(D_i+2\eta)}
 &&\text{on }\{Y_n=i\}.
\end{aligned}
\]
Because \(2\eta\le u\), these prove (i) and (iii).

\smallskip
\noindent\emph{Early exit.}
Fix a cycle \(C\) with \(\delta<d(C)\), and set
\[
 m:=e^{N(d(C)-\delta)},
 \qquad t:=e^{N(d(C)-\delta/2)}.
\]
For any nonnegative random variable \(T\),
\[
 \mathbf1_{\{T\le m\}}\le e\,e^{-T/m}.
\]
Also, for every \(x\in C\), splitting according to whether
\(\tau_{C^c}\le t\) gives, using \eqref{eq:bana},
\[
\begin{aligned}
 \mathbb E_x^{Q_N}[e^{-\tau_{C^c}/m}]
 &\le
 \mathbb P_x^{Q_N}(\tau_{C^c}\le t)+e^{-t/m}\\
 &\le K e^{-\gamma N}
       +\exp\{-e^{N\delta/2}\}.
\end{aligned}
\]
The second comparison bound in \eqref{eq:comparison} yields, on \(\{Y_n\in C\}\),
\[
\begin{aligned}
 \mathbb P(
   \widehat\theta_C^{(n)}\le m\mid\mathcal H_n)
 \le
 eB_N\left(K e^{-\gamma N}
           +\exp\{-e^{N\delta/2}\}\right)
 \le C_0e^{-\gamma N/2},
\end{aligned}
\]
for all sufficiently large \(N\). This gives the second bound in part (ii).

\smallskip
\noindent\emph{Late exit.}
For any proper cycle \(C\), put
\[
 M_{C,N}:=e^{N(d(C)+2\eta)},
 \qquad L_{C,N}:=\lceil2M_{C,N}\rceil.
\]
The conditional mean bound for $\widehat\theta_C^{(n)}$ and Markov's inequality give
\[
 \mathbb P(
   \widehat\theta_C^{(n)}>L_{C,N}\mid\mathcal H_n)
 \le\frac12
 \qquad\text{on }\{Y_n\in C\}.
\]
On the event that the cycle has not yet been left, this same
bound applies at the beginning of the next block.
Iterating at the indices
\(n+kL_{C,N}\) consequently gives
\[
 \mathbb P(
   \widehat\theta_C^{(n)}>kL_{C,N}\mid\mathcal H_n)
 \le 2^{-k},
 \qquad k\ge0.
\]
Since \(\delta-2\eta\ge\delta/2\), it follows, for all
sufficiently large \(N\), that
\[
\begin{aligned}
 \mathbb P(
   \widehat\theta_C^{(n)}>e^{N(d(C)+\delta)}
   \mid\mathcal H_n)
 \le
 2^{-\left\lfloor
       e^{N(d(C)+\delta)}/L_{C,N}
      \right\rfloor}
 \le \exp\{-e^{N\delta/4}\}.
\end{aligned}
\]
Taking \(c=\min\{\gamma/2,\delta/4\}\) completes the proof of (ii).
All comparisons used \(R_N\) defined with \(\alpha_0\),
so the constants are uniform over \(0<\alpha\le\alpha_0\).
Finally, increase \(N_0\) to satisfy the finitely many
thresholds above simultaneously for all cycles and labels.
\end{proof}

\begin{lemma}[Physical gate clock]
\label{lem:physical-gate-clock}
For every \(A<\infty\) and \(\kappa>0\), the gates in
\eqref{eq:minimum-gates} can be chosen sufficiently small that we can find positive constants $c,c'$ such that for all sufficiently large $N$,
\begin{equation}
\label{eq:physical-gate-clock}
    \sup_{\nu\in\PN(\T)}
    \mathbb P_\nu\left(
        \max_{-1\leq n\leq\lfloor\e^{AN}\rfloor}
        (\eta_{n+1}-\eta_n)
        >\e^{\kappa N}
    \right)
    \le e^{-c\,e^{c'N}}.
\end{equation}
\end{lemma}

\begin{proof}
Fix \(\gamma>0\).  We first record uniform entry and exit estimates.

\smallskip
\noindent
\emph{Uniform entry into the gate set.}
For every \(q\in\Qq\), there exists $S_q>0$, a path $\gamma_q:[0,S_q] \rightarrow \Qq$ and a local minimum component $\mathcal{M}_k$ such that $\gamma_q(0)=q$, $\gamma_q(S_q) \in G_k^-$ and $ I_{[0,S_q]}(\gamma_q) < \gamma/2$.
To see this, for
an arbitrary initial measure, first follow the deterministic flow for
a positive time; this has zero action and gives a smooth strictly
positive density.  Step~3 in the proof of
Lemma~\ref{lem:singleton-exit-cost} then gives the required path to a
local-minimum component.  Proposition~\ref{prop:uniform-LDP}\textup{(ii)}, Corollary~\ref{cor:finite-cover-lower}
and compactness of \(\Qq\) therefore give
\(T_{\mathrm{ent}}<\infty\) such that
\begin{equation}
\label{eq:clock-uniform-entry}
    \inf_{\nu\in\PN(\T)}
    \mathbb P_\nu\left(
        \tau_{G^-}^N\leq T_{\mathrm{ent}}
    \right)
    \geq \e^{-\gamma N}
\end{equation}
for all sufficiently large \(N\).

\smallskip
\noindent
\emph{Uniform exit from the gates.}
The local cost that is needed uniformly over the entrance point is
\begin{equation}
\label{eq:uniform-local-gate-cost}
    \ell_{\mathrm{gate}}
    :=
    \max_{1\leq i\leq J}
    \sup_{q\in K_i}
 \Vq\bigl(q,\Qq\setminus\overline{G_i^+}\bigr).
\end{equation}

As in Step~1 of
Lemma~\ref{lem:uniform-shortening-arbitrary-gate-points},
zero-action smoothing followed by low-action affine
interpolation between densities close in $C^1$ to a stationary
representative of $\mathcal M_i$ shows that
$\ell_{\mathrm{gate}}$ can be made arbitrarily small by
suitably shrinking the gates.

The uniform LDP lower bound in Proposition~\ref{prop:uniform-LDP}\textup{(ii)} and a
finite covering of the finitely many compact sets
\(K_i=\overline{G_i^-}\) then give \(T_{\mathrm{ex}}<\infty\) such that
\begin{equation}\label{eq:clock-uniform-exit}
 \inf_{\substack{1\le i\le J\\q\in K_i\cap\PN(\T)}}
 \mathbb P_q\left(\tau_{(G_i^+)^c}^N\le T_{\mathrm{ex}}\right)
 \ge e^{-N(\ell_{\mathrm{gate}}+\gamma)}
\end{equation}
for all sufficiently large $N$.

Now, we combine the above two estimates. Let \(B:=T_{\mathrm{ent}}+T_{\mathrm{ex}}\).  By the strong Markov
property, from an arbitrary configuration and at an arbitrary
stopping time, the probability of entering an inner gate and then
leaving the corresponding outer gate during the next interval of
length \(B\) is at least
\[
    q_N:=\e^{-N(\ell_{\mathrm{gate}}+2\gamma)}.
\]
Applying this estimate in successive blocks of length \(B\), with
\(
    m_N:=\left\lfloor\frac{\e^{\kappa N}}{B}\right\rfloor,
\)
gives, uniformly in \(n\),
\begin{equation}
\label{eq:one-gate-clock-tail}
\begin{aligned}
    \mathbb P_\nu\left(
        \eta_{n+1}-\eta_n>\e^{\kappa N}
        \,\middle|\,
        \mathcal F_{\eta_n}
    \right)
    &\leq (1-q_N)^{m_N}\\
    &\leq
    \exp\left\{
        -c\e^{N(\kappa-\ell_{\mathrm{gate}}-2\gamma)}
    \right\}.
\end{aligned}
\end{equation}
The same estimate applies to \(\eta_0-\eta_{-1}=\eta_0\).
Choose the gates and then \(\gamma\) so that
\(\ell_{\mathrm{gate}}+2\gamma<\kappa\). Note that these bounds are also uniform in the starting measure $\nu \in \PN(\T)$. A union bound over at most
\(2+\e^{AN}\) increments proves \eqref{eq:physical-gate-clock}.
\end{proof}

\begin{lemma}[Excursion count bound]
\label{lem:fast-gate-clock}
Fix strictly nested gates, independently of \(N\), and let
\[
 \mathcal N_N(T):=\#\{n\ge0:\eta_n\le T\}
\]
be the number of completed inner-to-outer passages by time \(T\).
For every \(a\ge0\) and \(\kappa>0\), there exist \(c>0\) and
\(N_0\in\mathbb N\) such that, for every \(N\ge N_0\),
\begin{equation}\label{eq:fast-gate-clock}
 \sup_{\nu\in\PN(\T)}
 \mathbb P_\nu\bigl(
   \mathcal N_N(e^{aN})>e^{(a+\kappa)N}
 \bigr)
 \le \exp\{-cNe^{(a+\kappa)N}\}.
\end{equation}
\end{lemma}

\begin{proof}
Write \(D_n:=\eta_n-\tau_n\) for the duration of the \(n\)-th
inner-to-outer passage. Strict nesting gives a constant \(b>0\)
such that the \(d_{\Qq}\)-distance from each closed inner gate
to the complement of its outer gate is at least \(b\).
Consequently, completing a passage requires displacement at
least \(b\) from its entrance point.

We first show that, for some fixed \(s,c_0>0\),
\begin{equation}\label{eq:short-gate-passage}
 \mathbb P_\nu(D_n<s\mid\mathcal F_{\tau_n})
 \le e^{-c_0N},
\end{equation}
uniformly over \(n\) and the entrance configuration.
Choose \(s>0\) with \(\|W'\|_\infty \, s\le b/2\).
Coupling each particle to its position at time \(\tau_n\) gives
\[
 \sup_{0\le v\le s}
 d_{\Qq}(X_{\tau_n+v}^N,X_{\tau_n}^N)
 \le \|W'\|_\infty \, s+\sqrt{\frac{2}{\beta}}
 \left(
   \frac1N\sum_{j=1}^N
   \sup_{0\le v\le s}
   |B_{\tau_n+v}^j-B_{\tau_n}^j|^2
 \right)^{1/2}.
\]
Thus \(D_n<s\) requires the average inside parentheses to be
at least \(\beta b^2/8\).
For a standard Brownian motion \(B\), choose \(\lambda>0\)
sufficiently small that
\[
 K_\lambda:=
 \mathbb E\exp\left\{
   \lambda\sup_{0\le v\le1}|B_v|^2
 \right\}<\infty.
\]
The strong Markov property of the driving Brownian motions,
Brownian scaling, and the exponential Markov inequality yield
\[
 \mathbb P_\nu(D_n<s\mid\mathcal F_{\tau_n})
 \le
 \exp\left\{
   -N\left(\frac{\lambda\beta b^2}{8s}
            -\log K_\lambda\right)
 \right\}.
\]
Reducing \(s\) makes the expression in parentheses a positive
constant \(c_0\), proving \eqref{eq:short-gate-passage}.

Now suppose that \(\mathcal N_N(T)\ge m\). Thus,
\(
 \sum_{n=0}^{m-1}D_n\le T.
\)
If \(m>2T/s\), at least \(\lceil m/2\rceil\)
of the first \(m\) passages have duration less than \(s\).

For any fixed set \(I\subseteq\{0,\ldots,m-1\}\), iterating
\eqref{eq:short-gate-passage} gives
\[
 \mathbb P_\nu(D_n<s\text{ for every }n\in I)
 \le e^{-c_0N|I|}.
\]
A union bound over subsets of size \(\lceil m/2\rceil\)
therefore gives, whenever \(m>2T/s\),
\[
 \mathbb P_\nu(\mathcal N_N(T)\ge m)
 \le
 \binom{m}{\lceil m/2\rceil}
 e^{-c_0N\lceil m/2\rceil}
 \le 2^m e^{-c_0Nm/2}
 \le e^{-c_0Nm/4}
\]
for all sufficiently large \(N\), uniformly in \(\nu\).

Finally, take
\[
 T=e^{aN},
 \qquad
 m=\lfloor e^{(a+\kappa)N}\rfloor+1.
\]
Then \(m>2T/s\) for all sufficiently large \(N\), and
\(\{\mathcal N_N(T)>e^{(a+\kappa)N}\}
 =\{\mathcal N_N(T)\ge m\}\).
The preceding bound proves \eqref{eq:fast-gate-clock}
with \(c=c_0/4\).
\end{proof}

\begin{theorem}[Two-sided physical cycle-exit times]
\label{prop:two-sided-physical-cycles}
For each proper cycle \(C\), set
\[
 K_C:=\bigcup_{i\in C}K_i,
 \qquad
 \theta_C^N
 :=\inf\{t\ge0:X_t^N\in K_{C^c}\}.
\]
For every \(\varepsilon>0\), the gates can be chosen sufficiently
small, independently of \(N\), so that there exist constants
\(c,C_0>0\) and \(N_0\in\mathbb N\) satisfying
\begin{equation}\label{eq:physical-two-sided-cycle}
 \sup_{\nu\in K_C\cap\PN(\T)}
 \mathbb P_\nu\left(
   \theta_C^N\notin
   [e^{N(d(C)-\varepsilon)},e^{N(d(C)+\varepsilon)}]
 \right)
 \le C_0e^{-cN}
\end{equation}
and
\[
 \sup_{\nu\in K_C\cap\PN(\T)}
 \mathbb P_\nu\left(
   \theta_C^N>e^{N(d(C)+\varepsilon)}
 \right)
 \le \exp\{-e^{cN}\}
\]
for every proper cycle \(C\) and every \(N\ge N_0\).
In particular, the gates and constants may be chosen simultaneously
for all proper cycles.
\end{theorem}

\begin{proof}
Fix a proper cycle \(C\). It suffices to consider
\(0<\varepsilon<d(C)\). The estimates for larger \(\varepsilon\) follow because the
permitted exit-time interval increases with \(\varepsilon\).
All estimates below are uniform over
\(\nu\in K_C\cap\PN(\T)\).

Since the process starts in \(K_C\), we have \(\tau_0=0\).
Recalling
\(
 \widehat\theta_C:=\inf\{n\ge0:Y_n\notin C\},
\)
the construction of the gate times gives
\(
 \theta_C^N=\tau_{\widehat\theta_C}.
\)

For the upper tail, recall that \(\eta_{-1}=0\), and define
\[
 E_N:=
 \left\{
   \max_{0\le n\le \left\lceil e^{N(d(C)+\varepsilon/3)}\right\rceil}
   (\eta_n-\eta_{n-1})\le e^{N\varepsilon/3}
 \right\}.
\]
On \(\{\widehat\theta_C\le \left\lceil e^{N(d(C)+\varepsilon/3)}\right\rceil\}\cap E_N\),
\[
 \theta_C^N
 =\tau_{\widehat\theta_C}
 \le \eta_{\left\lceil e^{N(d(C)+\varepsilon/3)}\right\rceil}
 =\sum_{n=0}^{\left\lceil e^{N(d(C)+\varepsilon/3)}\right\rceil}(\eta_n-\eta_{n-1})
 \le (\left\lceil e^{N(d(C)+\varepsilon/3)}\right\rceil+1)e^{N\varepsilon/3}
 <e^{N(d(C)+\varepsilon)}
\]
for all sufficiently large \(N\).
The term \(n=0\) includes the initial interval from \(0\)
to \(\eta_0\). Consequently,
\[
 \{\theta_C^N>e^{N(d(C)+\varepsilon)}\}
 \subseteq
 \{\widehat\theta_C>\left\lceil e^{N(d(C)+\varepsilon/3)}\right\rceil\}\cup E_N^c.
\]
Lemma~\ref{lem:conditional-FW}(ii), with
\(\delta=\varepsilon/3\), bounds the first event.
Lemma~\ref{lem:physical-gate-clock} and
\eqref{eq:one-gate-clock-tail}, with
\(\kappa=\varepsilon/3\), bound each clock increment.
After choosing the gates sufficiently small, a union bound gives
constants \(c_1,c_2>0\) such that
\[
 \mathbb P_\nu(\theta_C^N>e^{N(d(C)+\varepsilon)})
 \le
 \exp\{-e^{c_1N}\}
 +\exp\{-e^{c_2N}\}
 \le \exp\{-e^{c_+N}\}
\]
for some \(c_+>0\) and all sufficiently large \(N\).

Now, we address the lower tail. If exit occurs at renewal index \(k\), then
\(\eta_{k-1}\le\tau_k=\theta_C^N\).
Thus at least \(k\) inner-to-outer passages have been
completed by the physical exit time. It follows that
\[
 \{\theta_C^N<e^{N(d(C)-\varepsilon)}\}
 \subseteq
 \{\widehat\theta_C\le  \left\lfloor e^{N(d(C)-\varepsilon/2)}\right\rfloor\}
 \cup
 \{\mathcal N_N(e^{N(d(C)-\varepsilon)})> \left\lfloor e^{N(d(C)-\varepsilon/2)}\right\rfloor\}.
\]
Lemma~\ref{lem:conditional-FW}(ii), again with
\(\delta=\varepsilon/3\), gives
\[
 \mathbb P_\nu(\widehat\theta_C\le \left\lfloor e^{N(d(C)-\varepsilon/2)}\right\rfloor)
 \le C_1e^{-c_-N}
\]
for some \(C_1,c_->0\).
For the second event, Lemma~\ref{lem:fast-gate-clock} with
\(a=d(C)-\varepsilon\) and \(\kappa=\varepsilon/2\) yields
\[
 \mathbb P_\nu(\mathcal N_N(e^{N(d(C)-\varepsilon)})>\left\lfloor e^{N(d(C)-\varepsilon/2)}\right\rfloor)
 \le
 \exp\{-c_3N e^{N(d(C)-\varepsilon/2)}\}
\]
for some \(c_3>0\). These two bounds prove the required
exponential estimate for the lower tail.

Combining the two tails proves the theorem. Since there
are finitely many proper cycles, all gate requirements can
be imposed simultaneously, and common constants follow
by decreasing \(c\) and increasing \(C_0,N_0\).
\end{proof}

\begin{lemma}[Physical time to a strictly lower minimum]
\label{lem:physical-descent}
For a nonglobal minimum $i$, use $A_i,h_i,D_i$ from
\eqref{eq:descent-height-definition} and set
$\sigma_i^N=\inf\{t\ge0:X_t^N\in\bigcup_{j\in A_i}K_j\}$.
For every $\varepsilon>0$, sufficiently small gates give $c,C_0>0$
such that, for all large $N$,
\begin{equation}\label{eq:physical-descent}
 \sup_{\nu\in K_i\cap\PN(\T)}
 \mathbb P_\nu\bigl(\sigma_i^N>e^{N(D_i+\varepsilon)}\bigr)
 \le C_0e^{-cN}.
\end{equation}
\end{lemma}

\begin{proof}
Since the process starts in \(K_i\),
\(\sigma_i^N=\tau_{\widehat\tau_{A_i}}\).
Lemma~\ref{lem:conditional-FW}(iii), with 
\(u=\varepsilon/4\), and Markov's inequality give, uniformly
over \(\nu\in K_i\cap\PN(\T)\),
\[
 \mathbb P_\nu(\widehat\tau_{A_i}>\lceil e^{N(D_i+\varepsilon/2)}\rceil)
 \le \frac{e^{N(D_i+\varepsilon/4)}}{\lceil e^{N(D_i+\varepsilon/2)}\rceil}
 \le e^{-N\varepsilon/4}.
\]
Lemma~\ref{lem:physical-gate-clock}, with
\(\kappa=\varepsilon/4\), bounds every increment
\(\eta_n-\eta_{n-1}\), \(0\le n\le \lceil e^{N(D_i+\varepsilon/2)}\rceil\), by
\(e^{N\varepsilon/4}\), except on an event whose probability
is superexponentially small in \(N\).
On this clock event and \(\{\widehat\tau_{A_i}\le \lceil e^{N(D_i+\varepsilon/2)}\rceil\}\),
using \(\eta_{-1}=0\), we obtain
\[
 \sigma_i^N
 \le\eta_{\lceil e^{N(D_i+\varepsilon/2)}\rceil}
 =\sum_{n=0}^{\lceil e^{N(D_i+\varepsilon/2)}\rceil}(\eta_n-\eta_{n-1})
 \le(\lceil e^{N(D_i+\varepsilon/2)}\rceil+1)e^{N\varepsilon/4}
 <e^{N(D_i+\varepsilon)}
\]
for all sufficiently large \(N\). Combining the two
exceptional-event bounds proves the claim.
\end{proof}

\section{Deterministic time localization and the Gibbs comparison}
\label{sec:endpoint-lemmas}

In this section, we move from estimates on the random cycle-exit times to localization estimates at deterministic time points.

In Lemma~\ref{lem:local-block-endpoint} we show that the empirical measure process, started from a neighborhood of a local minimum, remains confined close to it with high probability over time horizons exponentially growing with $N$. In Lemma~\ref{lem:uniform-minimum-localization}, we use recurrence properties of $X^N$ along with the confinement estimate to show that the \emph{law} of $X^N_t$ puts most of its mass on the union of arbitrarily small neighborhoods of the local minima, regardless of its initialization, for every sufficiently large \emph{deterministic} time. In Lemma~\ref{lem:uniform-uphill-endpoint}, we use a comparison technique of the law of $X^N_t$ with the empirical projection of the stationary Gibbs measure of the joint particle motion to show that, for any $t \ge 0$, the probability of $X^N_t$, started near a local minimum, lying in a higher energy region is exponentially unlikely in $N$. For times $t\ge1$, this bound follows from the static Gibbs large deviations and a finite-time Girsanov comparison. The proof uses the finite-horizon path LDP only to treat $[0,1]$ and does not use the long-excursion analysis. Finally, in Lemma~\ref{lem:lower-energy-persistence}, we use Lemmas~\ref{lem:uniform-minimum-localization} and~\ref{lem:uniform-uphill-endpoint} to show that, for any $a \in \mathbb{R}$, if $X^N$ is started close to a local minimum orbit with energy less than $a$, then at all positive times, its law localizes solely around local minimum orbits whose energy levels are less than $a$.

\begin{lemma}[Local confinement from a stable gate]
\label{lem:local-block-endpoint}
For every attracting minimum \(\mathcal M_i\) and every sufficiently
small \(r>0\), there exist an open neighborhood \(V_i\) of
\(\mathcal M_i\), constants \(c,C>0\), and \(N_0\in\mathbb N\)
such that, for every \(N\ge N_0\) and \(T\ge0\),
\begin{equation}\label{eq:local-block-endpoint}
 \sup_{\nu\in V_i\cap\PN(\T)}
 \mathbb P_\nu\left(
   \sup_{0\le s\le T}d_{\Qq}(X_s^N,\mathcal M_i)\ge r
 \right)
 \le C(1+T)e^{-cN}.
\end{equation}
\end{lemma}

\begin{proof}
We first justify stability of the deterministic semiflow $t \mapsto S_tq$ for $q$ sufficiently close to $\mathcal{M}_i$. Choose a stable gate \(G\) for \(\mathcal M_i\) inside
\(B_{r/2}\), as in Definition~\ref{def:stable-gate}.
Choose \(0<\eta<r/2\) so that
\(\overline V_i:=\overline{B_\eta}\subset G\).
The stable-gate property gives
\[
 S_t(\overline V_i)\subset B_{r/2},
 \qquad t\ge0.
\]
Uniform attraction on \(\overline G\) then gives a fixed
\(s_0\ge1\) such that
\[
 S_{s_0}(\overline V_i)\subset B_{\eta/2}.
\]
The closed set of paths on the time interval \([0,s_0]\), starting in \(\overline V_i\), that
leave \(B_r\) or end outside \(V_i\)
therefore contains no zero-action path.
Compactness of action sublevel sets and
Proposition~\ref{prop:uniform-LDP}(i) give, for some \(c>0\)
and all sufficiently large \(N\),
\[
 \sup_{\nu\in\overline V_i\cap\PN(\T)}
 \mathbb P_\nu\left(
   \sup_{0\le s\le s_0}d_{\Qq}(X_s^N,\mathcal M_i)\ge r
   \ \text{or}\ X_{s_0}^N\notin V_i
 \right)
 \le e^{-cN}.
\]
On the complementary event, the process stays in \(B_r\)
and returns to \(V_i\), so the same bound applies conditionally to the
next block $[s_0, 2s_0]$ by the Markov property.
Iterating over \(\lceil T/s_0\rceil\) full blocks, including
the block containing \(T\), gives an exit probability at most
\[
 \left\lceil\frac{T}{s_0}\right\rceil e^{-cN}
 \le (1+T)e^{-cN}.
\]
The case \(T=0\) is immediate.
\end{proof}
For the following lemma, recall the attraction basin $\mathcal A_i$ of $\mathcal M_i$ defined in \eqref{eq:doa}.

\begin{lemma}[Localization in law at every large enough deterministic observation time]
\label{lem:uniform-minimum-localization}
Let \(O_j\) be pairwise disjoint open neighborhoods of the local
minima, with \(\overline O_j\subset\mathcal A_j\).
For every \(\bar\kappa>0\), there exist closed neighborhoods
\(F_j\Subset O_j\) of \(\mathcal M_j\), a number
\(\kappa\in(0,\bar\kappa)\), and constants \(c,C>0\), \(N_0\),
with the following property.

Set \(F:=\bigcup_jF_j\) and \(L_N:=e^{\kappa N}\).
For every \(N\ge N_0\), initial state \(\nu\in\PN(\T)\),
and deterministic \(t\ge L_N\), define
\[
 \sigma_t^N:=\inf\{s\ge t-L_N:X_s^N\in F\}.
\]
With probability at least \(1-Ce^{-cN}\), one has
\(\sigma_t^N\le t-L_N/2\), and the process stays in \(O_j\)
throughout \([\sigma_t^N,t]\), where \(j\) is the unique
index satisfying \(X_{\sigma_t^N}^N\in F_j\).
In particular,
\begin{equation}\label{eq:uniform-minimum-localization}
 \sup_{\nu\in\PN(\T)}\sup_{t\ge L_N}
 \mathbb P_\nu\left(X_t^N\notin\bigcup_jO_j\right)
 \le Ce^{-cN}.
\end{equation}
Moreover, for every \(i\),
\[
 \sup_{\nu\in F_i\cap\PN(\T)}
 \mathbb P_\nu\left(
   X_s^N\notin O_i\text{ for some }s\in[0,L_N]
 \right)
 \le Ce^{-cN}.
\]
\end{lemma}

\begin{proof}
Apply Lemma~\ref{lem:local-block-endpoint} with radii \(r_j\)
such that \(B_{r_j}(\mathcal M_j)\subset O_j\).
Choose closed neighborhoods \(F_j\) inside the corresponding
initial neighborhoods \(V_j\). Taking common constants over
the finitely many minima gives \(a,C_1>0\) such that for all $1 \le j \le J$,
\begin{equation}\label{eq:confinement}
 \sup_{\nu\in F_j\cap\PN(\T)}
 \mathbb P_\nu\left(
   X_s^N\notin O_j\text{ for some }s\in[0,T]
 \right)
 \le C_1(1+T)e^{-aN},
 \qquad T\ge0.
\end{equation}
Choose \(0<\kappa<\min\{\bar\kappa,a/2\}\).

Recall
\(
 \tau_F^N:=\inf\{s\ge0:X_s^N\in F\}.
\) The uniform-entry argument \eqref{eq:clock-uniform-entry},
applied to the interiors of the \(F_j\) with
\(\gamma=\kappa/2\), gives a fixed \(S<\infty\) such that
\[
 \inf_{\nu\in\PN(\T)}
 \mathbb P_\nu(\tau_F^N\le S)\ge e^{-\kappa N/2}.
\]
Iteration in blocks of length \(S\) therefore yields
\[
 \sup_{\nu\in\PN(\T)}
 \mathbb P_\nu(\tau_F^N>L_N/2)
 \le
 (1-e^{-\kappa N/2})^{\lfloor L_N/(2S)\rfloor}
 \le \exp\{-c_1e^{\kappa N/2}\}
\]
for some \(c_1>0\) and all sufficiently large \(N\).

For \(t\ge L_N\), apply this estimate from time \(t-L_N\).
Except on the indicated exceptional event, the first entrance
time \(\sigma_t^N\) lies in \([t-L_N,t-L_N/2]\).
At this entrance time, the strong Markov property and the
confinement estimate \eqref{eq:confinement} above apply to whichever \(F_j\) is reached.
Since \(t-\sigma_t^N\le L_N\), the probability of failure
of either entrance or subsequent confinement is at most
\[
 \exp\{-c_1e^{\kappa N/2}\}
 +C_1(1+L_N)e^{-aN}
 \le Ce^{-cN}.
\]
This bound is uniform in the initial state and in \(t\).
Starting in \(F_i\), the final assertion follows directly
from the confinement estimate \eqref{eq:confinement} with \(T=L_N\).
\end{proof}

Let $\pi_N$ denote the configuration Gibbs law with Hamiltonian given by \eqref{eq:HN},
and let $\Pi_N$ be the law of $[\mu^N]$ under $\pi_N$.
Set
\[
 \mathcal W(\mu):=
 \frac12\iint_{\T^2}W(x-y)\,\mu(\dd x)\mu(\dd y).
\]
Sanov's theorem
gives an LDP for the empirical measure \(\mu^N\) constructed from particles distributed as \(\lambda^{\otimes N}\), with speed \(N\) and good rate
\(\Ent(\mu\mid\lambda)\).
The Gibbs law \(\pi_N\) is obtained by weighting this product
law by \(e^{-\beta N\mathcal W(\mu^N)}\) and normalizing.
Since \(\mathcal W\) is bounded and continuous, the exponential
tilting theorem gives the rate
\[
 \widetilde J(\mu)
 =
 \Ent(\mu\mid\lambda)+\beta\mathcal W(\mu)
 -
 \inf_{\nu\in\Pp(\T)}
 \bigl\{\Ent(\nu\mid\lambda)+\beta\mathcal W(\nu)\bigr\}
 =
 \beta\bigl(\F(\mu)-F_{\min}\bigr);
\]
see \cite[Theorems~3.2 and~1.2(ii)]{Rezakhanlou2017}.

On the compact space \(\T\), the weak and \(W_2\) topologies
coincide. The quotient map \(\mu\mapsto[\mu]\) is continuous,
and \(\F\) is rotation invariant. The contraction principle
\cite[Theorem~1.2(i)]{Rezakhanlou2017} therefore gives an LDP
for \(\Pi_N\) on \(\Qq\), with speed \(N\) and good rate
\begin{equation}\label{eq:gibbs-rate-endpoint}
 J(q)
 =\inf_{[\mu]=q}\widetilde J(\mu)
 =\beta\bigl(\F(q)-F_{\min}\bigr).
\end{equation}

Although every empirical measure is atomic and satisfies
\(\Ent(\mu^N\mid\lambda)=+\infty\), the LDP describes
probabilities of neighborhoods of candidate limiting measures.
The entropy term arises from Sanov's theorem; the finite-\(N\)
Gibbs weight involves only the interaction energy
\(\mathcal W(\mu^N)\).

\begin{lemma}[Uniform upward-energy bound]
\label{lem:uniform-uphill-endpoint}
Let $\mathcal M_i$ be a local-minimum orbit.  Let $A\subset\Qq$ be
closed and satisfy $\inf_A\F>f_i$.  There exist $\eta,c,C>0$ such that
\begin{equation}\label{eq:uniform-uphill-endpoint}
 \sup_{d_{\Qq}([\nu],\mathcal M_i)\le\eta}
 \sup_{t\ge0}\mathbb P_\nu(X_t^N\in A)\le Ce^{-cN}.
\end{equation}
The supremum over $\nu$ is restricted to $\PN(\T)$. This lemma assumes \eqref{eq:W-assumption} but does not require \eqref{eq:finite-critical-orbits}.
\end{lemma}

\begin{proof}
First we compare two deterministic initial configurations.
Equip \(\T^N\) with
\[
 d_N(x,y)^2:=\sum_{j=1}^N d_\T(x_j,y_j)^2.
\]
Write \(b_N=(b_N^1,\ldots,b_N^N)\) for the configuration
drift in \eqref{eq:particle-system} of the joint angular motion (which is naturally identified with motion on $\T^N$), namely
\[
 b_N^j(x):=-\frac1N\sum_{k=1}^N W'(x_j-x_k),
 \qquad x\in\T^N,\quad 1\le j\le N.
\]
Thus \(b_N=-\nabla\mathcal H_N\), and it satisfies
\[
 |b_N(z)-b_N(z')|\le L|z-z'|,
 \qquad L:=2\|W''\|_\infty,
\]
with \(L\) independent of \(N\).

Choose two points $x,y \in \T^N$. Choose coordinate lifts in \(\mathbb R^N\), again denoted by $x,y$, such that
\(h:=y-x\) satisfies \(|h|=d_N(x,y)\).
Let \(\mathbf X\) solve the angular particle SDE \eqref{eq:particle-system} started from \(x\), and set
\[
 \mathbf Y_s:=\mathbf X_s+(1-s)_+h, \quad s \ge 0.
\]
Thus \(\mathbf Y_0=y\), and the two paths agree for every
\(s\ge1\). During \([0,1]\), the process \(\mathbf Y\) solves
the original SDE \eqref{eq:particle-system} with the additional drift
\[
 u_s:=b_N(\mathbf X_s)-b_N(\mathbf Y_s)-h,
 \qquad
 |u_s|\le [1+L(1-s)]|h|.
\]
Put \(v_s:=\sqrt{\beta/2}\,u_s\) and define
\[
 R:=
 \exp\left\{
   -\int_0^1v_s\cdot\dd B_s
   -\frac12\int_0^1|v_s|^2\,\dd s
 \right\}.
\]
The control is bounded, so Novikov's condition holds.
Under the probability measure \(R\,\dd\mathbb P\),
Girsanov's theorem makes
\[
 \widetilde B_s
 :=B_s+\int_0^{s\wedge1}v_u\,\dd u
\]
a Brownian motion. Substituting \(\dd B_s
=\dd\widetilde B_s-v_s\,\dd s\) removes the additional drift.
Consequently, under this changed probability,
\(\mathbf Y\) has the law of the original diffusion started at \(y\).

All expectations below are under the original probability.
The identity
\[
 R^{-1}
 =
 \exp\left\{
   \int_0^1v_s\cdot\dd B_s
   -\frac12\int_0^1|v_s|^2\,\dd s
 \right\}
 \exp\left\{\int_0^1|v_s|^2\,\dd s\right\}
\]
expresses \(R^{-1}\) as a mean-one stochastic exponential
times a factor bounded by \(e^{c_0|h|^2}\), where
\[
 \int_0^1|v_s|^2\,\dd s
 \le \frac\beta2|h|^2
       \int_0^1[1+L(1-s)]^2\,\dd s
 =c_0|h|^2,
 \qquad
 c_0:=\frac\beta2(1+L+L^2/3).
\]
Hence
\begin{equation}\label{eq:girsanov-second-moment}
 \mathbb E[R^{-1}]
 \le e^{c_0d_N(x,y)^2}.
\end{equation}
For an event \(E\) depending only on the path during
\([1,\infty)\), the equality of \(\mathbf X\) and \(\mathbf Y\)
on this interval and Cauchy--Schwarz give
\begin{equation}\label{eq:path-event-comparison}
\begin{aligned}
 \mathbb P_x(E)^2
 &=\left(
   \mathbb E[
     R^{1/2}\mathbf1_{\{\mathbf Y\in E\}}R^{-1/2}]
   \right)^2\\
 &\le
 \mathbb E[R\mathbf1_{\{\mathbf Y\in E\}}]\,
 \mathbb E[R^{-1}]
 \le e^{c_0d_N(x,y)^2}\mathbb P_y(E).
\end{aligned}
\end{equation}

We now apply this comparison to quotient empirical measures.
For \(x\in\T^N\) (identified with $\mathbb R^N$), write
\[
 \mu_x^N:=\frac1N\sum_{j=1}^N\delta_{x_j},
 \qquad
 p_t(x):=\mathbb P_x([\mu_{\mathbf X_t}^N]\in A).
\]
Permutations and common rotations preserve the dynamics and the
quotient empirical measure. Let \(\mathfrak S_N\) denote the group of permutations of
\(\{1,\ldots,N\}\). The group \(\mathfrak S_N\times\T\)
acts on configurations by relabeling and common rotation:
for \(g=(\sigma,\theta)\), set
\[
 (gx)_j:=x_{\sigma^{-1}(j)}+\theta \pmod{2\pi},
 \qquad 1\le j\le N.
\] Thus
\[
 p_t(gy)=p_t(y),
 \qquad g\in\mathfrak S_N\times\T.
\]
For each fixed rotation \(\theta\), an optimal coupling of
the two equal-weight empirical measures can be chosen as a
matching of their particles \cite[Proposition 2.1]{PeyreCuturi2019}. Hence
\[
 W_2^2(\mu_x^N,R_\theta\mu_y^N)
 =
 \frac1N\min_{\sigma\in\mathfrak S_N}
 \sum_{j=1}^N
 d_\T(x_j,y_{\sigma^{-1}(j)}+\theta)^2.
\]
Taking the infimum over \(\theta\) and using the definitions
of the group action and \(d_{\Qq}\), we obtain
\[
 \inf_{g\in\mathfrak S_N\times\T}d_N(x,gy)^2
 =
 N d_{\Qq}([\mu_x^N],[\mu_y^N])^2.
\]

Let \(G:=B_\eta(\mathcal M_i)\), and suppose
\(d_{\Qq}([\mu_x^N],\mathcal M_i)\le\eta\).
For every \(y\) with \([\mu_y^N]\in G\), the last display
and the triangle inequality give an alignment of $x,y$ (given as a permutation of its coordinates followed by a common rotation) with squared
distance at most \(4N\eta^2\).
Applying \eqref{eq:path-event-comparison} to this alignment,
and using \(p_t(gy)=p_t(y)\), yields
\[
 p_t(x)^2\le e^{4c_0N\eta^2}p_t(y),
 \qquad t\ge1.
\]
Integrate over such \(y\) against \(\pi_N\).
Stationarity gives
\(\int p_t(y)\,\pi_N(\dd y)=\Pi_N(A)\).
Therefore, for \(\nu=\mu_x^N\),
\begin{equation}\label{eq:gibbs-uphill-comparison}
\begin{aligned}
 \mathbb P_\nu(X_t^N\in A)^2
 &\le
 \frac{e^{4c_0N\eta^2}}{\Pi_N(G)}
 \int_{\{[\mu_y^N]\in G\}}p_t(y)\,\pi_N(\dd y)\\
 &\le
 e^{4c_0N\eta^2}\frac{\Pi_N(A)}{\Pi_N(G)},
 \qquad t\ge1.
\end{aligned}
\end{equation}

Choose \(g>0\) such that \(\inf_A\F\ge f_i+g\).
The static LDP upper bound, discussed just before the lemma, for the closed set \(A\), and
the lower bound for the open set \(G\) containing
\(\mathcal M_i\), give, for every fixed \(\zeta>0\)
and all sufficiently large \(N\),
\[
 \Pi_N(A)
 \le e^{-N[\beta(f_i+g-F_{\min})-\zeta]},
 \qquad
 \Pi_N(G)
 \ge e^{-N[\beta(f_i-F_{\min})+\zeta]}.
\]
Consequently,
\[
 \frac{\Pi_N(A)}{\Pi_N(G)}
 \le e^{-N(\beta g-2\zeta)}.
\]
Choose \(\eta\) and \(\zeta\) so that
\(4c_0\eta^2<\beta g/4\) and \(\zeta<\beta g/8\).
Taking square roots in \eqref{eq:gibbs-uphill-comparison}
then gives
\[
 \sup_{d_{\Qq}([\nu],\mathcal M_i)\le\eta}
 \sup_{t\ge1}
 \mathbb P_\nu(X_t^N\in A)
 \le e^{-N\beta g/4}
\]
for all sufficiently large \(N\).

It remains to treat \([0,1]\).
Since \(A\) is closed and excludes the stationary orbit
\(\mathcal M_i\), continuity of the deterministic flow gives
a sufficiently small closed initial neighborhood whose
deterministic trajectories avoid \(A\) throughout \([0,1]\).
The closed event of hitting \(A\) during this interval
therefore contains no zero-action path from that neighborhood.
Compactness of action sublevel sets and the uniform LDP upper
bound give an exponentially small hitting probability.
Restricting the initial neighborhood accordingly and combining
the two time ranges proves the claim.
\end{proof}

\begin{lemma}[Endpoint persistence below an energy threshold]
\label{lem:lower-energy-persistence}
Fix $a\in\mathbb R$ such that
$I_a=\{j:f_j<a\}$ is nonempty, and put
$\mathcal M_{<a}=\{\mathcal M_j:j\in I_a\}$.
For every $r>0$ there are neighborhoods $V_j$ of $\mathcal M_j$,
$j\in I_a$, and $c,C>0$ such that
\begin{equation}\label{eq:lower-energy-persistence}
 \max_{j\in I_a}\sup_{\nu\in V_j\cap\PN(\T)}\sup_{t\ge0}
 \mathbb P_\nu\left(\dist_{\Qq}(X_t^N,\mathcal M_{<a})>r\right)
 \le Ce^{-cN}.
\end{equation}
\end{lemma}

\begin{proof}
Choose disjoint neighborhoods $O_k$ of all minima such that
\[
 \overline O_k\subset\mathcal A_k,
 \qquad O_k\subset B_r(\mathcal M_k).
\]
Let $F_k\Subset O_k$ be the closed neighborhoods and
$L_N=e^{\kappa N}$ the time scale supplied by
Lemma~\ref{lem:uniform-minimum-localization}.
Choose neighborhoods $V_j$ of $\mathcal M_j$, $j\in I_a$,
with $\overline V_j\subset\operatorname{int}F_j$.

Suppose first that $I_a$ does not contain every minimum, and set
\[
 A:=\bigcup_{k\notin I_a}\overline O_k.
\]
For $q\in\overline O_k\subset\mathcal A_k$, energy dissipation,
convergence to $\mathcal M_k$, and lower semicontinuity imply
$\F(q)\ge f_k$; this is immediate if $\F(q)=+\infty$.
Consequently, $A$ is closed and
\[
 \inf_A\F\ge\min_{k\notin I_a}f_k\ge a>f_j,
 \qquad j\in I_a.
\]
By Lemma~\ref{lem:uniform-uphill-endpoint}, we may therefore
shrink the $V_j$ so that, for all sufficiently large $N$,
\[
 \max_{j\in I_a}
 \sup_{\nu\in\PN(\T):[\nu]\in V_j}\sup_{t\ge0}
 \mathbb P_\nu(X_t^N\in A)
 \le C_1e^{-c_1N}.
\]
If $I_a$ contains every minimum, take $A=\varnothing$ and
omit this estimate.

For $0\le t\le L_N$, the confinement assertion of
Lemma~\ref{lem:uniform-minimum-localization} keeps the process
in $O_j\subset B_r(\mathcal M_j)$, starting from $V_j$,
except with exponentially small probability.

For $t\ge L_N$, observe that
\[
 \{\dist_{\Qq}(X_t^N,\mathcal M_{<a})>r\}
 \subset
 \left\{X_t^N\notin\bigcup_k O_k\right\}
 \cup\{X_t^N\in A\}.
\]
Indeed, membership in any $O_j$ with $j\in I_a$ makes the
distance on the left smaller than $r$.
Lemma~\ref{lem:uniform-minimum-localization} bounds the first event uniformly in
the initial state and in $t\ge L_N$, while the preceding
upward-energy estimate bounds the second uniformly in time.
Combining the two time ranges proves the result.
Since there are finitely many minima, the constants and the
lower bound on $N$ can be chosen uniformly over $j\in I_a$.
\end{proof}

\section{Metastability and phase transition: Proof of Theorem~\ref{prop:verified-MT}}
\label{sec:verify-assumption-one}
We are now ready to prove Theorem~\ref{prop:verified-MT}.

For this section suppose the minimum values $f_1,\ldots,f_J$ are
pairwise distinct. 
For each nonglobal minimum define $h_i,D_i$ by
\eqref{eq:descent-height-definition} and $C^{max}_i$ by \eqref{eq:maximal-bottom-cycle}.

\begin{lemma}[The maximal cycle with prescribed bottom]
\label{lem:maximal-bottom-cycle}
The set $C^{max}_i$ is the maximal energy cycle with bottom $\mathcal M_i$,
$\Gamma(C^{max}_i)=h_i$, and $d(C^{max}_i)=D_i$.  Every other minimum in $C^{max}_i$
has energy strictly greater than $f_i$.
\end{lemma}

\begin{proof}
The ultrametric inequality makes $C^{max}_i$ the component of $i$ immediately
below level $h_i$.  No strictly lower minimum belongs to it by the
definition of $h_i$, and no equal-energy minimum exists by the present
distinctness condition.  A lower minimum is joined at level $h_i$.
Thus the exit height is $h_i$, and the next larger cycle containing
$i$ has a different, lower bottom.  The depth formula \eqref{eq:cycle-depth-height} gives $D_i$.
\end{proof}

\begin{proof}[Proof of Theorem~\ref{prop:verified-MT}]
By Lemma~\ref{lem:maximal-bottom-cycle}, $C^{max}_i$ is a proper cycle
with unique bottom $\mathcal M_i$ and depth $d(C^{max}_i)=D_i$.

For \eqref{eq:verified-M1}, fix $r>0$ and
$0<\varepsilon<D_i$. Choose the excursion gates $K_j$ so that
Theorem~\ref{prop:two-sided-physical-cycles} applies with
margin $\varepsilon/2$. Then choose disjoint open neighborhoods
$O_j$ of the minima such that
\[
 \overline O_j\subset\operatorname{int}K_j\cap\mathcal A_j,
 \qquad O_i\subset B_r(\mathcal M_i).
\]
Apply Lemma~\ref{lem:uniform-minimum-localization} to these
neighborhoods. Denote its closed inner neighborhoods by $F_j$,
and choose its exponent so that
\[
 0<\kappa<D_i-\varepsilon,\qquad L_N=e^{\kappa N}.
\]

Set
\[
 A:=\bigcup_{j\in C^{max}_i\setminus\{i\}}\overline O_j.
\]
On $\overline O_j\subset\mathcal A_j$, energy dissipation,
convergence, and lower semicontinuity give $\F\ge f_j$.
Since $f_j>f_i$ for every $j\in C^{max}_i\setminus\{i\}$, the
closed set $A$ satisfies $\inf_A\F>f_i$ whenever it is nonempty.
Choose $\eta>0$ sufficiently small that
$\overline B_\eta(\mathcal M_i)\subset\operatorname{int}F_i$
and Lemma~\ref{lem:uniform-uphill-endpoint} gives
\[
 \mathbb P_\nu(X_t^N\in A)\le C_1e^{-c_1N},
 \qquad t\ge0,
\]
uniformly over the initial measures in
\eqref{eq:verified-M1}. If $A=\varnothing$, this probability
is zero.

For $0\le t\le L_N$, the confinement assertion of Lemma~\ref{lem:uniform-minimum-localization} gives the desired bound, since the
initial state belongs to $F_i$ and $O_i\subset B_r(\mathcal M_i)$.

Now let $L_N\le t\le e^{N(D_i-\varepsilon)}$. If
$\theta_{C^{max}_i}^N>t$ and $X_t^N\in O_j$, then $j\in C^{max}_i$:
otherwise $O_j\subset K_j$ would imply
$\theta_{C^{max}_i}^N\le t$. If also $X_t^N\notin A$, this label
must be $i$. Hence
\[
 \begin{split}
 \mathbb P_\nu\bigl(d_{\Qq}(X_t^N,\mathcal M_i)>r\bigr)
 \le{}&
 \mathbb P_\nu(\theta_{C^{max}_i}^N\le t)
 +\mathbb P_\nu(X_t^N\in A)
 +\mathbb P_\nu\left(X_t^N\notin\bigcup_jO_j\right).
 \end{split}
\]
The first term is exponentially small by \eqref{eq:physical-two-sided-cycle}, because
\[
 t\le e^{N(D_i-\varepsilon)}
   <e^{N(d(C^{max}_i)-\varepsilon/2)}.
\]
The second term is controlled by the upward-energy estimate in Lemma~\ref{lem:uniform-uphill-endpoint},
and the third by the localization result Lemma~\ref{lem:uniform-minimum-localization}. This proves
\eqref{eq:verified-M1} uniformly over the stated time interval.

For \eqref{eq:verified-M2}, fix $r,\varepsilon>0$. Apply
Lemma~\ref{lem:lower-energy-persistence} with $a=f_i$, obtaining
starting neighborhoods $V_j$ for all $j$ with $f_j<f_i$.
Choose the excursion gates sufficiently small that
\[
 K_j\subset V_j\quad(f_j<f_i)
\]
and Lemma~\ref{lem:physical-descent} applies with margin
$\varepsilon/2$. Choose $\eta>0$ with
$\overline B_\eta(\mathcal M_i)\subset K_i$, and recall
\[
 \sigma_i^N
 :=\inf\left\{s\ge0:X_s^N\in\bigcup_{j:f_j<f_i}K_j\right\}.
\]
Lemma~\ref{lem:physical-descent} gives
\[
 \mathbb P_\nu(\sigma_i^N>e^{N(D_i+\varepsilon/2)})\le C_2e^{-c_2N},
\]
uniformly over the required initial measures.

Fix any deterministic $t\ge e^{N(D_i+\varepsilon/2)}$. On
$\{\sigma_i^N\le e^{N(D_i+\varepsilon/2)}\}$, the state at $\sigma_i^N$ belongs to
some $K_j\subset V_j$ with $f_j<f_i$. The remaining time
$t-\sigma_i^N$ is random but nonnegative. Since the persistence
estimate in Lemma~\ref{lem:lower-energy-persistence} is uniform over all starting states in these
neighborhoods and all elapsed times, the strong Markov property
gives
\[
 \mathbb P_\nu\left(
   \dist_{\Qq}(X_t^N,\mathcal L_i)>r,\,
   \sigma_i^N\le e^{N(D_i+\varepsilon/2)}\right)
 \le C_3e^{-c_3N}.
\]
Consequently,
\[
 \mathbb P_\nu\bigl(
   \dist_{\Qq}(X_t^N,\mathcal L_i)>r\bigr)
 \le C_2e^{-c_2N}+C_3e^{-c_3N},
 \qquad t\ge e^{N(D_i+\varepsilon/2)}.
\]
The constants are independent of $t$. This proves
\eqref{eq:verified-M2}.
\end{proof}

\section{Metastable propagation-of-chaos}
\label{sec:metastable-poc}

In this section, we prove Theorems~\ref{thm:poc-metastable}, \ref{thm:poc-global}, \ref{thm:poc-pathwise}, \ref{thm:poc-running-averages} and Corollary~\ref{cor:poc-averaged-measures}.

Throughout this section, assume \eqref{eq:W-assumption}. The initial
configuration may be random, and is independent of the driving
Brownian motions. We assume \eqref{eq:poc-initial}, namely,
\begin{equation*}
 \Wass(\mu_0^N,\mu_0)\xrightarrow{\mathbb P}0
 \quad\text{for a deterministic }\mu_0\in\Pp(\T),
\end{equation*}
and write $\mu_t=S_t\mu_0$ and $q_t=[\mu_t]$.
No independence or exchangeability of the initial particles is
required; in particular, both deterministic configurations and
i.i.d.\ initialization are allowed.

Recall that $X_t^N=[\mu_t^N]$. To avoid notational overload, write
\begin{equation}\label{eq:poc-error-average}
 \Delta_N(t):=d_{\Qq}(X_t^N,q_t),
 \qquad
 A_N(T):=\frac1T\int_0^T\Delta_N(t)^2\,\dd t,\quad T>0.
\end{equation}

Set $B_0:=\operatorname{diam}(\Qq)<\infty$, so that
$0\le\Delta_N(t)^2,A_N(T)\le B_0^2$.
The paths of $\Delta_N$ are continuous, as are $T\mapsto A_N(T)$
for $T>0$; the time suprema below are therefore measurable.
All expectations below include the randomness of the initial
configuration. 

\subsection{Finite-time approximation and deterministic stability}

\begin{lemma}[Finite-time approximation from empirical convergence]
\label{lem:poc-finite-time}
Under \eqref{eq:poc-initial}, for every fixed $S<\infty$,
\begin{equation}\label{eq:poc-finite-time}
 \mathbb E\left[
   \sup_{0\le t\le S}\Wass(\mu_t^N,\mu_t)^2
 \right]\longrightarrow0.
\end{equation}
\end{lemma}

\begin{proof}
Fix $a>0$. Continuity of the deterministic semiflow allows us to
choose $\xi>0$ such that
\[
 \sup_{0\le t\le S}\Wass(S_t\nu,\mu_t)<a/2
 \quad\text{whenever }\Wass(\nu,\mu_0)\le\xi.
\]
The closed set of paths starting in this closed $\xi$-ball and
having distance at least $a$ from $\mu_\cdot$ somewhere on $[0,S]$
contains no zero-action path. Compactness of action sublevel sets
and Proposition~\ref{prop:uniform-LDP}(i) therefore give constants
$c,C>0$ such that
\[
 \mathbb P\left(
   \sup_{0\le t\le S}\Wass(\mu_t^N,\mu_t)\ge a
 \right)
 \le
 \mathbb P\bigl(\Wass(\mu_0^N,\mu_0)>\xi\bigr)+Ce^{-cN}
\]
for all sufficiently large $N$. The right-hand side
tends to zero by \eqref{eq:poc-initial}. Boundedness of the
Wasserstein distance on $\Pp(\T)$ proves \eqref{eq:poc-finite-time}.
\end{proof}

\subsection{Pointwise propagation-of-chaos results near metastable wells: Proofs of Theorems~\ref{thm:poc-metastable}, \ref{thm:poc-global} and \ref{thm:poc-pathwise}}

\begin{proof}[Proof of Theorem~\ref{thm:poc-metastable}]
For (i), fix $0<\varepsilon<D_i$ and $r>0$.
Apply \eqref{eq:verified-M1} with radius $r/2$ and margin
$\varepsilon$, and let $\eta>0$ be its initial radius.
Since $q_t\to\mathcal M_i$, choose a fixed $t_0$ such that
\[
 d_{\Qq}(q_{t_0},\mathcal M_i)<\eta/2,
 \qquad
 \sup_{t\ge t_0}d_{\Qq}(q_t,\mathcal M_i)<r/2.
\]
Define
\[
 E_N:=\{d_{\Qq}(X_{t_0}^N,\mathcal M_i)\le\eta\}.
\]
The triangle inequality and Lemma~\ref{lem:poc-finite-time} give
\[
 \mathbb P(E_N^c)
 \le \mathbb P\bigl(\Delta_N(t_0)>\eta/2\bigr)
 \longrightarrow0.
\]

On $E_N$, the random empirical measure at time $t_0$ satisfies
the initial-state condition in \eqref{eq:verified-M1}.
Since that estimate is uniform over all such initial measures,
the Markov property gives, for every $t_0\le t\le e^{N(D_i-\varepsilon)}$,
\[
 \mathbb P\left(
   d_{\Qq}(X_t^N,\mathcal M_i)>r/2,\ E_N
 \right)\le Ce^{-cN}.
\]
Here the estimate is applied for the elapsed time $t-t_0$, which
is at most $e^{N(D_i-\varepsilon)}$.
Moreover, the deterministic reference is within $r/2$ of
$\mathcal M_i$, so $\Delta_N(t)>r$ requires
$d_{\Qq}(X_t^N,\mathcal M_i)>r/2$. Consequently,
\[
 \sup_{t_0\le t\le e^{N(D_i-\varepsilon)}}
 \mathbb P\bigl(\Delta_N(t)>r\bigr)
 \le \mathbb P(E_N^c)+Ce^{-cN}\longrightarrow0.
\]
Lemma~\ref{lem:poc-finite-time} handles $[0,t_0]$, proving
convergence in probability uniformly over deterministic observation times in $[0,e^{N(D_i-\varepsilon)}]$.
Finally,
\[
 \sup_{0\le t\le e^{N(D_i-\varepsilon)}}\mathbb E\,\Delta_N(t)^2
 \le r^2+B_0^2
       \sup_{0\le t\le e^{N(D_i-\varepsilon)}}\mathbb P(\Delta_N(t)>r).
\]
Taking the limit superior and then letting $r\downarrow0$
proves \eqref{eq:poc-metastable-before}.

For (ii), write
\[
 q_i:=\dist_{\Qq}(\mathcal M_i,\mathcal L_i)
     =4\delta_i>0.
\]
Apply \eqref{eq:verified-M2} with radius $q_i/4$ and margin
$\varepsilon/2$, and let $\eta$ be its initial radius.
Choose a fixed $t_0$ such that
\[
 d_{\Qq}(q_{t_0},\mathcal M_i)<\eta/2,
 \qquad
 \sup_{t\ge t_0}d_{\Qq}(q_t,\mathcal M_i)<q_i/4.
\]
Define $E_N$ as above using this $\eta$ and $t_0$.
Again, $\mathbb P(E_N^c)\to0$.

For all
sufficiently large $N$,
\(
 t\ge e^{N(D_i+\varepsilon)}
 \Rightarrow
 t-t_0\ge e^{N(D_i+\varepsilon/2)}.
\)
Conditioning at $t_0$, the Markov property and the uniform
estimate \eqref{eq:verified-M2} therefore give
\[
 \sup_{t\ge e^{N(D_i+\varepsilon)}}
 \mathbb P\left(
   \dist_{\Qq}(X_t^N,\mathcal L_i)>q_i/4
 \right)
 \le \mathbb P(E_N^c)+Ce^{-cN}.
\]
If $\dist_{\Qq}(X_t^N,\mathcal L_i)\le q_i/4$, the triangle
inequality implies
\[
 \Delta_N(t)\ge
 q_i-\dist_{\Qq}(X_t^N,\mathcal L_i)
     -d_{\Qq}(q_t,\mathcal M_i)
 \ge q_i/2>\delta_i.
\]
Thus $\{\Delta_N(t)\le\delta_i\}$ is contained in the exceptional
event bounded above. This proves
\eqref{eq:poc-metastable-after}, uniformly over all the stated
deterministic times.
\end{proof}

\begin{remark}[The cycle determining the threshold]
\label{rem:poc-maximal-cycle}
If a proper cycle $C$ has bottom $\mathcal M_i$, then
$C\subset C_i^{max}$ and $d(C)\le D_i$. Hence
\eqref{eq:poc-metastable-before} also holds with $d(C)$ in place
of $D_i$ whenever $0<\varepsilon<d(C)$. A matching failure
statement at depth $d(C)$ generally requires $C=C_i^{max}$.
No conclusion is asserted at the critical scale $D_i$.
\end{remark}

\begin{proof}[Proof of Theorem~\ref{thm:poc-global}]
Fix $r>0$. The closed set
$A_r:=\{q:d_{\Qq}(q,\mathcal M_*)\ge r/2\}$ has
$\inf_{A_r}\F>F_{\min}$, unless it is empty. This follows from
compactness, lower semicontinuity, and uniqueness of the global
orbit. Lemma~\ref{lem:uniform-uphill-endpoint} therefore controls
the probability of $A_r$ at every time, started from any $\nu \in \mathcal P_N(\T)$ with $d_{\Qq}([\nu],\mathcal M_*) \le \eta$.
Its proof uses the Gibbs LDP and the path comparison, and does
not require \eqref{eq:finite-critical-orbits}.
Choose a fixed $t_0$ so that $d_{\Qq}(q_{t_0},\mathcal M_*)<\eta/2$ and
$d_{\Qq}(q_t,\mathcal M_*)<r/2$ for every $t\ge t_0$.
Lemma~\ref{lem:poc-finite-time} places $X_{t_0}^N$ in the
required neighborhood with probability tending to one.
Conditioning at $t_0$ and using the triangle inequality now gives
\[
 \sup_{t\ge t_0}\mathbb P(\Delta_N(t)>r)\longrightarrow0.
\]
Lemma~\ref{lem:poc-finite-time} treats $[0,t_0]$. Letting $r$ decrease to
zero and using $\Delta_N\le B_0$ proves \eqref{eq:poc-global}.
\end{proof}

\begin{proof}[Proof of Theorem~\ref{thm:poc-pathwise}]
Fix $r>0$ and use Lemma~\ref{lem:local-block-endpoint} with an
outer radius smaller than $r/2$. Choose a fixed $t_0$ so that
$q_{t_0}$ lies in the interior of its initial neighborhood and
$d_{\Qq}(q_t,\mathcal M_i)<r/2$ for all $t\ge t_0$.
By Lemma~\ref{lem:poc-finite-time}, the approximation on $[0,t_0]$
is good and $X_{t_0}^N$ belongs to that neighborhood with probability
tending to one. Conditional on this entrance event, the probability
of leaving the prescribed outer ball before time $T_N$ is at most
\[
 C(1+T_N)e^{-cN}\longrightarrow0.
\]
The triangle inequality proves convergence in probability of the
supremum in \eqref{eq:poc-pathwise}; boundedness gives convergence
of its expectation. If $T_N<t_0$, Lemma~\ref{lem:poc-finite-time} suffices.
\end{proof}

\begin{remark}[Placement of the supremum]
Theorems~\ref{thm:poc-metastable} and~\ref{thm:poc-global} concern
each deterministic observation time, with uniform error bounds.
Theorem~\ref{thm:poc-pathwise} controls an entire trajectory on the
specified interval, which requires subexponentially growing time horizons, as compared to the exponential ones in the other two results. 
\end{remark}

\subsection{Maximal estimates for running time averages}

We present two estimates involving time-averaged processes that will be used in Section \ref{sec:pathavproof} to prove Theorem~\ref{thm:poc-running-averages} and Corollary~\ref{cor:poc-averaged-measures}.

\begin{lemma}[A maximal time-average bound on a finite horizon]
\label{lem:poc-dyadic}
Let $Z(t)$ be a jointly measurable process with $0\le Z(t)\le K$,
where $K>0$.
For $H\ge1$ and $a>0$,
\begin{equation}\label{eq:poc-dyadic}
 \mathbb P\left(
   \sup_{1\le T\le H}\frac1T\int_0^T Z(t)\,\dd t>a
 \right)
 \le
 \frac{4K}{a}\bigl(1+\lceil\log_2 H\rceil\bigr)
 \sup_{0\le t\le H}\mathbb P(Z(t)>a/2).
\end{equation}
\end{lemma}

\begin{proof}
Choose at most $1+\lceil\log_2H\rceil$ endpoints consisting of
successive powers of two and the final endpoint $H$.
Every $T\in[1,H]$ has an endpoint $u$ with $T\le u\le2T$.
For a nonnegative function, its average up to $T$ is at most
twice its average up to $u$. On the event on the left, there exists $T\in[1,H]$ such that
$T^{-1}\int_0^T Z(t)\,\dd t>a$. Since
$
 Z(t)\le \frac a2+K\mathbf1_{\{Z(t)>a/2\}},
$
we have
\[
 \frac1T\int_0^T\mathbf1_{\{Z(t)>a/2\}}\,\dd t
 >\frac{a}{2K},
\]
and the above comparison gives
\[
 \frac1u\int_0^u\mathbf1_{\{Z(t)>a/2\}}\,\dd t
 >\frac{a}{4K}.
\]
Thus exceeding the original threshold at any time $T$ forces
this indicator-average bound at one of the selected endpoints.

Markov's inequality, Fubini's theorem, and a union bound give
\eqref{eq:poc-dyadic}.
\end{proof}

\begin{lemma}[Maximal averages from a neighborhood of the global orbit]
\label{lem:poc-global-maximal}
Assume the unique-global-orbit hypothesis of
Theorem~\ref{thm:poc-global}, and write
$r_*(q):=d_{\Qq}(q,\mathcal M_*)$.
For every $a>0$ there exist $\eta,c,C>0$ and $N_0$ such that,
for all $N\ge N_0$,
\begin{equation}\label{eq:poc-global-maximal}
 \sup_{\substack{\nu\in\PN(\T)\\
                  d_{\Qq}([\nu],\mathcal M_*)\le\eta}}
 \mathbb P_\nu\left(
   \sup_{T>0}\frac1T\int_0^T r_*(X_t^N)^2\,\dd t>a
 \right)
 \le Ce^{-cN}.
\end{equation}
No finiteness assumption \eqref{eq:finite-critical-orbits} is needed for this estimate.
\end{lemma}

\begin{proof}
Let $\mathbf Q_N$ be the law of the full configuration path
started with Gibbs distribution $\pi_N$. Under $\mathbf Q_N$,
the quotient empirical measure process is stationary with marginal law $\Pi_N$.
For $\rho>0$, let
\[
 A_\rho:=\{q\in\Qq:r_*(q)\ge\rho\}.
\]
Suppose $A_\rho$ is nonempty. It is compact and excludes
$\mathcal M_*$, the unique zero of the rate function
$J(q)=\beta(\F(q)-F_{\min})$. Lower semicontinuity therefore gives
\(
 \gamma_\rho:=\inf_{q\in A_\rho}J(q)>0.
\)
The static LDP upper bound (see \eqref{eq:gibbs-rate-endpoint}) now yields
\[
 \limsup_{N\to\infty}\frac1N\log\Pi_N(A_\rho)
 \le-\gamma_\rho.
\]
Choosing any $0<c_\rho<\gamma_\rho$, we obtain, for all
sufficiently large $N$,
\begin{equation}\label{eq:poc-global-static}
 \Pi_N(r_*\ge\rho)\le e^{-c_\rho N}.
\end{equation}
If $A_\rho$ is empty, this bound is immediate.

Apply the maximal ergodic theorem
\cite[p.~248]{KeanePetersen2006} to the nonnegative stationary sequence
\[
 Z_k^{(h)}:=\frac1h\int_{kh}^{(k+1)h}Z(t)\,\dd t,
 \qquad k\ge0.
\]
This gives
\[
 \mathbb P\left(
 \sup_{n\ge1}\frac1{nh}\int_0^{nh}Z(t)\,\dd t>b
 \right)
 \le\frac{\mathbb E Z(0)}{b}.
\]
Taking $h=2^{-m}\downarrow0$ and using continuity of the
time averages as functions of their positive upper endpoint
gives the same bound with the supremum over all $T>0$:
\begin{equation}\label{eq:poc-stationary-maximal}
 \mathbb P\left(
   \sup_{T>0}\frac1T\int_0^T Z_t\,\dd t>b
 \right)\le\frac{\mathbb E Z_0}{b},\qquad b>0.
\end{equation}

It suffices to consider $0<a<B_0^2$. Set $\rho=\sqrt{a/4}$ and
\[
 E_a^{(1)}:=
 \left\{\sup_{U>0}\frac1U
       \int_1^{1+U}r_*(X_t^N)^2\,\dd t>a/2\right\}.
\]
Since $r_*^2\le a/4+B_0^2\mathbf1_{\{r_*>\rho\}}$,
stationarity, \eqref{eq:poc-stationary-maximal}, and
\eqref{eq:poc-global-static} imply
\begin{equation}\label{eq:poc-stationary-average-tail}
 \mathbf Q_N(E_a^{(1)})
 \le \frac{4B_0^2}{a}e^{-c_\rho N}.
\end{equation}

We next transfer this path-event estimate to starts near
$\mathcal M_*$. For any event $E$ depending only on the quotient
path after time one, the matching and integration argument used
in \eqref{eq:gibbs-uphill-comparison}, now applied to
$\mathbb P_x(E)$, gives
\begin{equation}\label{eq:poc-gibbs-path-comparison}
 \mathbb P_\nu(E)^2
 \le
 e^{4c_0N\eta^2}
 \frac{\mathbf Q_N(E)}{\Pi_N(B_\eta(\mathcal M_*))},
 \qquad d_{\Qq}([\nu],\mathcal M_*)\le\eta.
\end{equation}
Indeed, \eqref{eq:path-event-comparison} applies to such events,
their probabilities are unchanged by permutations and common
rotations, and integrating over Gibbs starting configurations in
$B_\eta(\mathcal M_*)$ gives the displayed denominator.
Here $c_0$ is the constant in \eqref{eq:girsanov-second-moment}.
For every fixed $\eta>0$, the denominator tends to one by the
Gibbs LDP. Choose $\eta$ so small that
$4c_0\eta^2<c_\rho/2$. Equations
\eqref{eq:poc-stationary-average-tail} and
\eqref{eq:poc-gibbs-path-comparison} then give an exponentially
small bound for $\mathbb P_\nu(E_a^{(1)})$ for any initial law $\nu$ satisfying $d_{\Qq}([\nu],\mathcal M_*)\le\eta$.

Finally, shrink $\eta$ so that deterministic trajectories from
its closed ball have $r_*^2<a/4$ throughout $[0,1]$.
The uniform finite-horizon LDP bounds
$\mathbb P_\nu(\sup_{0\le t\le1}r_*(X_t^N)^2>a/2)$
exponentially. Splitting the integral at time one gives
\[
 \sup_{T>0}\frac1T\int_0^T r_*(X_t^N)^2\,\dd t
 \le
 \sup_{0\le t\le1}r_*(X_t^N)^2
 +\sup_{U>0}\frac1U\int_1^{1+U}r_*(X_t^N)^2\,\dd t.
\]
The two preceding bounds prove \eqref{eq:poc-global-maximal}.
\end{proof}

\begin{remark}[The entropy formulation]
\label{rem:poc-entropy}
The same stationary estimates give an entropy version.
Write $\operatorname H(P\mid Q)$ for relative entropy. If an
initial configuration law $P^N$ satisfies
$h_N:=\operatorname H(P^N\mid\pi_N)=o(N)$, its full path law
$\mathbf P_N$ satisfies
$\operatorname H(\mathbf P_N\mid\mathbf Q_N)=h_N$:
the conditional dynamics given the initial configuration are
identical under the two laws. The binary entropy inequality
\[
 \mathbf P_N(E)\log\frac1{\mathbf Q_N(E)}
 \le h_N+\log2
\]
therefore transfers every exponentially small stationary event
above to an event of probability tending to zero. In particular,
both the endpoint distance to $\mathcal M_*$ and the supremum
of its running squared-distance averages converge to zero in
expectation.
Writing
$Z_N:=\int_{\T^N}e^{-\beta\mathcal H_N(x)}
                   \lambda^{\otimes N}(\dd x)$,
direct calculation for a representative minimizer $\mu_*$ gives
\[
 \operatorname H(\mu_*^{\otimes N}\mid\pi_N)
 =\beta NF_{\min}+\log Z_N
 +\frac{\beta}{2}
   \left(W(0)-\iint W(x-y)\,\mu_*(\dd x)\mu_*(\dd y)\right)
 =o(N),
\]
using $N^{-1}\log Z_N\to-\beta F_{\min}$.
Condition \eqref{eq:poc-initial} alone need not give finite
initial entropy. The path comparison
\eqref{eq:poc-gibbs-path-comparison} is why the basin results
below also cover deterministic initial configurations.
\end{remark}

\subsection{Sharp pathwise propagation-of-chaos results for running averages: Proofs of Theorem~\ref{thm:poc-running-averages} and Corollary~\ref{cor:poc-averaged-measures}}\label{sec:pathavproof}

\begin{proof}[Proof of Theorem~\ref{thm:poc-running-averages}]
For (i), use \eqref{eq:verified-M1} with radius $r=\sqrt{a/16}$
and margin $\varepsilon/2$, and denote its initial radius by
$\eta$. Choose a fixed $t_0$ with
\[
 d_{\Qq}(q_{t_0},\mathcal M_i)<\eta/2,\qquad
 \sup_{t\ge t_0}d_{\Qq}(q_t,\mathcal M_i)^2<a/8.
\]
By Lemma~\ref{lem:poc-finite-time}, the entrance event
$E_N=\{d_{\Qq}(X_{t_0}^N,\mathcal M_i)\le\eta\}$ has
probability tending to one. Conditional on $E_N$, the Markov
property and \eqref{eq:verified-M1} give
\[
 \sup_{0\le u\le e^{N(D_i-\varepsilon)}}
 \mathbb P\left(
   d_{\Qq}(X_{t_0+u}^N,\mathcal M_i)^2>a/16
   \,\middle|\,E_N\right)
 \le Ce^{-cN}.
\]
Apply Lemma~\ref{lem:poc-dyadic} under this conditional law,
with $Z(u)=d_{\Qq}(X_{t_0+u}^N,\mathcal M_i)^2$ and
threshold $a/8$, which gives
\[
 \mathbb P\left(
   \sup_{1\le U\le e^{N(D_i-\varepsilon)}}\frac1U
   \int_{t_0}^{t_0+U}d_{\Qq}(X_t^N,\mathcal M_i)^2\,\dd t>a/8
   \,\middle|\,E_N\right)
 \le C_a(1+N)e^{-cN}\longrightarrow0.
\]

Splitting at $t_0$ and using the triangle inequality gives
\begin{align*}
 \sup_{0<T\le e^{N(D_i-\varepsilon)}}A_N(T)
 &\le \sup_{0\le t\le t_0+1}\Delta_N(t)^2
 +2\sup_{1\le U\le e^{N(D_i-\varepsilon)}}\frac1U
     \int_{t_0}^{t_0+U}d_{\Qq}(X_t^N,\mathcal M_i)^2\,\dd t\\
 &\quad+2\sup_{t\ge t_0}d_{\Qq}(q_t,\mathcal M_i)^2.
\end{align*}
Lemma~\ref{lem:poc-finite-time} gives $\sup_{0\le t\le t_0+1}\Delta_N(t)^2\to0$ in probability.
Thus the probability that the left-hand side exceeds $a$
tends to zero. Boundedness by $B_0^2$ proves (i).

For (ii), by
\eqref{eq:poc-metastable-after},
\[
 p_N:=\sup_{t\ge e^{N(D_i+\varepsilon/2)}}
       \mathbb P(\Delta_N(t)\le\delta_i)\longrightarrow0.
\]
For every $T\ge e^{N(D_i+\varepsilon)}$,
\[
 A_N(T)\ge\delta_i^2\left[
   1-\frac{e^{N(D_i+\varepsilon/2)}}{T}
   -\frac1T\int_{e^{N(D_i+\varepsilon/2)}}^T
       \mathbf1_{\{\Delta_N(t)\le\delta_i\}}\,\dd t
 \right].
\]
The expected value of the last term in brackets is at most $p_N$,
uniformly over these $T$, and $e^{N(D_i+\varepsilon/2)}/T\le e^{-N\varepsilon/2}$.
For all large $N$, the event $A_N(T)\le\delta_i^2/2$ therefore
requires this last term to be at least $1/4$.
Markov's inequality bounds its probability by $4p_N$, proving (ii).

For (iii), fix $a>0$ and use Lemma~\ref{lem:poc-global-maximal}
with threshold $a/8$. Choose a fixed $t_0$ so that $q_{t_0}$
lies within half the resulting initial radius and
$\sup_{t\ge t_0}d_{\Qq}(q_t,\mathcal M_*)^2<a/8$.
Lemma~\ref{lem:poc-finite-time} again gives an entrance event $E_N$
with probability tending to one. Conditional on $E_N$,
Lemma~\ref{lem:poc-global-maximal} bounds the probability that
\[
 \sup_{U>0}\frac1U\int_{t_0}^{t_0+U}
          d_{\Qq}(X_t^N,\mathcal M_*)^2\,\dd t>a/8
\]
by $Ce^{-cN}$. The same splitting inequality as in (i), now
with an unrestricted tail supremum, proves convergence in
probability of $\sup_{T>0}A_N(T)$ to zero.
Boundedness proves \eqref{eq:poc-average-global}.
\end{proof}

\begin{proof}[Proof of Corollary~\ref{cor:poc-averaged-measures}]
Choose a measurable optimal coupling $\gamma^N_t$ of
$\widetilde\mu_t^N$ and $\mu_t$. By \cite[Corollary~5.22]{Villani2009},
$\gamma^N_t$ can be chosen jointly measurably in $(t,\omega)$. The measure $T^{-1}\int_0^T\gamma^N_t\,\dd t$ couples their time averages,
and its quadratic transport cost is $A_N(T)$.
This proves \eqref{eq:poc-aligned-averages}, including the
quotient bound since taking an infimum over rotations can only
decrease the distance.
For \eqref{eq:poc-occupation-averages}, the measure
\[
 \frac1T\int_0^T\delta_{(X_t^N,q_t)}\,\dd t
\]
couples $\Xi_T^N$ and $\Xi_T$ and has cost $A_N(T)$.
Apply Theorem~\ref{thm:poc-running-averages}.
\end{proof}

\section{Applications and explicit interaction potentials: proofs}
\label{sec:applicationsproof}

\begin{proof}[Proof of Proposition~\ref{prop:app-finite}]
Every stationary density satisfies the Gibbs equation
\begin{equation}\label{eq:app-gibbs}
 \rho(x)=\frac{\exp\{-\beta W*(\rho\lambda)(x)\}}
 {\int_\T\exp\{-\beta W*(\rho\lambda)(y)\}\,\lambda(\dd y)};
\end{equation}
conversely, every solution is stationary and critical for $\F$.
See \cite[Theorem~2.3 and Proposition~2.4]{CarrilloEtAl2020}.
We solve the higher-mode consistency equations by contraction,
leaving one scalar analytic equation for the first mode.

Put $b_k=\beta|a_k|$ and $e_k(x)=(\cos(kx),\sin(kx))$.
Omit indices with $b_k=0$ and use the real Hilbert space
\[
 \mathcal H=\left\{z=(z_k)_{k\ge2}:
 z_k\in\Rr^2,\quad
 \|z\|_{\mathcal H}^2=\sum_{k\ge2}\frac{|z_k|^2}{b_k}<\infty
 \right\},
 \qquad
 \mathcal K=\{z\in\mathcal H:|z_k|\le b_k\}.
\]
An empty set of retained indices is interpreted as the zero space.
For $|u|\le A$ and $z\in\mathcal H$, define
\begin{equation}\label{eq:app-tail-map}
 \rho_{u,z}(x)=Z(u,z)^{-1}
 \exp\left\{u\cos x+\sum_{k\ge2}z_k\cdot e_k(x)\right\},
 \qquad
 (T_uz)_k=\beta a_k\int_\T e_k\rho_{u,z}\,\dd\lambda.
\end{equation}
Observe that $T_uz$ gives the Fourier coefficients of $-\beta W*(\rho_{u,z}\lambda)$ of order $\ge 2$.
The series is uniformly convergent for $z\in\mathcal H$, by
Cauchy--Schwarz and $\sum b_k<\infty$, and $T_u$ maps
the closed convex set $\mathcal K$ into itself.  

We compute the Fr\'echet derivative of $T_u$. Define the isometry
\[
 J:\mathcal H\longrightarrow\ell^2,
 \qquad (Jz)_k=\frac{z_k}{\sqrt{b_k}},
\]
where indices with $b_k=0$ are omitted throughout. Set
\[
 G(x)=(\sqrt{b_k}\,e_k(x))_{k\ge2},
 \qquad
 S=\operatorname{diag}
 \bigl(\operatorname{sgn}(a_k)\mathrm{Id}_2\bigr).
\]
Since $|e_k(x)|=1$, we have
\[
 \|G(x)\|_{\ell^2}^2=\sum_{k\ge2}b_k=B+R<\infty.
\]

Fix $u$, write $y=Jz$, and abbreviate
$\rho_y=\rho_{u,J^{-1}y}$. In these coordinates,
\[
 \rho_y(x)
 =Z(y)^{-1}
   \exp\{u\cos x+\langle y,G(x)\rangle_{\ell^2}\},
 \qquad
 Z(y)=\int_\T
   \exp\{u\cos x+\langle y,G(x)\rangle_{\ell^2}\}\,\dd\lambda(x).
\]
Write $\mathbb E_y$ for expectation under $\rho_y\lambda$.
The transformed map $\widehat T_u:=JT_uJ^{-1}$ satisfies
\[
 (\widehat T_u(y))_k
 =\frac{\beta a_k}{\sqrt{b_k}}\,\mathbb E_y[e_k]
 =\operatorname{sgn}(a_k)\,\mathbb E_y[G_k],
\]
and hence
\[
 \widehat T_u(y)=S\,\mathbb E_y[G].
\]
Note that, for a direction $v\in\ell^2$,
\[
 D\rho_y[v](x)
 =\rho_y(x)\left(
   \langle v,G(x)\rangle_{\ell^2}
   -\mathbb E_y\langle v,G\rangle_{\ell^2}
 \right).
\]
Thus, writing
$m_y:=\mathbb E_y[G]$, we obtain
\[
 \begin{aligned}
 D\widehat T_u(y)[v]
 &=S\left(
    \mathbb E_y[G\langle v,G\rangle_{\ell^2}]
    -m_y\langle v,m_y\rangle_{\ell^2}
   \right)
 =S\,\mathbb E_y\left[
    (G-m_y)\langle G-m_y,v\rangle_{\ell^2}
   \right].
 \end{aligned}
\]
These differentiations under the integral are justified by the
uniform bound on $\|G(x)\|_{\ell^2}$. 
Hence, defining the covariance operator by
\[
 C_yv:=\operatorname{Cov}_{\rho_y\lambda}(G)v
 :=\mathbb E_y\left[
   (G-m_y)\langle G-m_y,v\rangle_{\ell^2}
 \right],
\]
we obtain
\[
 J(D_zT_u)J^{-1}=S C_y.
\]

The operator $S: \ell^2 \rightarrow \ell^2$ is an isometry, and $C_y: \ell^2 \rightarrow \ell^2$ is positive semidefinite and
self-adjoint. Therefore
\begin{equation}\label{eq:app-covariance-bound}
 \begin{aligned}
 \|D_zT_u\|_{\mathcal H\to\mathcal H} =\|C_y\|_{\mathrm{op}} =\sup_{\|v\|_{\ell^2}\le1}
   \operatorname{Var}_y(\langle v,G\rangle_{\ell^2})
 \le\mathbb E_y\|G\|_{\ell^2}^2
 =B+R.
 \end{aligned}
\end{equation}
Thus, under \eqref{eq:app-simple-condition}, $T_u$ is a
contraction on $\mathcal H$, with contraction constant
$B+R<1$, uniformly over $|u|\le A$.

For contractivity under \eqref{eq:app-refined-condition}, first remove all modes $k\ge3$ and let
\[
 \rho^{(0)}(x)=Z_0^{-1}\exp\{u\cos x+z_2\cdot e_2(x)\},
 \qquad |u|\le A, \,|z_2|\le B.
\]
If $B=0$, take $z_2=0$. Pairing $x$ and $x+\pi$ gives
\begin{equation}\label{eq:app-Z0-bound}
 Z_0=\int_\T\cosh(u\cos x)e^{z_2\cdot e_2(x)}\,\dd\lambda
 \le\cosh(A)I_0(B).
\end{equation}
For every unit vector $v\in\Rr^2$, the function
$f(x)=v\cdot e_2(x)$ satisfies
\begin{equation}\label{eq:app-second-mode-bound}
 \int_\T (1-f^2)e^{u\cos x+z_2\cdot e_2(x)}\,\dd\lambda
 \ge\frac12.
\end{equation}
Indeed, first pair $x$ and $x+\pi$ and use $\cosh(u\cos x)\ge1$.
Then pair $x$ and $x+\pi/2$; the factor $1-f^2$ is unchanged,
while $z_2\cdot e_2$ changes sign. A second use of $\cosh\ge1$
gives a lower bound $\int(1-f^2)\,\dd\lambda=1/2$. Hence
\[
 B\operatorname{Var}_{\rho^{(0)}\lambda}(v\cdot e_2)
 \le B\left(1-\frac1{2Z_0}\right)\le q_0.
\]
For a unit vector $\xi=(\xi_2,\xi_{\ge3})$ in $\ell^2$, the triangle inequality for standard deviations
and $\sum_{k\ge3}b_k=R$ now imply
\[
 \operatorname{Var}_{\rho^{(0)}\lambda}(\langle\xi,G\rangle_{\ell^2})
 \le\left(\sqrt{q_0}\,|\xi_2|
             +\sqrt R\,\|\xi_{\ge3}\|\right)^2
 \le q_0+R.
\]
Noting that $\rho_{u,z}/\rho^{(0)}\le e^{2R}$, and, using
$\operatorname{Var}(f)=\inf_c\mathbb E(f-c)^2$, we obtain under \eqref{eq:app-refined-condition},
\begin{equation}\label{eq:app-refined-contraction}
 \begin{aligned}
 \|D_zT_u\|_{\mathcal H\to\mathcal H}
 &=\sup_{\|\xi\|_{\ell^2}\le1}
   \operatorname{Var}_\rho(\langle\xi,G\rangle_{\ell^2})\\
 &\le e^{2R}\sup_{\|\xi\|_{\ell^2}\le1}
    \operatorname{Var}_{\rho^{(0)}\lambda}(\langle\xi,G\rangle_{\ell^2})
 \le e^{2R}(q_0+R)<1,
 \end{aligned}
\end{equation}
uniformly over $|u|\le A$ and $z\in\mathcal K$.

Under either hypothesis, Banach's fixed-point theorem gives a
unique $z(u)\in\mathcal K$ with $T_uz(u)=z(u)$.
The map $(u,z)\mapsto T_uz$ is real analytic on
$\Rr\times\mathcal H$: the map
$z\mapsto\sum z_k\cdot e_k$ is bounded into $C(\T)$,
the exponential is analytic there, and its normalizing integral
is positive. The contraction bound makes $I-D_zT_u$ invertible
at every fixed point. The analytic implicit function theorem therefore shows that $z(u)$ is analytic
on a neighborhood of $[-A,A]$.

Reflection $x\mapsto-x$ preserves the fixed-point equations.
Uniqueness implies that all sine components of $z(u)$ vanish.
Also $z(0)=0$, since $T_0(0)=0$. Write
$\rho_u=\rho_{u,z(u)}$ and consider
\begin{equation}\label{eq:app-scalar-reduction}
 h(u):=u-A\int_\T\cos x\,\rho_u(x)\,\lambda(\dd x),
 \qquad 0\le u\le A.
\end{equation}
After rotating a stationary density so its first moment is real
and nonnegative, \eqref{eq:app-gibbs} gives precisely a zero of
$h$ (as the higher order Fourier relations are satisfied by the fixed-point property). Conversely, a zero gives a stationary density, since the
sine consistency equation vanishes by reflection symmetry.
The zero $u=0$ gives only $\lambda$.
Finally, $h$ is real analytic on an open neighborhood of $[0,A]$ and
\[
 h(A)=A\left(1-\int_\T\cos x\,\rho_A(x)\,\lambda(\dd x)\right)>0,
\]
because $\rho_A$ is strictly positive. Thus $h$ is not identically
zero and has finitely many zeros on this compact interval.
Each zero determines at most one rotation orbit, proving the claim.
\end{proof}

\begin{proof}[Proof of Proposition~\ref{prop:app-coexistence}]
We first justify local minimality of $\lambda$ for
$\beta<\beta_*$ in the Wasserstein topology. This is also
consistent with the local-attraction result of
\cite[Theorem~2.11]{ChayesPanferov2010}.
Let
\[
 K_+(x):=\sum_{k\ge1}a_k^+\cos(kx),\quad
 a_k^+:=\max(a_k,0),\quad
 Q:=\sum_{k\ge1}a_k^+|m_k(\mu)|^2,\quad
 g:=\frac2aK_+*(\mu-\lambda),
\]
for some probability measure $\mu$.
As $\mu\to\lambda$ weakly, $\|g\|_\infty\to0$, since $K_+$
is continuous. Moreover,
\[
 \int g\,\dd\lambda=0,\qquad
 \int g\,\dd\mu=\frac{2Q}{a},\qquad
 \int g^2\,\dd\lambda\le\frac{2Q}{a},
\]
where the last bound follows from routine calculations using the orthogonality in the Fourier decomposition.
For any $\eta>0$, the entropy variational formula and a Taylor
bound, valid when $\|g\|_\infty$ is sufficiently small, give
\[
 \Ent(\mu\mid\lambda)
 \ge\int g\,\dd\mu-\log\int e^g\,\dd\lambda
 \ge\frac{2Q}{a}-\frac{1+\eta}{2}\int g^2\,\dd\lambda
 \ge\frac{1-\eta}{a}Q.
\]
Consequently,
\begin{equation}\label{eq:app-uniform-local}
 \mathcal E_\beta(\mu)
 \ge\left(\frac{1-\eta}{a}-\frac\beta2\right)Q.
\end{equation}
For $\beta<2/a$, choose $\eta$ so the coefficient is positive.
The inequality is then strict if $Q>0$. If $Q=0$ and
$\mu\ne\lambda$, strict positivity follows from
$\Ent(\mu\mid\lambda)>0$ and
$\sum_k a_k|m_k(\mu)|^2\le0$.
Thus $\lambda$ is a strict Wasserstein local minimum.

We use the following well-known variational facts. For a continuous
transition, $\beta_c=\beta_*$, and by the definition of the transition point $\beta_c$, $\lambda$ is the unique global minimizer for
$0<\beta\le\beta_c$. For a discontinuous
transition, there is coexistence of the uniform (global) minimizer and a
nonuniform global-minimum orbit at $\beta_c$; see
\cite[Propositions~2.12 and~2.13]{ChayesPanferov2010} and
\cite[Propositions~5.6 and~5.7]{CarrilloEtAl2020}.
Moreover, if $\beta_c < \beta_*$, then the transition is necessarily discontinuous, and the entropy estimate above shows that $\lambda$ is a
strict local minimum in the Wasserstein topology for every
$0<\beta<\beta_*$. Consequently, if
$\beta_c<\beta_*$, it is a nonglobal local minimum throughout
$(\beta_c,\beta_*)$, where a nonuniform global minimum also
exists. These variational conclusions do not require finiteness
of the critical set of minimizers.

Finiteness of the critical orbits now permits application of
Lemma~\ref{lem:poc-minimum-stability} to every such minimum.
All the stated propagation-of-chaos and cycle-exit conclusions
then follow from the results cited in the proposition.
\end{proof}

\begin{proof}[Proof of Corollary~\ref{cor:app-explicit-family}]
Condition \eqref{eq:app-discontinuous-fourier} holds, so the cited
phase-transition criterion gives $\beta_c<2/a_1$.
For $0<\beta\le2/a_1$, the quantities in
Proposition~\ref{prop:app-finite} satisfy
\[
 A\le2,\qquad B\le\frac{11}{10},\qquad R\le\frac1{500}.
\]
The left-hand side of \eqref{eq:app-refined-condition} increases
in these three nonnegative variables and is therefore at most
\begin{equation}\label{eq:app-numerical-bound}
 e^{1/250}\left[
 \frac{11}{10}\left(1-\frac1{2\cosh(2)I_0(11/10)}\right)
 +\frac1{500}\right]<0.996<1.
\end{equation}
For example, the series bounds
$\cosh(2)<3.763$, $I_0(11/10)<1.327$, and
$e^{1/250}<1.00401$ imply the displayed inequality.
Apply Propositions~\ref{prop:app-finite}
and~\ref{prop:app-coexistence}.
\end{proof}

\subsection*{Acknowledgements}

The author is partially supported by the NSF-CAREER award DMS-2141621 and the NSF RTG grant DMS-2134107. The author thanks Amarjit Budhiraja, Benjamin Seeger, Pierre Monmarch\'e and Andr\'e Schlichting for several useful discussions related to the project.

\subsection*{AI usage declaration}

The author used ChatGPT (OpenAI) to assist with literature searches, improve the exposition, and explore and refine mathematical arguments. The author takes full responsibility for the final manuscript, including all mathematical statements, proofs, and references.

\bibliographystyle{plain}
\bibliography{references_rev}

\end{document}